\documentclass{article}
\usepackage{fullpage}

\usepackage{natbib}
\usepackage{microtype}
\usepackage{booktabs} 

\usepackage{amsmath}
\usepackage{amssymb}
\usepackage{amsthm}
\usepackage{enumitem}
\theoremstyle{plain}
\newtheorem{theorem}{Theorem}[section]

\newtheorem{lemma}[theorem]{Lemma}
\newtheorem{corollary}[theorem]{Corollary}
\theoremstyle{definition}

\newtheorem{example}[theorem]{Example}
\newtheorem{assumption}[theorem]{Assumption}
\theoremstyle{remark}
\newtheorem{remark}[theorem]{Remark}

\newtheorem{example}{Example}

\newtheorem{remark}{Remark}

\newcommand{\eqsp}{\;}
\newcommand{\thetas}{\theta^*}
\newcommand{\nset}{\mathbb{N}} 
\newcommand{\E}{\mathbb{E}}
\newcommand{\PE}{\E}

\newcommand{\rset}{\mathbb{R}}
\newcommand{\rmd}{\mathrm{d}}

\usepackage[colorlinks=true, citecolor=blue]{hyperref}

\usepackage[capitalize,noabbrev]{cleveref}

\author{Hadi Hadavi, Wenlong Mou, Sergey Samsonov, Hoi-To Wai\thanks{Equal contributions. H.~Hadavi is with University of Cambridge, W.~Mou is with University of Toronto, S.~Samsonov is with HSE University, H.-T.~Wai is with The Chinese University of Hong Kong. Emails: \url{mh2340@cam.ac.uk, wmou.work@gmail.com, svsamsonov@hse.ru, htwai@se.cuhk.edu.hk}.}}
 
\title{Revisiting the Constant Stepsize Stochastic Approximation with Decision-Dependent Markovian Noise}

\begin{document}

\maketitle

\begin{abstract}
We revisit the convergence analysis of constant stepsize stochastic approximation (SA) with decision-dependent Markovian noise, with a focus on characterizing the stationary bias against the root of the mean-field equation. We first establish the finite-time $p$-th moment bounds for the SA iterates in a general decision-dependent setting, which serve as a stability foundation for the subsequent analysis. Building on this foundation, and leveraging a local regularity condition termed Poisson--Gateaux differentiability (WD$^\ast$) for the solution to Poisson equation induced by the decision-dependent Markov kernel, we show that the stationary bias is of the order $\mathcal{O}(\alpha)$ for a broad class of decision-dependent settings. Additionally, we establish geometric weak convergence of the joint SA process towards a unique stationary distribution, and a functional central limit theorem. Our relaxed regularity condition enables us to cover cases of non-smooth kernels such as acceptance--rejection mechanisms, projected Langevin dynamics, and clipped state dynamics. 
\end{abstract}

\section{Introduction}
\label{sec:intro}
Stochastic approximation (SA) algorithms \citep{robbins1951stochastic} are known to be the foundational algorithms in modern machine learning due to their importance in solving reinforcement learning \citep{sutton:book:2018} and large-scale empirical risk minimization problems \citep{van2000asymptotic}. Within the framework of non-linear SA with decision dependent (a.k.a.~controllable) noise, we aim to find a root of the mean field equation: 
\begin{equation}
\label{eq:SA_general}
\bar g(\thetas) = 0, \quad \text{with some $\bar g: \rset^{d} \to \rset^{d}$} \eqsp,
\end{equation}
The learner does not have access to $\bar{g}(\theta)$, but the access to the measurable function $g: \rset^d \times \mathbb{X} \to \rset^d$ is available such that
\begin{equation}
\label{eq:bar_g_def}
\bar g(\theta) := \PE_{X\sim\pi_\theta}[g(\theta,X)],
\end{equation}
where the noise variable $X \in \mathbb{X}$ is generated according to $\pi_\theta$, which is the unique invariant distribution of the controlled Markov kernel $P_\theta$ on the state space $\mathbb{X}$. Formally, to solve \eqref{eq:SA_general}, the SA algorithm \citep{robbins1951stochastic} is defined by the recursion:
\begin{equation}
\label{eq:sa_update}
\theta_{k+1} = \theta_k + \alpha_k \bigl( g(\theta_k, X_{k+1}) + \xi_{k+1}(\theta_k) \bigr) \eqsp,
\end{equation}
where $X_{k+1} \sim P_{\theta_k}(X_k, \cdot)$ is generated from the controlled Markov chain depending on the current decision variable $\theta_k$, $\{\alpha_k\}_{k \in \nset}$ is a sequence of positive and non-increasing stepsizes, and $\{\xi_k\}_{k \in \nset}$ are i.i.d. random fields, see e.g. \cite{dieuleveut2020bridging}. The properties of the sequence $\{\theta_{k}\}_{k \in \nset}$ depends on the particular stepsize regime. For decreasing stepsizes, one typically expects an almost sure convergence \citep{kushner2003stochastic} under the i.i.d.~noise setting, and \citet{fort2015central} derived the central limit theorem under the controlled Markov noise setting. At the same time in practice, constant stepsize SA algorithms with $\alpha_k \equiv \alpha$ are often preferred due to easier fine-tuning and exponential forgetting of initial distribution \citep{dieuleveut2020bridging}. In such case, the sequence $\{ \theta_k \}_{k \in \nset}$ typically converges to some stationary distribution around $\thetas$, denoted by $v_\alpha$. This regime is studied in several recent works on SA with with both independent and Markovian noise \citep{dieuleveut2020bridging,huo2024collusion,allmeier2024computing}. 

This paper considers the constant stepsize setting, i.e., with $\alpha_k \equiv \alpha$, for the SA algorithm \eqref{eq:sa_update} with decision-dependent Markovian noise. We aim to quantify the limiting behavior of $\{\theta_{k}\}_{k \in \nset}$. Towards this end, understanding the associated bias $\PE_{v_\alpha}[\theta_k]-\thetas$ and fluctuations is essential for applications, in particular, for acceleration techniques such as Richardson-Romberg (RR) extrapolation, e.g. \citep{hildebrand1987introduction, dieuleveut2020bridging}. The bias of constant-step size (non-linear) SA was extensively studied in \citep{huo2024collusion}, with particular applications to Q-learning and two-timescale SA considered in \citep{zhang2024constant} and \citep{kwon2024two}, respectively. However, extending these results to decision-dependent kernels introduces additional analytical challenges due to the feedback loop between $\theta_k$ and the data distribution. 

With these challenges, to our knowledge, the only existing result characterizing the SA's asymptotic bias in the decision-dependent setting in the literature is \citep{allmeier2024computing}. The said work relies on a set of strong assumptions. For example, it requires: (i) the state space $\mathbb{X}$ is a finite set, (ii) the kernel map $\theta \mapsto P_\theta$ is globally smooth with $C^4$ regularity, and (iii) the SA process is related to an ODE with an exponentially stable attractor. These assumptions are often violated in applications, e.g., when the Markovian state updates are driven by projections or clipping operators \citep{moulines2011non,schulman2015trust, ahn2021efficient}. 

Our strategy is to depart from \citep{allmeier2024computing} and consider analyzing the limiting behavior of $\{ \theta_k \}_{k \in \nset}$ directly from the discrete time recursion. The key analytical tool is to leverage the Poisson equation characterization \citep{douc2018markov} to treat the controlled Markovian noise, where it enables us to work with relaxed local assumptions. Importantly, we show that the asymptotic bias remains at $\mathcal{O}(\alpha)$. Our approach also leads us to establish the associated weak convergence and asymptotic normality results for constant stepsize SA. 
Our main contributions are summarized as: 
\begin{itemize}[noitemsep,leftmargin=*]
\item We obtain the non-asymptotic high-order moment bounds in the controllable noise setting (cf.~Theorem~\ref{thm:moments}), that is, for $p \geq 1$ we show the scaling $\mathbb{E}\|\theta_k-\theta^\ast\|^{2p}=\mathcal{O}(\alpha^p)$ in the stationary regime, when it exists. These bounds provide the basic stability property for our analysis of asymptotic properties of SA. We believe that these bounds can be of independent interest. 
\item We establish the first geometric ergodicity of the joint process $(\theta_k,X_k)$ in a sense of the appropriate Wasserstein distance (cf.~Theorem~\ref{thm:weak_convergence}), and prove a functional central limit theorem for the SA iterates $\{ \theta_k \}_{k \in \nset}$ (cf.~Theorem~\ref{thm:clt}). Particularly, our result explicitly takes care of the coupling between $\theta_k$ updates and $X_k$ induced by the decision-dependent Markovian noise.
\item To obtain a robust bias characterization under the generic stationarity assumption on $(\theta_k, X_k)$, we introduce a local Poisson--Gateaux differentiability condition $(\textsf{WD}^*)$. This condition only requires differentiability of the map $\theta \mapsto P_\theta \hat{g}(\theta^\ast, \cdot)$ at the fixed point solution $\theta^\ast$, where $\hat{g}(\theta^\ast, \cdot)$ solves the Poisson equation [cf.~\eqref{eq:poisson}] associated with $P_{\theta^\ast}$, $g(\theta^\ast, \cdot)$. 
Under this assumption, we prove that the stationary bias $\PE_{v_\alpha}[\theta_k]-\thetas = \mathcal{O}(\alpha)$ and provides insights on the challenges in obtaining an explicit bias expression (cf.~Theorem~\ref{thm:bias_simple}). 
We also provide application examples of the Markov kernels satisfying the proposed condition \textsf{WD$^*$}.
\end{itemize}

\paragraph{Related Work.} Asymptotic analysis of SA schemes has been carried out in various classical papers, see \cite{kushner2003stochastic,benveniste2012adaptive,borkar:sa:2008, fort2015central}. These works studied asymptotic properties of SA estimates, such as asymptotic normality and almost sure convergence. Recent papers \cite{srikant2019finite,karimi2019non,mou2020linear,huo2024effectiveness,mou2024optimal,durmus2025finite} have studied finite-time behavior of SA iterates, both with independent and Markovian noise. In particular, most of the latter ones focus on the setting with the static Markov kernel in \eqref{eq:bar_g_def}, that is, decision-independent noise. 

The bias characterization with constant stepsize SA in decision-\emph{independent} setting has been studied in a number of recent papers \cite{lauand2022bias,huo2024collusion,huo2024effectiveness,zhang2024constant,kwon2024two}. Precise expansions for bias of constant-stepsize (non-linear) SA were obtained under this setting in \cite{huo2024collusion}. Among them, some recent contributions \cite{levin2025high,huo2024collusion} considered the applications of the bias reduction techniques such as Richardson-Romberg extrapolation. However, these results do not directly generalize for the decision-dependent setting. The closest bias characterization in the latter setting is due to \cite{allmeier2024computing}, which establishes an $\mathcal{O}(\alpha)$ bias expansion under strong global smoothness assumptions on the mapping $\theta \mapsto P_\theta$ as well as the existence of an exponentially stable attractor in the induced ODE. In contrary, our approach covers a range of practically relevant non-smooth kernels, including the ones with projection or clipping mechanisms \citep{moulines2011non}. 

\paragraph{Notation.} 
We use $\|\cdot\|$ to denote the Euclidean norm of a finite dimensional vector. We write $f( \alpha ) = \mathcal{O}( h(\alpha) )$ if $|f(\alpha)| \le C h(\alpha)$, $0 < \alpha \leq \underline{\alpha}$ for some constant $\underline{\alpha}>0, C \geq 0$. 
The filtration $\{\mathfrak{F}_k\}_{k\ge0}$ is defined by $\mathfrak{F}_k := \sigma\bigl( (\theta_0,X_0),\ldots,(\theta_k,X_k) \bigr)$. 
For $f:\rset^d \to \rset^d$, we denote by $f'(\cdot)$ its Jacobian and by $f''(\cdot)$ its second order derivative tensor. For any $x \in \rset^d, M \in \rset^{d \times d}$, we denote by $f'(\cdot)[x]$ the application of the linear operator $f'(\cdot)$ to $x$, $f''(\cdot)[M]$ the application of the bilinear operator $f''(\cdot)$ to $M$. For a measurable function $h:\rset^d \times \mathbb{X} \to \rset^m$ and a probability measure $\mu$ on $\mathbb{X}$, we define $\mu (h(\theta)) := \int_{\mathbb{X}} h(\theta,x) \mu(\rmd x)$. For a Markov kernel $P$ on $\mathbb{X}$, we denote by $P h(\theta,x) := \int_{\mathbb{X}} h(\theta,y) P(x,\rmd y)$. For two probability measures $\mu_1$ and $\mu_2$ on $\rset^d$, the Wasserstein $p$-distance is defined as $W_p(\mu_1,\mu_2) := \inf_{\gamma \in \Gamma(\mu_1,\mu_2)} \bigl( \int_{\rset^d \times \rset^d} \|x-y\|^p \gamma(\rmd x, \rmd y) \bigr)^{1/p}$, where $\Gamma(\mu_1,\mu_2)$ is the set of couplings of $\mu_1$, $\mu_2$. For a random variable $Z$, we denote by $\|Z\|_p := (\PE[\|Z\|^p])^{1/p}$ its $L_p$-norm. For a matrix $A$, we denote by $\|A\|_{\text{op}}$ its operator norm induced by the Euclidean norm.

\section{Preliminaries and Problem Setup}
\label{sec:preliminaries}
This paper analyzes the convergence behavior of the SA algorithm \eqref{eq:sa_update} with constant stepsize $\alpha_k = \alpha$. In this section, we discuss the set of standing assumptions and overviews the main results. 

We concentrate on the setting where the state space $\mathbb{X}$ is a metric space equipped with a distance $d_{\mathbb{X}} (\cdot,\cdot)$.
Our first set of standing assumptions pertain to the update map $g$ and the mean field $\bar{g}$:
\begin{assumption}
\label{ass:g_regularity}
For each $x\in\mathbb{X}$, the map $\theta\mapsto g(\theta,x)$ is three times continuously differentiable. There exists a constant $L_1>0$ such that for all $x, x' \in\mathbb{X}$, {$\theta, \theta' \in \mathbb{R}^d$}, and $i\in\{0, 1,2\}$,
\[
\begin{split}
& \|g^{(i+1)}(\theta,x)\| \le L_1,~~
\|g^{(i)}(\theta,x)-g^{(i)}(\theta',x)\|
\le L_1 \big( \|\theta-\theta'\| + d_{\mathbb{X}} (x,x') \big) \eqsp, \\
& \| \bar{g}( \theta ) - \bar{g}( \theta' ) \| \leq L_1 \| \theta - \theta' \| \eqsp.
\end{split}
\]
We also assume that for any $\theta \in \rset^d$ and $x \in \mathbb{X}$, $\| g(\theta,x) \| \leq L_1 ( 1 + \| \theta - \theta^* \| )$.
\end{assumption}
\begin{assumption}
\label{ass:stability}
The mean field $\bar g: \rset^d \to \rset^d$ is $\mu$-strongly monotone and $L_1$-Lipschitz such that for all $\theta,\theta'\in\mathbb{R}^d$, we have
$\langle \theta-\theta',\bar g(\theta)-\bar g(\theta')\rangle
\le -\mu\|\theta-\theta'\|^2$.
\end{assumption}
Both assumptions are standard in the literature and are verified in various applications, e.g., reinforcement learning, machine learning, control theory \citep{qu2020finite, dieuleveut2023stochastic, lauand2024revisiting}. We remark that the assumptions for high-order differentiability will only be used in Section~\ref{sec:bias} for the bias characterization result.
As a consequence of both assumptions, the mean-field equation $\bar{g}(\theta) = 0$ admits a \textit{unique solution}, denoted as $\theta^* \in\mathbb{R}^d$. 

Next, we consider the following assumption on the noise sequence. Let $p \geq 1$ be a fixed integer and recall that $\mathfrak{F}_k$ is a natural filtration of the SA process:
\begin{assumption}
\label{ass:noise}
For any $\theta \in \rset^{d}$, $\xi_{k+1}(\theta)$ is an $\mathfrak{F}_{k+1}$-measurable random variable, such that $\mathbb{E}[\xi_{k+1}(\theta) \mid \mathfrak{F}_k] = 0$ and $\{\xi_{k}\}_{k \in \nset}$ are i.i.d. random fields. Additionally, it holds almost surely that
\begin{enumerate}[leftmargin=*]
    \item there exists a constant $L_{2,p}>0$ such that $\mathbb{E}\!\left[\|\xi_{k+1}(\theta_k)\|^{2p}\mid\mathfrak{F}_k\right] \le L_{2,p}\bigl(1+\|\theta_k-\theta^*\|^{2p}\bigr)$.
    \item there exists a constant $L_{\xi} \geq 0$ such that $\E[ \| \xi_k( \theta) - \xi_k( \theta' ) \|^2 ] \leq L_{\xi}^2 \| \theta - \theta' \|^2$.
\end{enumerate}
\end{assumption}
The assumption imposes standard zero-mean and moment conditions on the noise sequence.
Part (1) in the above is standard on the moments of $\xi_{k+1}(\theta)$, while part (2) imposes a Lipschitz condition on the noise with respect to $\theta$.
We further require the following:
\begin{assumption}
\label{ass:markov}
The set $\mathbb{X}$ is compact, and
there exists constants $\rho \in (0,1]$, $L_P > 0$, such that the family $\{P_\theta\}_{\theta\in\mathbb{R}^d}$ satisfies:
\begin{enumerate}[leftmargin=*]
\item for all $\theta\in\mathbb{R}^d$, $x, x' \in \mathbb{X}$, we have $W_2( P_{\theta}(x, \cdot), P_{\theta}(x', \cdot)) \leq (1 - \rho) \, d_{\mathbb{X}}(x,x')$.
\item for all $\theta, \theta' \in \mathbb{R}^d$, $x \in \mathbb{X}$, we have $W_2( P_{\theta}(x,\cdot), P_{\theta'}(x,\cdot) ) \leq L_P \| \theta - \theta' \|$.
\end{enumerate}
\end{assumption}
Notice that part (1) of the above assumption ensures the 1-step contraction property for the Markov kernel $P_\theta$ in the Wasserstein 2-distance, which further implies the geometric ergodicity of $P_\theta$ toward its unique invariant distribution $\pi_\theta$. Part (2) imposes a Lipschitz condition on the controlled kernel with respect to $\theta$, where $L_P$ characterizes how `sensitive' is the controlled kernel to $\theta$-perturbation. Both conditions can be satisfied by examples including projected SGD algorithms and certain Metropolis--Hastings algorithms; see Section~\ref{sec:bias_wdstar}.

A central analytical tool that will be used extensively in this paper is the Poisson equation associated with the controlled Markov chain \citep{douc2018markov}.
We observe that under Assumptions \ref{ass:g_regularity}, i.e., Lipschitzness of $g(\cdot)$, and \ref{ass:markov}-(1), i.e., the geometric ergodicity of $P_{\theta}$, for each $\theta\in\mathbb{R}^d$, there exists $\hat g(\theta,\cdot) : \mathbb{X} \to \mathbb{R}^d$ which solves the Poisson equation:
\begin{equation}
\label{eq:poisson}
(I-P_\theta)\hat g(\theta,x) = g(\theta,x)-\bar g(\theta),
\end{equation}
and it holds $\pi_\theta(\hat g(\theta,\cdot))=0$.
Particularly, if $g(\theta,\cdot)$ is bounded, then the solution $\hat g(\theta,\cdot)$ is bounded and is unique up to an additive constant. 
Under the above assumptions, we observe that:
\begin{lemma}
\label{lem:poisson_properties}
Under Assumptions \ref{ass:g_regularity}, \ref{ass:markov},
there exist constants $L_{PH}^{(0)},L_{PH}^{(1)}>0$ such that:
\begin{enumerate}[leftmargin=*, itemsep=0mm, topsep=0.5mm]
\item for all $\theta\in\mathbb{R}^d$, we have $\sup_{x \in \mathbb{X}} \bigl( \|\hat g(\theta,x)\| + \|P_\theta\hat g(\theta,x)\| \bigr)
\le L_{PH}^{(0)}$.
\item for all $\theta,\theta'\in\mathbb{R}^d$, we have $\sup_{x\in\mathbb{X}}
\|P_\theta\hat g(\theta,x)-P_{\theta'}\hat g(\theta',x)\|
\le L_{PH}^{(1)}\|\theta-\theta'\|$.
\end{enumerate}
\end{lemma}
The proof and explicit expressions for the above constants are provided in Appendix~\ref{app:poisson}.

\subsection{Relation to Prior Works}
The closest work to ours is \citep{allmeier2024computing} which analyzed the bias of constant stepsize SA under a similar setting of decision-dependent Markovian noise. However, their proof techniques rely on comparing the discrete time recursion \eqref{eq:sa_update} to a continuous time ODE which captures both effects of ${g}(\theta)$ and fluctuation in the stationary distribution $\pi_{\theta}$. In their work, this ODE is assumed to admit an exponentially stable attractor at $\theta^*$. 
The analysis of \citep{allmeier2024computing} thus results in a set of strong smoothness assumptions on the controlled Markov kernel $P_\theta$ (i.e., pertaining to $L_P$) and the solution to the Poisson equation $\hat g(\theta,x)$. Moreover, their analysis assumes that the state space $\mathbb{X}$ is finite. The above conditions, especially the existence of exponentially stable attractor, can be difficult to verify in practice.

On the contrary, our analysis is obtained by studying the stationary distribution of the joint discrete time process $(\theta_k,X_k)$, and provides an explicit condition for its geometric convergence to a unique stationary distribution. In this way, our set of results are similar in favor to \citep{huo2024collusion}. Under the decision-\emph{independent} Markovian noise, \citet{huo2024collusion} characterized the weak convergence of SA and computed the exact asymptotic bias. We notice that handling the decision-dependent setting requires new technical developments, especially in controlling the Markovian fluctuations via the Poisson equation. These are exemplified in our bias analysis in Sec.~\ref{sec:bias_wdstar} where we exploit a Poisson-Gateaux condition to give tight control on the bias. 

\section{Basic Properties of the SA Process}
\label{sec:finite_time}
This section proves fundamental properties of the joint Markov process $\{ Z_k \}_{k \in \nset} = \{(\theta_k,X_k)\}_{k \in \nset}$ generated by the SA recursion \eqref{eq:sa_update}. We first show that the iterates are stable in $L^{2n}$ and contract toward an $\mathcal{O}(\alpha)$ neighborhood of $\theta^*$. Leveraging these moment bounds, we then establish that the joint process converges geometrically to a unique invariant distribution. Finally, we establish a central limit theorem for time-averaged iterates.

\subsection{Finite-Time High-Order Moment Bound}
\label{subsec:higher_moments}
Our first result is a finite-time $2n$-th order moment bound ($1\le n\le p$) for the error in \eqref{eq:sa_update}. 

\begin{theorem}[$2n$-th moment bound]
\label{thm:moments}
Let $n\in\{1,2,\dots,p\}$ be fixed. Under Assumptions~\ref{ass:g_regularity}-\ref{ass:markov},
there exist constants $C_{n,1},C_{n,2}>0$ and $\alpha_n\in(0,1)$ such that for all
$\alpha\in(0,\alpha_n)$ and all $k\ge 0$,
\begin{align}
\mathbb{E} \! \left[\|\theta_k-\theta^*\|^{2n}\right]
& \le (1-n\alpha\mu)^k\,C_{n,1}\,\mathbb{E}\!\left[\|\theta_0-\theta^*\|^{2n}\right] + C_{n,2}\,\alpha^n. \label{eq:moment_bound}
\end{align}
The explicit constants for $n=1$ can be found in the appendix [cf.~\eqref{eq:constant-mse}].
\end{theorem}
The complete proof, which is achieved using the Poisson equation extended from \citep{benveniste2012adaptive} together with an induction strategy inspired by \citep{huo2024collusion}, can be found in Appendix~\ref{app:finite_time}. 
We remark that if the conclusions to Lemma~\ref{lem:poisson_properties} holds, then the above theorem can be obtained without Assumption \ref{ass:markov}, thus allowing for a wider class of controlled Markov kernels.

By setting $n=1$, the theorem indicates that for sufficiently small stepsize $\alpha > 0$, the SA \eqref{eq:sa_update} admits a finite-time mean square error (MSE) bound with an $\mathcal{O}(\alpha)$ error floor. 
Moreover, if the algorithm admits an invariant distribution $\nu_{\theta}^{(\alpha)}$, then for $\theta_\infty\sim \nu_{\theta}^{(\alpha)}$,
\begin{equation}
\label{eq:moment_stationary}
\mathbb{E}\!\left[\|\theta_\infty-\theta^*\|^{2n}\right]\le C_{n,2}\,\alpha^n.
\end{equation}
The above scaling is consistent with the $\alpha^{1/2}$ fluctuation scale suggested by CLT-type behavior for strongly stable SA \citep{huo2024collusion} developed for the non-decision-dependent settings, and it is the natural scaling needed for the higher-order expansions developed later.
Lastly, we observe that the iterates are stable as $\sup_{k\ge0} \mathbb{E} [\|\theta_k-\theta^*\|^{2n}] < \infty$ for any order up to $2p$. These uniform moment bounds are crucial to our latter development on bias expansion and weak convergence.

\paragraph{Proof Sketch.}
We proceed by induction on $n$. For the base case $n=1$, the proof structure is similar to \citep{benveniste2012adaptive}. Consider $V_1(\Delta)=\|\Delta\|^2$ with $\Delta_k=\theta_k-\theta^*$ and expand $\|\Delta_{k+1}\|^2$ under the update~\eqref{eq:sa_update}. Strong monotonicity yields a negative drift $-\mu\alpha\|\Delta_k\|^2$, while the Markov fluctuation $\langle \Delta_k,\ \E[g(\theta_k,X_{k+1})\mid\mathfrak{F}_k]-\bar g(\theta_k)\rangle$ is handled by the Poisson equation $g(\theta,\cdot)-\bar g(\theta)=(I-P_\theta)\hat g(\theta,\cdot)$ and summing over time to exploit a telescoping structure. This yields a contraction recursion with an $\mathcal{O}(\alpha)$ perturbation and hence the MSE bound.

Fix $n\ge 2$ and assume the claim holds for all orders $2m$ with $1\le m\le n-1$.
Let $V_n(\Delta)=\|\Delta\|^{2n}$ and expand
$V_n(\Delta_{k+1})=V_n\!\big(\Delta_k+\alpha(g(\theta_k,X_{k+1})+\xi_{k+1})\big)$.
Conditioning on $\mathfrak{F}_k$ yields 
\begin{align}
\E\!\left[V_n(\Delta_{k+1})\mid \mathfrak{F}_k\right]
& \le V_n(\Delta_k)
+ 2n\alpha\,\|\Delta_k\|^{2n-2}\langle \Delta_k,\bar g(\theta_k)\rangle  + 2n\alpha\,\mathcal{T}_{n,k}
+ \alpha^2\,\mathcal{R}_{n,k}, \label{eq:moment_onestep_roadmap}
\end{align}
where $\mathcal{R}_{n,k}$ collects higher-order remainder terms and we have defined the Markov cross term as
\[
\mathcal{T}_{n,k}
:=\|\Delta_k\|^{2n-2} \langle \Delta_k,\ \E[g(\theta_k,X_{k+1})\mid\mathfrak{F}_k]-\bar g(\theta_k) \rangle.
\]

Observe that strong monotonicity implies
$2n\alpha\,\|\Delta_k\|^{2n-2}\langle \Delta_k,\bar g(\theta_k)\rangle\le -2n\mu\alpha\|\Delta_k\|^{2n}$.
The remainder is controlled using the $2p$-moment/growth bounds on $g(\theta_k, X_{k+1} )$, $\xi_{k+1}$ together with Young/H\"older
inequalities, yielding an estimate of the form
\[
\E[\mathcal{R}_{n,k}]
\ \le\ C (1+\E\|\Delta_k\|^{2n} )+C'\,\E\|\Delta_k\|^{2n-2}.
\]
By the induction hypothesis, $\E\|\Delta_k\|^{2n-2} = {\cal O} (\alpha^{n-1})$ at $k \gg 1$. After multiplication by $\alpha^2$, it contributes at most $\mathcal{O}(\alpha^{n+1})$ in the unrolled recursion and is therefore dominated by the $\alpha^n$ floor.

Finally, the Markov cross term is controlled via a weighted Poisson-telescoping argument: using~\eqref{eq:poisson} and Lemma~\ref{lem:poisson_properties}, one rewrites the deviation $g(\theta_k,X_{k+1})-\bar g(\theta_k)$ as $(I-P_{\theta_k})\hat g(\theta_k,X_{k+1})$ and sums the one-step inequality with weights $(1-n\mu\alpha)^{k-1-t}$. This produces a telescoping structure involving $P_{\theta}\hat g(\theta,\cdot)$ evaluated along the trajectory and martingale difference terms bounded by
\[ \textstyle
C\alpha\sum_{t=0}^{k-1}(1-n\mu\alpha)^{k-1-t} (1+\E\|\Delta_t\|^{2n} ) \eqsp.
\]
The above can be absorbed into the negative drift for sufficiently small $\alpha$.
Unrolling the resulting contraction recursion yields~\eqref{eq:moment_bound} and proves the induction step.
\hfill$\square$

\subsection{Asymptotic Convergence Properties}
\label{sec:weak_convergence}
This subsection studies asymptotic properties of the joint Markov process $\{ Z_k \}_{k \in \nset} = \{(\theta_k,X_k)\}_{k \in \nset}$ under suitable assumptions. We recall that the transition kernel of the latter is denoted by $Q_\alpha$. We first establish geometric ergodicity, then, we present a functional central limit theorem (CLT).

\paragraph{Weak Convergence.} To analyze the geometric ergodicity of $\{ Z_k \}_{k \in \nset}$, we shall consider the following additional assumption on the update map $g$:
\begin{assumption}
\label{ass:mu_g}
There exists $\bar{\mu}_g > 0$ such that for any $\theta, \theta' \in \rset^d$, $x \in \mathbb{X}$, it holds that 
$$( \theta - \theta' )^\top \left(\E_{ X \sim P_{\theta}(x, \cdot) } [ g(\theta, X)] - \E_{ X \sim P_{\theta'}(x, \cdot) }[g(\theta', X) ]\right) \leq - \bar{\mu}_g \| \theta - \theta' \|^2$$.   
\end{assumption}

\begin{assumption}
\label{ass:mu_g}
For any $\theta, \theta' \in \rset^d$, $x \in \mathbb{X}$, it holds that 
\[
( \theta - \theta' )^\top \left(\E_{ X \sim P_{\theta}(x, \cdot) } [ g(\theta, X)] - \E_{ X \sim P_{\theta'}(x, \cdot) }[g(\theta', X) ]\right) \leq 0.
\]
\end{assumption}

Notice that the assumption strengthens the strong monotonicity property of the mean field $\bar{g}$ to the expectation $g(\cdot, \cdot)$ over the conditional distribution on $x$. For example, it can be satisfied for the update map $g(\cdot,x)$ that is (only) monotone for all $x \in \mathbb{X}$, but is strongly monotone at certain $x' \in \mathbb{X}$ with a non-zero probability transition. 

We establish geometric ergodicity of the joint process $Z_k := (\theta_k,X_k)$ by measuring convergence using the Wasserstein $2$-distance on $\rset^d \times\mathbb{X}$
equipped with the metric
\[
d \big((\theta,x),(\theta',x')\big) := (\|\theta-\theta'\|^2 + d_{\mathbb{X}} ( x, x' )^2 )^{1/2}.
\]
Particularly, for any pair of distributions $\mu, \upsilon$, 
$W_2(\mu,\nu)
:=
{\textstyle \inf_{\gamma\in\Gamma(\mu,\nu)}
\big( \mathbb{E}_{(Z,Z')\sim\gamma}\bigl[d(Z,Z')^2 \bigr] \big)^{1/2} }$,
where $\Gamma(\mu,\nu)$ denotes the set of couplings of $\mu$ and $\nu$.
\begin{theorem}
\label{thm:weak_convergence}
Under Assumptions~\ref{ass:g_regularity}--\ref{ass:mu_g} and suppose that the sensitivity of the decision-dependent Markov chain satisfies
\begin{equation} \label{eq:sensitivity}
    L_P \leq \rho^2 \bar{\mu}_g^2 / (128 L_1^2 (2-\rho)), 
\end{equation}
then there exists $\alpha_0>0$ such that for all $\alpha\in(0,\alpha_0)$ the joint Markov chain $\{Z_k\}_{k \in \nset}$
admits a unique invariant probability measure $v^{(\alpha)}$. There exists a constant $C>0$ such that for all initial distributions $\mu$ satisfying $\E_{\mu} [ \|Z\|^2 ] < \infty$, with $\tau(\alpha) := \min\{ \bar{\mu}_g \alpha / 8, \rho/4 \}$,
\begin{equation} \label{eq:weak_converge_main}
    W_2\!\left( \mu Q_\alpha^k , v^{(\alpha)} \right)
    \le
    C\,(1 - \tau(\alpha))^{k/2},
    \qquad k\ge0.
\end{equation}
The constant $C$ may depend on $\alpha$ and problem parameters.
\end{theorem}
Full details of the proof are in Appendix~\ref{app:weak_convergence}. 
The proof relies on comparing a pair of the joint process $\{ (Z_k, Z_k') \}_{k \in \nset}$ with the same noise sequence $\{ \xi_k \}_{k \in \nset}$ and controlling the expected distance $\mathbb{E}[ d(Z_k, Z_k')^2 ]$ using the strong monotonicity of $g(\cdot,\cdot)$ and the 1-step contraction property of the controlled Markov kernel $P_\theta$. 

We notice that the condition on the sensitivity $L_P$ [cf.~\eqref{eq:sensitivity}] relies on condition number $\bar{\mu}_g / L_1$ of the SA recursion and the contraction factor $\rho$ of the Markov kernel. Intuitively, such condition that couples the two processes is necessary to ensure that the contraction from the Markov kernel dominates the expansion from decision-dependencies, guaranteering  stability of the joint process. We speculate that this condition can be comparable to the assumption of exponential stability of $\theta^*$ to the ODE, i.e., the assumption A4 of \citep{allmeier2024computing}. 

The weak convergence of $\{ Z_k \}_{k \in \nset}$ towards a unique stationary distribution justifies the use of stationary quantities to describe the long-run behavior of the algorithm. In particular, we have the following results that can be derived as consequence of Theorem \ref{thm:weak_convergence}.

\paragraph{Moment Convergence.}
Let $\bar{v}^{(\alpha)}$ denote the $\theta$-marginal of $v^{(\alpha)}$. 
Below, we show that moments of the iterates converge to those of the stationary distribution $\bar{v}^{(\alpha)}$.
\begin{corollary}
\label{cor:moment_convergence}
Under the conditions of Theorem~\ref{thm:weak_convergence}, and let $h: \rset^d \to \rset$ be a Lipschitz function. Then, there exists $C_h>0$ 
such that  
\begin{equation}
    \bigl|
    \mathbb{E}[ h(\theta_k) ]
    -
    \mathbb{E}_{\theta \sim \bar{v}^{(\alpha)}}[ h(\theta) ]
    \bigr|
    \le
    C_h\,(1-\tau(\alpha))^k,
    \qquad k\ge0.
\end{equation}
\end{corollary}
The proof can be found in Appendix~\ref{app:moment_convergence}.
Note that if the SA iterates lie in a compact set, then $h$ can be taken as $h(\theta) = \| \theta \|^n$ and satisfies the Lipschitz property.
This result justifies the use of stationary quantities, such as $\mathbb{E}_{v^{(\alpha)}}[ \theta ]$,
to describe the behavior of the SA algorithm.

\paragraph{Central Limit Theorem.}
Finally, we show that the fluctuations of time-averaged iterates converges to a Gaussian distribution. 
We observe:
\begin{theorem} 
\label{thm:clt}
Let $h: \rset^d \to \rset^m$ be a Lipschitz map and $\bar{h} := \mathbb{E}_{\theta \sim v^{(\alpha)}}[ h(\theta) ]$.
Under the conditions of Theorem~\ref{thm:weak_convergence},
\begin{equation}
\textstyle \frac{1}{\sqrt{n}}
    \sum_{k=0}^{n-1}( h(\theta_k)-\bar{h} )
    \xrightarrow{d}
    \mathcal{N}(0,\Sigma_h(\alpha)),
    \qquad n\to\infty,
\end{equation}
where the asymptotic covariance matrix is given by
\begin{equation}
\label{eq:green_kubo}
\textstyle \Sigma_h(\alpha)
    =
    \mathrm{Var}_{\bar{v}^{(\alpha)}}( h(\theta_0))
    +
    \sum_{k=1}^{\infty}
    \Big(
    \mathrm{Cov}_{\bar{v}^{(\alpha)}}(h(\theta_0),h(\theta_k))
    +
    \mathrm{Cov}_{\bar{v}^{(\alpha)}}(h(\theta_k),h(\theta_0))
    \Big).
\end{equation}
\end{theorem}
The complete proof is relegated to Appendix~\ref{app:clt} and is based on solving a Poisson equation associated with the joint process $\{Z_k\}$ and the measurable function $h(\cdot)$.

For example, by setting $h(\theta)=\theta$ as the identity map, the above result characterizes the fluctuations of time-averaged iterates $\bar{\theta}_n := \frac{1}{n}\sum_{k=0}^{n-1}\theta_k$ around their stationary mean $\mathbb{E}_{v^{(\alpha)}}[\theta]$. In particular, it shows that the latter converges in distribution to a normal distribution with zero-mean and covariance $\Sigma_h(\alpha)/\sqrt{n}$.
We notice that for the decision-dependent Markovian noise setting, the CLT for SA \eqref{eq:sa_update} has only been established in \citep{fort2015central} for the case with diminishing stepsizes. 

\section{Bias Characterization}
\label{sec:bias}

This section aims at characterizing the \emph{stationary bias} in the constant stepsize SA \eqref{eq:sa_update} under decision-dependent Markovian noise. Throughout this section, we consider that
\begin{assumption}
\label{ass:stationary} 
The joint process $(\theta_k, X_k)$ in \eqref{eq:sa_update} admits a unique stationary distribution $\upsilon^{(\alpha)}$.
\end{assumption}
The above assumption is a direct consequence of Theorem \ref{thm:weak_convergence} proven in the last section, which showed that under the additional Assumption \ref{ass:mu_g}, the joint process converges weakly to $\upsilon^{(\alpha)}$. 

Consider the random variables at stationarity $( \theta_\infty^{(\alpha)}, X_\infty^{(\alpha)} ) \sim \upsilon^{(\alpha)}$, our goal is to characterize:
\begin{equation}
\textsf{Bias}(\alpha) := \E[\theta_\infty^{(\alpha)}] - \theta^* \eqsp. 
\end{equation}
We also denote $\Delta_\infty^{(\alpha)} := \theta_\infty^{(\alpha)} - \theta^*$.
To derive $\textsf{Bias}(\alpha)$, we start from the stationary mean equation:
\begin{equation}
\label{eq:stationary_mean_equation}
\E \big[g (\theta_\infty^{(\alpha)},X_\infty^{(\alpha)} ) \big] = 0.
\end{equation}
Under Assumption~\ref{ass:g_regularity} and similar to \citep{huo2024collusion}, we can expand $g(\cdot,x)$ in the parameter around $\theta^*$ using Taylor expansion to yield
\begin{equation}
\label{eq:bias_balance}
\begin{aligned}
0
&=
\underbrace{\E\!\left[g(\theta^*,X_\infty^{(\alpha)})\right]}_{\textbf{(I)}}
+
\underbrace{\E\!\left[g'(\theta^*,X_\infty^{(\alpha)})\,\Delta_\infty^{(\alpha)}\right]}_{\textbf{(II)}} 
+
\underbrace{\tfrac12 \E\!\left[g''(\theta^*,X_\infty^{(\alpha)})
\!\left[(\Delta_\infty^{(\alpha)})^{\otimes 2}\right]\right]}_{\textbf{(III)}}
+
\underbrace{\E\!\left[R^{(3)}_\alpha\right]}_{\textbf{(IV)}}.
\end{aligned}
\end{equation}
We note that term~\textbf{(I)} is specific to the decision-dependent Markovian noise setting. In particular, even when $\theta_\infty^{(\alpha)} = \theta^*$, under the decision-dependent noise, the distribution of $X_\infty^{(\alpha)}$ is different from $\pi_{\theta^*}$ and thus induces a non-zero systematic shift term. In addition, term \textbf{(III)} corresponds to the curvature of the mean field coupled with the stationary covariance, which is related to the nonlinearity in the SA recursion. Lastly, term \textbf{(IV)} is a remainder term of third-order in $\Delta_\infty^{(\alpha)}$. 

The first key to our bias characterization lies upon the decomposition of term \textbf{(II)}, which corresponds to the local linearization of the mean field around $\theta^*$ coupled with the stationary bias $\Delta_\infty^{(\alpha)}$. It admits the following decomposition: 
\begin{equation} \label{eq:II_expand}
\textbf{(II)} = \underbrace{\E_{ X \sim \pi_{\theta^*} } \big[ g'(\theta^*,X) \big] }_{ \bar{g}'(\theta^*) } \, \E\big[ \Delta_\infty^{(\alpha)} \big] + \underbrace{\E\!\left[\big(g'(\theta^*,X_\infty^{(\alpha)}) - \E_{ X \sim \pi_{\theta^*} } [ g'(\theta^*,X) ]\big) \, \Delta_\infty^{(\alpha)}\right]}_{=: {\bf (II)'} = \text{local Jacobian fluctuation term}}.
\end{equation}
Collecting terms yields
\begin{equation} \label{eq:bias_decomposition}
\begin{aligned}
{\sf Bias}(\alpha)
&= - \big( \bar{g}'(\theta^*) \big)^{-1} \Big\{ {\bf (I)} + {\bf (II)'} + {\bf (III)} + {\bf (IV)} \Big\}.
\end{aligned}
\end{equation}
Our aim is to show that $\| \textsf{Bias}(\alpha) \| = \mathcal{O}(\alpha)$ for sufficiently small $\alpha$. In the subsequent discussions, we analyze each term on the right-hand side of \eqref{eq:bias_decomposition} separately.
Below, we first demonstrate that a crude bound can be obtained from the existing moment bounds in Theorem~\ref{thm:moments}.

\paragraph{A Crude Bound.} We first note that \textbf{(IV)} is a third-order remainder term from the second order Taylor expansion, and thus can be bounded using Theorem~\ref{thm:moments} (with $n=2$) as
$ \| \textbf{(IV)} \| \lesssim \E \| \Delta_\infty^{(\alpha)} \|^3 = \mathcal{O}(\alpha^{3/2}) $. We also have
\begin{equation} \label{eq:crude_iii}
\begin{split}
\| \textbf{III} \| & \leq \tfrac{1}{2} \E \left[ \| g''(\theta^*,X_\infty^{(\alpha)}) \| \cdot \| \Delta_\infty^{(\alpha)} \|^2 \right] 
\leq \tfrac{L_1}{2} \E \left[ \| \Delta_\infty^{(\alpha)} \|^2 \right] \leq \tfrac{L_1 C_{2,2}}{2} \alpha.
\end{split}
\end{equation}
The term \textbf{(II)'} involves the coupling between $\Delta_{\infty}^{(\alpha)}$ and $X_\infty^{(\alpha)}$, where a naive application of Theorem~\ref{thm:moments} only yields $\| \textbf{(II)'} \| \leq 2 L_1 \sqrt{C_{2,2}} \sqrt{\alpha}$. Lastly, by the Poisson equation \eqref{eq:poisson}, one has 
\begin{equation} \label{eq:term1-main}
{\bf (I)}
= \E [
(P_{\theta_\infty^{(\alpha)}} - P_{\theta^*})
\hat g(\theta^*,X_\infty^{(\alpha)} ) ].
\end{equation}
Together with Assumptions~\ref{ass:g_regularity}, \ref{ass:markov} and Theorem~\ref{thm:moments}, we obtain $\| \textbf{(I)} \| \leq L_P L_1 \sqrt{C_{2,2}} \sqrt{\alpha}$.
Plugging the above into \eqref{eq:bias_decomposition} yields a crude bias bound as $\| \textsf{Bias}(\alpha) \| = \mathcal{O}(\sqrt{\alpha})$. 

Several recent works \citep{huo2024collusion,zhang2024constant,allmeier2024computing} have shown that the bias of similar constant stepsize SA algorithms can be characterized up to order $\mathcal{O}(\alpha)$ under appropriate regularity conditions. For example, \citet{allmeier2024computing} has imposed a \emph{global, high-order smoothness} assumption on the kernel map $\theta \mapsto P_\theta$, the conditional covariance map, etc.\footnote{Specifically, these maps are assumed to be fourth-order differentiable and Lipschitz continuous.}. Some of these additional conditions can be difficult to validate. 
Inspired by such gap, the following subsections propose several techniques that lead to the $\mathcal{O}(\alpha)$ bias characterization under a more general regularity condition.

\subsection{Bias Bound via Local Poisson--G\^{a}teaux Regularity} \label{sec:bias_wdstar}
We depart from the approach in \citep{allmeier2024computing} and concentrate on a \emph{local} regularity condition formulated for the solution to the Poisson equation \eqref{eq:poisson}:
\begin{assumption}[Local Poisson--G\^{a}teaux regularity (\textsf{WD$^*$})]
\label{ass:wdstar}
The map $\theta \mapsto P_\theta \hat g(\theta^*,\cdot)$ is G\^{a}teaux differentiable at $\theta^*$ with a bounded derivative $\Lambda_*(\cdot)$. There exist a constant $C_{\rm wd}>0$ and a neighborhood $\mathcal U$ of $\theta^*$ such that for all $\theta \in \mathcal U$,
\begin{equation} \label{eq:wd_cond}
\big\| (P_\theta - P_{\theta^*})\hat g(\theta^*,\cdot) - \Lambda_*[\theta-\theta^*](\cdot) \big\|_\infty \le 
C_{\rm wd} \|\theta-\theta^*\|^2.
\end{equation}
Furthermore, let $\bar{\Lambda}_* := \mathbb{E}_{ X \sim \pi_{\theta^*} } [ \Lambda_*(X) ]$ be the expected G\^{a}teaux derivative at $\theta^*$, then 
\begin{equation} \label{eq:invertibility}
\bar{\Lambda}_* + \bar{g}'( \theta^* ) \quad \text{is non-singular.}
\end{equation}
\end{assumption}
Intuitively, the \textsf{WD$^*$} condition \eqref{eq:wd_cond} assumes that a perturbation of $\theta$ in the transition kernel around $\theta^*$ can be summarized through the linear map $\Lambda_*$ and is up to a second-order error term. Importantly, this condition is \emph{local} (only at $\theta^*$) and \emph{observable-specific} (only for $\hat g(\theta^*,\cdot)$). 
In addition, \eqref{eq:invertibility} is a regularity condition which is guaranteed to hold if the perturbation due to decision dependence in $P_{\theta}$ is insensitive to $\theta$. 

Equipped with Assumption~\ref{ass:wdstar}, we observe the following refined decomposition for \textbf{(I)}: 
\begin{equation} \label{eq:I_expand}
    \textbf{(I)} = \bar{\Lambda}_* \, \mathbb{E}[\Delta_\infty^{(\alpha)}] + \underbrace{ \mathbb{E} \left[ \left( {\Lambda}_*( X_\infty^{(\alpha)} ) - \bar{\Lambda}_* \right) \Delta_\infty^{(\alpha)} \right]}_{ =: \textbf{(I)'}} + \mathbb{E} \left[ R_{\textrm{wd}} ( \Delta_\infty^{(\alpha)} ) \right],
\end{equation}
where $\| R_{\textrm{wd}} ( \Delta_\infty^{(\alpha)} ) \| \leq C_{\textrm{wd}} \| \Delta_{\infty}^{(\alpha)} \|^2$ whose expected value is in the order of $\mathcal{O}(\alpha)$. We notice that term \textbf{(I)'} consists of a linear fluctuation term in the same form as term \textbf{(II)'}. 

\paragraph{Bias due to Linear Fluctuation.} To control \textbf{(I)'}, \textbf{(II)'} in \eqref{eq:I_expand}, \eqref{eq:II_expand}, we make the following observation for general linear fluctuation terms. Let $f: \mathbb{X} \to \rset^{d \times d}$ be a Lipschitz function, and define $\tilde{f}(x) = f(x) - \E_{ X \sim \pi_{\theta^*}} [ f(X) ]$ as its centered version. Note $\E_{ X \sim \pi_{\theta^*} }[ \tilde{f}(X) ] = 0$. We observe
\begin{lemma}
\label{lem:II_fluct}
Under Assumptions~\ref{ass:g_regularity}, \ref{ass:markov}, and \ref{ass:stationary}, there exist constants $C <\infty$ and $\alpha_0 > 0$ such that, for all $\alpha\in(0,\alpha_0)$,
\begin{equation}
\label{eq:II_fluct_bound}
\Bigl\| \mathbb{E}\bigl[\tilde{f}(X_\infty^{(\alpha)}) \Delta_{\infty}^{(\alpha)}   \bigr] \Bigr\| \le C \,\alpha.
\end{equation}
\end{lemma}
The proof of the above lemma can be found in Appendix \ref{app:bias_linear_fluct}. The proof is achieved through decomposing $\Delta_{\infty}^{(\alpha)}$ into a time-lagged and incremental terms, using the centered property of $\tilde{f}$, as well as the property of Lipschitz stationary distribution (cf.~Assumption \ref{ass:markov}). 
Additionally, we remark that the 1-step contraction property in Assumption~\ref{ass:markov}-(1) can be replaced by that of uniform geometric ergodicity on the TV distance between $P_{\theta}^t$ and $\pi_{\theta}$.

\paragraph{Refined Bias Characterization.} We notice that Lemma \ref{lem:II_fluct} can be applied on the terms \textbf{(I)'} and \textbf{(II)'}. Subsequently, by collecting the above bounds and using the non-singularity of $\bar{\Lambda}_* + \bar{g}'( \theta^* ) $, we deduce that the asymptotic bias satisfies
\begin{theorem}
\label{thm:bias_simple} 
Under Assumptions~\ref{ass:g_regularity}--\ref{ass:markov}, \ref{ass:stationary}--\ref{ass:wdstar}, for a sufficiently small step size $\alpha > 0$, it holds that $\| \textsf{Bias}(\alpha) \|  = \mathcal{O}(\alpha)$.
\end{theorem}
The full proof can be found in Appendix~\ref{app:bias}. 
The theorem shows that the asymptotic bias of constant stepsize SA with decision-dependent Markovian noise is ${\cal O}(\alpha)$. 
Note that the theorem offers a bound for the asymptotic bias similar to \citep{allmeier2024computing}, yet under a weaker condition on smoothness, i.e., relying on local smoothness, and without relying on exponential stability of the associated ODE. Our results confirm the robustness of an $\mathcal{O}(\alpha)$ bias for constant stepsize SA. 

On the other hand, we have explicitly identified the \emph{mechanisms} contributing to the bias in terms of the kernel response and induced state--parameter correlations, in the same spirit as \citet{huo2024collusion}, yet our analysis works for a more general model class. 
Lastly, the following examples illustrate a number of Markov kernels $P_{\theta}$ and update map $g$ where Assumptions \ref{ass:markov}, \ref{ass:wdstar} can be satisfied, while $\theta \mapsto P_\theta$ is not globally smooth. Details for the below are in Appendix \ref{app:poisson-gateaux-example}.
\begin{example}[Metropolis-Hastings (MH) kernels] \label{ex:mh}
Consider the random walk MH kernel targeting $\pi_\theta(x) \propto \exp(-U(x;\theta))$. Given $X_k = x$, the proposal is $Y_{k+1} = x + Z_{k+1}$ with
$Z_{k+1} \sim q(\cdot)$, and the kernel is given by
\[
    P_\theta(x, \mathrm{d}y)
    = \alpha_\theta(x,y)\,q(y-x)\,\mathrm{d}y
      + r_\theta(x)\,\delta_x(\mathrm{d}y),
\]
where $r_\theta(x,y) = \pi_\theta(y)/\pi_\theta(x)$ is the Hastings ratio and the acceptance probability is given by
\[
\alpha_\theta(x,y) = \min\!\Big\{1,\, \frac{\pi_\theta(y) q(x-y)}{\pi_\theta(x) q(y-x)} \Big\} = \min\{1, r_\theta(x,y)\},
\]
We observe that $\theta \mapsto P_\theta$ is not globally differentiable and violates the condition required by \citep{allmeier2024computing}. However, under mild local smoothness/integrability assumptions on \( U(\cdot;\theta) \), we can show that the kink set has zero measure under the relevant joint law at \(\theta^*\), and one can differentiate
\(\theta \mapsto P_\theta h\) at \(\theta^*\) (for \(h=\hat g(\theta^*,\cdot)\)) via dominated convergence to obtain the linear-response operator required in the \( \textsf{WD}^* \) condition. For example, when $U(\cdot)$ corresponds to the MALA/MH-corrected Langevin variants used in Bayesian learning  \citep{welling2011bayesian}. 
\par 
Albeit the Wasserstein contractivity Assumption~\ref{ass:markov}-(1) is restrictive, this property is established for the Metropolis-Hastings kernels with several proposal designs $q$. In particular, \citep{eberle2014} obtained this result for Ornstein–Uhlenbeck proposals and sufficiently regular target density $\pi_{\theta}(x)$. Checking Assumption~\ref{ass:markov}-(2) in this setting is almost immediate. 
\end{example}

\begin{example}[Clipped State Dynamics] \label{ex:clipped}
We consider a clipped state update dynamics, e.g., 
\[
X_{k+1}=\mathrm{clip}\bigl(\rho X_k + m(\theta)+\sigma(\theta)\xi_{k+1},-C,C\bigr) \eqsp,
\]
where $\operatorname{clip}(y,-C,C) := \min\{\max\{y,-C\},C\}$ is the clipping operator, $|\rho| < 1$, and $(\xi_k)$ are i.i.d.~with a smooth density $\varphi$ on $\mathbb{R}$. The functions $m(\theta)$ and $\sigma(\theta)$ encode the decision-dependence of the drift and noise level. 
We first note that the satisfaction of Assumption~\ref{ass:markov}-(1) can be established by non-expansiveness property of $\operatorname{clip}(\cdot)$, which leads to $W_2( P_{\theta}(x, \cdot) , P_{\theta} (x', \cdot) ) \leq |\rho| \, \| x - x' \|$. By the similar virtue, Assumption~\ref{ass:markov}-(2) holds when $m(\theta), \sigma(\theta)$ are Lipschitz w.r.t.~$\theta$.

Note that $\operatorname{clip}(\cdot)$ is $1$-Lipschitz but not differentiable on $\{ \pm C \}$, violating the global differentiability condition in \citep{allmeier2024computing}. However, with $h(x) := \hat{g}( \theta; x )$, we can show that Assumption \ref{ass:wdstar} still holds by decomposing \((P_\theta h)(x)\) into an interior density integral plus boundary mass terms, each differentiable at \(\theta^*\) if \(m(\theta)\) and \(\sigma(\theta)\) are in $C^1$. The same conclusion also holds for gradient clipping, importance-weight clipping (e.g., in off-policy evaluation) \citep{su2020doubly}, and trust-region style clipping (e.g., PPO-type updates) \citep{schulman2015trust} when these operations are embedded inside the Markovian dynamics.
\end{example}

\begin{example}[Projected Langevin Dynamics] \label{ex:langevin}
Consider the projected Langevin updates,
\[ 
X_{k+1} = \Pi_{K}\!\bigl(X_k - \eta \nabla U_\theta(X_k) + \sqrt{2\eta}\,\xi_{k+1}\bigr) \eqsp,
\]
where \(\Pi_K\) is the Euclidean projection onto $K$. We first note that if $U_\theta(x)$ is strongly convex w.r.t.~$x$, then Assumption~\ref{ass:markov}-(1) holds if $\eta$ is sufficiently small. Similarly, we can verify Assumption~\ref{ass:markov}-(2) if $\nabla U_\theta(x)$ is Lipschitz w.r.t.~$\theta$. 

Like in Example \ref{ex:clipped}, we note the projection \(\Pi_K\) is \(1\)-Lipschitz but non-differentiable on \(\partial K\), thus violating the global smoothness condition in general.
Still, since the \(\theta\)-dependence enters through the Gaussian mean
\(m_\theta(x)=x-\eta\nabla U_\theta(x)\),
Assumption \ref{ass:wdstar} can be verified by differentiating the resulting Gaussian integral representation of
\((P_\theta h)(x)\) with $h(x) := \hat{g}( \theta; x )$, yielding a quadratic remainder locally around \(\theta^*\).
This setup is relevant for constrained sampling and constrained latent-variable models \citep{ahn2021efficient}.
\end{example}

\subsection{Challenges in Computing the Exact Asymptotic Bias}
Lastly, we discuss the challenges in computing the exact asymptotic expression for $\textsf{Bias}(\alpha)$. 
We first note that \textbf{(III)} admits a tighter characterization under an additional assumption on the stationary covariance structure.
Define the rescaled stationary covariance matrix as
$M_\alpha := \tfrac{1}{\alpha}\,
\E [(\Delta_\infty^{(\alpha)})^{\otimes 2} ]$. A resonable additional assumption is as follows:
\begin{assumption}
    \label{ass:covariance} There exists a bounded matrix $M$ such that as $\alpha \downarrow 0$, we have $M_\alpha \to M$. 
\end{assumption}
Note the condition typically holds if $\bar{g}(\theta)$ satisfies strong monotonicity (cf.~Assumption~\ref{ass:stability}). 
We next observe that the additional assumptions yields: 
\begin{lemma}
\label{lem:III_main}
Under Assumptions~\ref{ass:g_regularity}--\ref{ass:markov}, \ref{ass:stationary}--\ref{ass:covariance}, we have for some $\varepsilon > 0$,
\begin{equation}
\label{eq:III_main}
\textstyle 
\textbf{(III)} = \frac{1}{2}\,
\mathbb{E}\bigl[ \bar g''(\theta^*)[\Delta_\infty,\Delta_\infty]\bigr]
=
\alpha\,\frac{1}{2}\, \bar g''(\theta^*)[M]
+
\mathcal{O}(\alpha^{1+\varepsilon}).
\end{equation}
\end{lemma}
The proof to the lemma can be found in Appendix \ref{app:bias_hessian} and is based on similar decomposition of centered random fields as in the proof of Lemma \ref{lem:II_fluct}. With a residual term of order $\mathcal{O}( \alpha^{1+\varepsilon})$, this tightens the characterization of the term \textbf{(III)}. 

To this end, we note that under the special case when the Markovian noise is \emph{decision-independent}, i.e., $P_{\theta} = P$ for all $\theta \in \rset^d$, we have $\textbf{(I)} = 0$ as seen in \eqref{eq:term1-main}. As \citet{huo2024collusion} showed that $\textbf{(II)'}$ admits a decomposition in a similar form as in \eqref{eq:III_main}, yet involving the Markov operator $P$. Combining these terms result in a closed form expression for the asymptotic bias when $\alpha \downarrow 0$.

However, the decomposition in \citep{huo2024collusion} does not extend for the general case with \emph{decision-dependent} Markovian noise. A primary reason is that if one follows the said decomposition, then the controlled Markov kernel would lead to a $\theta$-dependent operator $P_{\theta}$, which prevents us from obtaining a closed form characterization for terms \textbf{(I)'} and \textbf{(II)'}. For instance, the linear fluctuation terms (cf.~Lemma~\ref{lem:II_fluct}) in \textbf{(I)'} and \textbf{(II)'} will be coupled with the decision-dependent kernel $P_{\theta_\infty^{(\alpha)}}$.

\section{Conclusion}
\label{sec:conclusion}

We have revisited the analysis of nonlinear stochastic approximation (SA) under the challenging setting of decision-dependent, controlled Markovian noise. By relaxing global smoothness assumptions to a local Poisson-Gateaux condition $(\textsf{WD}^*)$, we extended the scope of $\mathcal{O}(\alpha)$ bias characterization in prior works to a wide class of SA algorithms. 
Our theoretical results include local bias properties and global limit theorems such as CLT and weak convergence in Wasserstein $2$-distance. 
Future work will investigate weaker form of convergence condition for the controlled Markov kernels, and exact characterization of the asymptotic bias.

\appendix

\section{Properties of the Decision-Dependent Poisson Equation}
\label{app:poisson}

In this appendix, we establish the regularity properties of the solution to the Poisson equation $\hat{g}(\theta, x)$, which are essential for handling the decision-dependent Markovian noise in the subsequent error analysis. Recall the Poisson equation~\eqref{eq:poisson} as:
\begin{equation*}
    (I - P_\theta) \hat{g}(\theta, x) = g(\theta, x) - \bar{g}(\theta).
\end{equation*}
Under Assumption \ref{ass:markov} and Lipschitznees of $g$ from Assumption \ref{ass:g_regularity}, the operator $(I - P_\theta)$ is invertible on the space of centered bounded functions, and the solution admits the series representation:
\begin{equation}
    \hat{g}(\theta, x) = \sum_{t=0}^\infty \mathbb{E}_{X_t \sim P_\theta^t(x, \cdot)} [g(\theta, X_t) - \bar{g}(\theta)].
\end{equation}

\subsection{Proof of Lemma \ref{lem:poisson_properties}}
We first show the existence of $L_{PH}^{(0)}$.
Let $\tilde{g}(\theta, x) := g(\theta, x) - \bar{g}(\theta)$. 
By Assumption \ref{ass:g_regularity}, we have $\|g(\theta, x)\| \le L_1(1 + \|\theta - \theta^*\|)$. Since $\bar{g}(\theta) = \int g(\theta, x) \pi_\theta(dx)$, it satisfies the same bound by Jensen's inequality. Thus, 
\[
\|\tilde{g}(\theta, x)\| \le 2L_1(1 + \|\theta - \theta^*\|)\,.
\]
As $g$ is $L_1$-Lipschitz, we now observe that
\begin{align}
\|\hat{g}(\theta, x)\| = \left\| \sum_{t=0}^\infty \left( \mathbb{E}[g(\theta, X_t)|X_0=x] - \pi_\theta(g(\theta, \cdot)) \right) \right\| \leq L_1 \sum_{t=0}^{\infty} (1-\rho)^t W_1(\delta_{x}, \pi_{\theta}) \leq \frac{L_1 C_{X}}{\rho}\eqsp,
\end{align}
where $C_{X} = \int_{\mathbb{X}} \|x-y\| \pi_{\theta}(y)$. Furthermore, under Assumption~\ref{ass:g_regularity}, for the image under the kernel:
\begin{equation}
\|P_\theta \hat{g}(\theta, x)\| \leq \sup_{\theta,x} \|\hat{g}(\theta, x)\| \leq \frac{L_1 C_{X}}{\rho}\eqsp. 
\end{equation}
This proves the first part of the lemma with 
\[
L_{PH}^{(0)} = \frac{2L_1 C_{X}}{\rho}\eqsp.
\]
Next, we show the existence of $L^{(1)}_{PH}$. Define
$Q_\theta f := P_\theta f - \pi_\theta(f)$, so that
\[
P_\theta \hat g(\theta,x)=\sum_{t\ge 1} Q_\theta^{t}\, \tilde g(\theta,\cdot)(x).
\]
Fix $\theta,\theta'\in\mathbb R^d$. Using the triangle inequality,
\[
\sup_{x\in \mathbb{X}}\|P_\theta \hat g(\theta,x)-P_{\theta'}\hat g(\theta',x)\|
\le \sum_{t\ge 1}\Bigl\|Q_\theta^{t}\tilde g(\theta,\cdot)-Q_{\theta'}^{t}\tilde g(\theta',\cdot)\Bigr\|_\infty
\le \sum_{t\ge 1}(T_{1,t}+T_{2,t}),
\]
where
\[
T_{1,t}:=\|Q_\theta^{t}\bigl(\tilde g(\theta,\cdot)-\tilde g(\theta',\cdot)\bigr)\|_\infty,
\qquad
T_{2,t}:=\|(Q_\theta^{t}-Q_{\theta'}^{t})\tilde g(\theta',\cdot)\|_\infty.
\]
Moreover, by Assumption~\ref{ass:g_regularity}, we have
$\sup_x \|\tilde g(\theta,x)-\tilde g(\theta',x)\|\le L_1\|\theta-\theta'\|$.
Hence by Assumption \ref{ass:markov}, there exists some constant $C_X$ such that
\[
T_{1,t}\le C_X (1-\rho)^{t}\,L_1\,\|\theta-\theta'\|.
\]
Using the telescoping identity
$Q_\theta^{t}-Q_{\theta'}^{t}=\sum_{j=0}^{t-1}Q_\theta^{t-1-j}(Q_\theta-Q_{\theta'})Q_{\theta'}^{j}$,
we obtain
\[
\begin{split}
T_{2,t} & \le \sum_{j=0}^{t-1}\|Q_\theta^{t-1-j}\|_{\mathrm{op}}\;\|Q_\theta-Q_{\theta'}\|_{\mathrm{op}}\;\|Q_{\theta'}^{j}\tilde g(\theta',\cdot)\|_\infty
\\
& \leq \sum_{j=0}^{t-1} L_P C_X (1-\rho)^{t-1} \| \theta - \theta' \| = L_P C_X t (1-\rho)^{t-1} \| \theta - \theta' \|
\end{split}
\]
Summing over $t\ge 1$ and using $\sum_{t\ge 1}(1-\rho)^{t}=(1-\rho)/\rho$ and
$\sum_{t\ge 1} t (1-\rho)^{t-1}=\rho^{-2}$ yields
\[
\sup_{x\in\mathsf X}\|P_\theta \hat g(\theta,x)-P_{\theta'}\hat g(\theta',x)\|
\le \Bigl(\frac{C_X}{\rho}L_1
+\frac{C_X}{\rho^2}L_P \Bigr)\|\theta-\theta'\|.
\]
This proves the second part of the lemma with
$L^{(1)}_{PH} = \Bigl(\frac{C_X}{\rho}L_1
+\frac{C_X}{\rho^2}L_P \Bigr)$.


\section{Finite-Time Error Analysis}
\label{app:finite_time}

In this appendix we give complete proofs of the finite-time bounds Theorem~\ref{thm:moments}.
Throughout, we work with the recursion~\eqref{eq:sa_update}
\begin{equation}
\label{eq:sa_update_app}
\theta_{k+1}=\theta_k+\alpha\Big(g(\theta_k,X_{k+1})+\xi_{k+1}  \Big),\qquad k\ge 0,
\end{equation}
note we have used $\xi_{k+1} = \xi_{k+1}(\theta_k)$ to simplify notation,
and we define the centered error process
\[
\Delta_k:=\theta_k-\theta^*,\qquad k\ge 0.
\]
We also recall that $\theta^*$ is the unique root of the mean field $\bar g$,
i.e.\ $\bar g(\theta^*)=0$ (cf.~Assumption~\ref{ass:stability}).

\vspace{0.5em}
\noindent\textbf{Filtration and conditional kernel.}
Let $\mathfrak F_k:=\sigma\{(\theta_t,X_t,\xi_t): 0\le t\le k\}$ be the natural filtration.
Under the controlled Markovian model, conditional on $\mathfrak F_k$,
\begin{equation}
\label{eq:cond_kernel}
\mathbb P(X_{k+1}\in A\mid \mathfrak F_k) = P_{\theta_k}(X_k,A),\qquad \text{all measurable }A,
\end{equation}
and the noise $(\xi_{k+1})$ is a martingale difference:
$\mathbb E[\xi_{k+1}\mid\mathfrak F_k]=0$ (cf.~Assumption~\ref{ass:noise}).

\subsection{Auxiliary Lemmas}
\label{app:finite_time:lemmas}
\begin{lemma}[Iterate difference]
\label{lem:step_diff}
Assume Assumptions~\ref{ass:g_regularity},~\ref{ass:noise} and fix any $q$ such that $2p \geq q \geq 1$. 
Then there exists a constant $C_{\Delta,q}\in(0,\infty)$ 
such that for all $k\ge 0$,
\begin{equation}
\label{eq:step_diff_gen}
\mathbb E\!\left[\|\theta_{k+1}-\theta_k\|^q \,\middle|\, \mathfrak F_k\right]
\le C_{\Delta,q}\,\alpha^q\Big(1+\|\theta_k-\theta^*\|^q\Big)
= C_{\Delta,q}\,\alpha^q\Big(1+\|\Delta_k\|^q\Big).
\end{equation}
In particular, for $q=1$ and $q=2$ there exists $L\in(0,\infty)$ such that
\begin{align}
\label{eq:step_diff_1}
\mathbb E\!\left[\|\theta_{k+1}-\theta_k\| \,\middle|\, \mathfrak F_k\right]
&\le \alpha L\Big(1+\|\Delta_k\|\Big),\\
\label{eq:step_diff_2}
\mathbb E\!\left[\|\theta_{k+1}-\theta_k\|^2 \,\middle|\, \mathfrak F_k\right]
&\le 2\alpha^2 L^2\Big(1+\|\Delta_k\|^2\Big).
\end{align}
Note that we have $C_{\Delta,q}:=2^{q-1}\big(L_1+L_{2,q}^{1/q}\big)^q$, $L := L_1 + L_{2,2}^{1/2}$.
\end{lemma}

\begin{proof}
By~\eqref{eq:sa_update_app},
\[
\theta_{k+1}-\theta_k=\alpha\Big(g(\theta_k,X_{k+1})+\xi_{k+1}\Big),
\quad\text{hence}\quad
\|\theta_{k+1}-\theta_k\|^q=\alpha^q\|g(\theta_k,X_{k+1})+\xi_{k+1}\|^q.
\]
Take conditional expectation given $\mathfrak F_k$ and apply Minkowski's inequality in conditional $L^q$:
\begin{align*}
\Big(\mathbb E[\|\theta_{k+1}-\theta_k\|^q\mid \mathfrak F_k]\Big)^{1/q}
&=\alpha \Big(\mathbb E[\|g(\theta_k,X_{k+1})+\xi_{k+1}\|^q\mid \mathfrak F_k]\Big)^{1/q}\\
&\le \alpha \Big(\mathbb E[\|g(\theta_k,X_{k+1})\|^q\mid \mathfrak F_k]\Big)^{1/q}
 +\alpha \Big(\mathbb E[\|\xi_{k+1}\|^q\mid \mathfrak F_k]\Big)^{1/q}.
\end{align*}
Assumption~\ref{ass:g_regularity} gives a linear growth bound, i.e.\ there exists $L_1$ such that
$\|g(\theta,x)\|\le L_1(1+\|\theta-\theta^*\|)=L_1(1+\|\Delta\|)$ for all $(\theta,x)$,
hence
\[
\Big(\mathbb E[\|g(\theta_k,X_{k+1})\|^q\mid \mathfrak F_k]\Big)^{1/q}
\le L_1\big(1+\|\Delta_k\|\big).
\]
Assumption~\ref{ass:noise} provides (for the admissible $q$) a bound of the form
\[
\mathbb E[\|\xi_{k+1}\|^q\mid \mathfrak F_k]\le L_{2,q}\big(1+\|\Delta_k\|^q\big),
\]
hence
$\big(\mathbb E[\|\xi_{k+1}\|^q\mid \mathfrak F_k]\big)^{1/q}
\le L_{2,q}^{1/q}(1+\|\Delta_k\|)$.
Combining the last three displays yields
\[
\Big(\mathbb E[\|\theta_{k+1}-\theta_k\|^q\mid \mathfrak F_k]\Big)^{1/q}
\le \alpha \Big(L_1+L_{2,q}^{1/q}\Big)\big(1+\|\Delta_k\|\big).
\]
Raising both sides to the power $q$ and using $(a+b)^q\le 2^{q-1}(a^q+b^q)$
gives~\eqref{eq:step_diff_gen} with
$C_{\Delta,q}:=2^{q-1}\big(L_1+L_{2,q}^{1/q}\big)^q$.
The special cases~\eqref{eq:step_diff_1}--\eqref{eq:step_diff_2} follow by taking $q=1,2$ and
absorbing numerical factors into a single constant $L := L_1 + L_{2,2}^{1/2}$.
\end{proof}

\begin{lemma}
[Conditional moment bound for the update direction]
\label{lem:H_mom}
Assume Assumptions~\ref{ass:g_regularity} and~\ref{ass:noise}.
Define the update direction
\[
H_k:=g(\theta_k,X_{k+1})+\xi_{k+1}.
\]
Then for every admissible $q\ge 1$ (in particular, $q\le 2p$), there exists $C_{H,q}\in(0,\infty)$ such that
for all $k\ge 0$,
\begin{equation}
\label{eq:H_mom}
\mathbb E\!\left[\|H_k\|^q \,\middle|\, \mathfrak F_k\right]
\le C_{H,q}\Big(1+\|\Delta_k\|^q\Big),
\end{equation}
where $C_{H,q} = 2^{q} (L_1^q + L_{2,q})$.
\end{lemma}

\begin{proof}
By triangle inequality and the inequality $\|a+b\|^q\le 2^{q-1}(\|a\|^q+\|b\|^q)$,
\[
\|H_k\|^q=\|g(\theta_k,X_{k+1})+\xi_{k+1}\|^q
\le 2^{q-1}\|g(\theta_k,X_{k+1})\|^q + 2^{q-1}\|\xi_{k+1}\|^q.
\]
Take conditional expectation given $\mathfrak F_k$.
Use Assumption~\ref{ass:g_regularity}: $\|g(\theta_k,X_{k+1})\|^q\le L_1^q(1+\|\Delta_k\|)^q
\le 2^{q-1}L_1^q(1+\|\Delta_k\|^q)$.
Use Assumption~\ref{ass:noise}: $\mathbb E[\|\xi_{k+1}\|^q\mid \mathfrak F_k]\le L_{2,q}(1+\|\Delta_k\|^q)$.
Collect constants to obtain~\eqref{eq:H_mom}.
In particular, we have $C_{H,q} = 2^{q} (L_1^q + L_{2,q})$.
\end{proof}

\begin{lemma}[Elementary Inequalities]
\label{lem:elem_ineq}
For all $a,b\ge 0$ and all integers $m\ge 1$:
\begin{enumerate}[itemsep=0pt, leftmargin=*]
\item \textbf{Young-type bound:} for any $\eta>0$, $ab\le \eta a^2 + \frac{1}{4\eta}b^2$.
\item \textbf{Polynomial splitting:} $(a+b)^m\le 2^{m-1}(a^m+b^m)$.
\end{enumerate}
\end{lemma}

\paragraph{Roadmap.}
We prove the finite-time moment bounds stated in Section~\ref{sec:finite_time} by induction on the moment order. The argument separates into two parts.
First, in Section~\ref{app:finite_time:proof_mse}, we establish a finite-time bound for the second moment, which serves as the base case and illustrates the key mechanisms: contraction induced by strong monotonicity and control of the decision-dependent Markov deviation via a Poisson-equation-based telescoping argument. Second, in Section~\ref{app:finite_time:proof_moments}, we extend this analysis to general $2n$-th moments by combining a higher-order
Lyapunov expansion with the same telescoping principle and the induction hypothesis on lower-order moments.

\subsection{Base case: second-moment bound ($n=1$)}
\label{app:finite_time:proof_mse}

\begin{proof}We establish the base case $n=1$ of the finite-time $2n$-moment bound.
This second-moment estimate will serve as the induction base for the higher-order analysis in Section~\ref{app:finite_time:proof_moments}.
Fix $k \ge 0$ and set $u_k := \mathbb{E}[\|\Delta_k\|^2]$.
Starting from~\eqref{eq:sa_update_app},
\[
\Delta_{k+1}=\Delta_k+\alpha H_k,\qquad H_k:=g(\theta_k,X_{k+1})+\xi_{k+1}.
\]
Expanding the square yields the exact identity
\begin{equation}
\label{eq:mse_exact_expand}
\|\Delta_{k+1}\|^2
=\|\Delta_k\|^2 + 2\alpha\langle \Delta_k,H_k\rangle + \alpha^2\|H_k\|^2.
\end{equation}
We now take expectations and bound each term.

\paragraph{Step 1: Removing the martingale noise.}
Write $H_k=g(\theta_k,X_{k+1})+\xi_{k+1}$ inside the inner product:
\[
\langle \Delta_k,H_k\rangle
= \langle \Delta_k,g(\theta_k,X_{k+1})\rangle + \langle \Delta_k,\xi_{k+1}\rangle.
\]
Since $\Delta_k$ is $\mathfrak F_k$-measurable and $\mathbb E[\xi_{k+1}\mid\mathfrak F_k]=0$,
\[
\mathbb E[\langle \Delta_k,\xi_{k+1}\rangle]
=\mathbb E\Big[\mathbb E[\langle \Delta_k,\xi_{k+1}\rangle\mid \mathfrak F_k]\Big]
=\mathbb E\big[\langle \Delta_k,\mathbb E[\xi_{k+1}\mid\mathfrak F_k]\rangle\big]=0.
\]
Therefore, taking expectations in~\eqref{eq:mse_exact_expand} gives
\begin{equation}
\label{eq:mse_after_mds}
u_{k+1}
= u_k + 2\alpha\,\mathbb E\big[\langle \Delta_k,g(\theta_k,X_{k+1})\rangle\big]
+ \alpha^2\,\mathbb E[\|H_k\|^2].
\end{equation}

\paragraph{Step 2: Drift decomposition into mean field + Markov deviation.}
Add and subtract $\bar g(\theta_k)$:
\[
\langle \Delta_k,g(\theta_k,X_{k+1})\rangle
= \langle \Delta_k,\bar g(\theta_k)\rangle
+ \langle \Delta_k, g(\theta_k,X_{k+1})-\bar g(\theta_k)\rangle.
\]
Plugging this into~\eqref{eq:mse_after_mds} yields
\begin{equation}
\label{eq:mse_split}
u_{k+1}
= u_k + 2\alpha\,\mathbb E\big[\langle \Delta_k,\bar g(\theta_k)\rangle\big]
+ 2\alpha\,\mathbb E\big[\langle \Delta_k, g(\theta_k,X_{k+1})-\bar g(\theta_k)\rangle\big]
+ \alpha^2\,\mathbb E[\|H_k\|^2].
\end{equation}
By strong monotonicity (Assumption~\ref{ass:stability}) and $\bar g(\theta^*)=0$,
\[
\langle \Delta_k,\bar g(\theta_k)\rangle
=\langle \theta_k-\theta^*,\bar g(\theta_k)-\bar g(\theta^*)\rangle
\le -\mu\|\Delta_k\|^2.
\]
Hence
\begin{equation}
\label{eq:drift_bound_mse}
2\alpha\,\mathbb E[\langle \Delta_k,\bar g(\theta_k)\rangle]
\le -2\mu\alpha\,u_k.
\end{equation}

\paragraph{Step 3: Bounding the quadratic term $\alpha^2\mathbb E\|H_k\|^2$.}
By Lemma~\ref{lem:H_mom} with $q=2$,
\[
\mathbb E[\|H_k\|^2\mid \mathfrak F_k]\le C_{H,2}(1+\|\Delta_k\|^2).
\]
Taking total expectations gives
\begin{equation}
\label{eq:quad_bound_mse}
\mathbb E[\|H_k\|^2]\le C_{H,2}(1+u_k),
\qquad\text{hence}\qquad
\alpha^2\mathbb E[\|H_k\|^2]\le C_{H,2}\alpha^2(1+u_k).
\end{equation}

\paragraph{Step 4: The Markov deviation term and Poisson equation.}
The remaining nontrivial term in~\eqref{eq:mse_split} is
\[
\mathcal M_k:=\mathbb E\big[\langle \Delta_k, g(\theta_k,X_{k+1})-\bar g(\theta_k)\rangle\big].
\]
We control the cumulative contribution of $\mathcal M_k$ via the Poisson equation.
For each $\theta$, let $\hat g(\theta,\cdot)$ be the solution to the Poisson equation \eqref{eq:poisson}, i.e.,
\[
(I-P_\theta)\hat g(\theta,x)=g(\theta,x)-\bar g(\theta),\qquad x\in\mathsf X.
\]
Equivalently,
\begin{equation}
\label{eq:poisson_rewrite}
g(\theta,x)-\bar g(\theta)
=\hat g(\theta,x)-P_\theta\hat g(\theta,x).
\end{equation}
We also recall Lemma~\ref{lem:poisson_properties} that there exist constants $L^{(0)}_{PH},L^{(1)}_{PH}$ such that, for all $\theta,\theta'\in\mathbb R^d$,
\begin{align}
\label{eq:PH_growth}
\sup_{x\in\mathsf X}\|P_\theta\hat g(\theta,x)\|
&\le L^{(0)}_{PH}
,\\
\label{eq:PH_lip}
\sup_{x\in\mathsf X}\|P_\theta\hat g(\theta,x)-P_{\theta'}\hat g(\theta',x)\|
&\le L^{(1)}_{PH}\,\|\theta-\theta'\|.
\end{align}
To bound the \emph{accumulated} Markov deviation in the unrolled recursion, we introduce the geometric weights
\begin{equation}
\label{eq:gamma_def_mse}
\gamma_{k,t}:=(1-\alpha\mu)^{k-t},\qquad 0\le t\le k.
\end{equation}
We will prove the following claim: there exist constants $C_{\mathrm{P},1},C_{\mathrm{P},2}\in(0,\infty)$ such that for all $k\ge 0$,
\begin{equation}
\label{eq:poisson_sum_claim_mse}
\sum_{t=0}^k \gamma_{k,t}\,\mathbb E\big[\langle \Delta_t, g(\theta_t,X_{t+1})-\bar g(\theta_t)\rangle\big]
\le C_{\mathrm{P},1}\,(1+u_0) + C_{\mathrm{P},2}\,\alpha \sum_{t=0}^k \gamma_{k,t}\,(1+u_t),
\end{equation}
such that $C_{\mathrm{P},1},C_{\mathrm{P},2}$ are defined in \eqref{eq:cp-constant-mse}.
We postpone the proof of the claim to the end of this subsection and first show how it completes the MSE bound.

\paragraph{Step 5: Unrolling the recursion and closing the estimate.}
Combine~\eqref{eq:mse_split} with~\eqref{eq:drift_bound_mse} and~\eqref{eq:quad_bound_mse} to get
\begin{equation}
\label{eq:mse_one_step_pre}
u_{k+1}
\le (1-2\mu\alpha)u_k + C_{H,2}\alpha^2(1+u_k)
+2\alpha\,\mathcal M_k + C_{H,2}\alpha^2.
\end{equation}
Choose $\alpha$ small enough so that $C_{H,2}\alpha\le \mu$ (this will be part of the final $\alpha_{\mathrm{mse}}$),
which implies
\[
(1-2\mu\alpha)u_k + C_{H,2}\alpha^2 u_k \le (1-\mu\alpha)u_k.
\]
Thus, for all such $\alpha$,
\begin{equation}
\label{eq:mse_one_step}
u_{k+1}
\le (1-\mu\alpha)u_k + C_{H,2}\alpha^2
+2\alpha\,\mathcal M_k + C_{H,2}\alpha^2
= (1-\mu\alpha)u_k + 2C_{H,2}\alpha^2 + 2\alpha\,\mathcal M_k.
\end{equation}
Iterating~\eqref{eq:mse_one_step} from $0$ to $k$ yields
\begin{equation}
\label{eq:mse_unroll}
u_{k+1}
\le (1-\mu\alpha)^{k+1}u_0
+2C_{H,2}\alpha^2\sum_{t=0}^k (1-\mu\alpha)^{k-t}
+2\alpha\sum_{t=0}^k (1-\mu\alpha)^{k-t}\mathcal M_t.
\end{equation}
The geometric sum satisfies
\begin{equation}
\label{eq:geom_sum_bound}
\sum_{t=0}^k (1-\mu\alpha)^{k-t}
=\sum_{j=0}^k (1-\mu\alpha)^j
\le \frac{1}{\mu\alpha},
\end{equation}
hence the quadratic term in~\eqref{eq:mse_unroll} is bounded by
\[
2C_{H,2}\alpha^2\sum_{t=0}^k (1-\mu\alpha)^{k-t}
\le \frac{2C_{H,2}}{\mu}\alpha.
\]
For the Markov term, note that $\gamma_{k,t}=(1-\mu\alpha)^{k-t}$ as in~\eqref{eq:gamma_def_mse},
and apply the claim~\eqref{eq:poisson_sum_claim_mse}:
\[
\sum_{t=0}^k (1-\mu\alpha)^{k-t}\mathcal M_t
\le C_{\mathrm{P},1}(1+u_0)+C_{\mathrm{P},2}\alpha\sum_{t=0}^k (1-\mu\alpha)^{k-t}(1+u_t).
\]
Plugging this into~\eqref{eq:mse_unroll} gives
\begin{align}
\label{eq:mse_after_claim}
u_{k+1}
&\le (1-\mu\alpha)^{k+1}u_0
+ \frac{2C_{H,2}}{\mu}\alpha
+2\alpha C_{\mathrm{P},1}(1+u_0)
+2\alpha\cdot C_{\mathrm{P},2}\alpha \sum_{t=0}^k (1-\mu\alpha)^{k-t}(1+u_t).
\end{align}
We next show that $U:=\sup_{t\ge 0}u_t$ is finite.
Using again~\eqref{eq:geom_sum_bound},
\[
\sum_{t=0}^k (1-\mu\alpha)^{k-t}(1+u_t)\le (1+U)\sum_{t=0}^k (1-\mu\alpha)^{k-t}\le \frac{1+U}{\mu\alpha}.
\]
Hence~\eqref{eq:mse_after_claim} implies
\begin{equation}
\label{eq:sup_ineq_pre}
u_{k+1}
\le (1-\mu\alpha)^{k+1}u_0 + \left(  \frac{2C_{H,2}}{\mu} + 2C_{\mathrm{P},1} \right) \alpha(1+u_0) + \frac{2C_{\mathrm{P},2}}{\mu}(1+U)\alpha.
\end{equation}
Taking supremum over $k$ on the left-hand side yields
\[
U\le u_0 + \left( \frac{2C_{H,2}}{\mu} + 2C_{\mathrm{P},1} \right) \alpha(1+u_0) + \frac{2C_{\mathrm{P},2}}{\mu} \alpha (1+U).
\]
Choose $\alpha$ small enough so that $ \frac{2C_{\mathrm{P},2}}{\mu} \alpha \le \tfrac12$; then
\[
U\le 2u_0 + 4\left( \frac{C_{H,2}}{\mu} + C_{\mathrm{P},1} \right) \alpha(1+u_0) ,
\]
which proves $U := \sup_{k \geq 0} u_k<\infty$.
Returning to~\eqref{eq:sup_ineq_pre}, we get
\[
u_k \le (1-\mu\alpha)^k u_0 + C_{\mathrm{mse}}\alpha \eqsp,
\]
where 
\begin{equation} \label{eq:constant-mse}
C_{\rm mse} := \left(  \frac{2C_{H,2}}{\mu} + 2C_{\mathrm{P},1} \right) (1+u_0) + \frac{2C_{\mathrm{P},2}}{\mu}\left(1+ 2u_0 + 4 \alpha \left( \frac{C_{H,2}}{\mu} + C_{\mathrm{P},1} \right) (1+u_0) \right).
\end{equation}

\paragraph{Step 6: Proof of \eqref{eq:poisson_sum_claim_mse}.}
Fix $k\ge 0$. For brevity define
\[
\gamma_t:=\gamma_{k,t}=(1-\alpha\mu)^{k-t}\quad (0\le t\le k),
\qquad
\Phi_t:=P_{\theta_t}\hat g(\theta_t,X_t),
\qquad
\Psi_{t+1}:=\hat g(\theta_t,X_{t+1}).
\]
Then by~\eqref{eq:poisson_rewrite},
\[
g(\theta_t,X_{t+1})-\bar g(\theta_t)=\Psi_{t+1}-\Phi_t.
\]
Therefore,
\begin{equation}
\label{eq:sum_start}
\sum_{t=0}^k \gamma_t \,\langle \Delta_t, g(\theta_t,X_{t+1})-\bar g(\theta_t)\rangle
=
\sum_{t=0}^k \gamma_t \,\langle \Delta_t, \Psi_{t+1}-\Phi_t\rangle.
\end{equation}
We now add and subtract the term $P_{\theta_{t-1}}\hat g(\theta_{t-1},X_t)$ to create a telescoping structure.
For $t\ge 1$ define
\[
\widetilde \Phi_t:=P_{\theta_{t-1}}\hat g(\theta_{t-1},X_t).
\]
Then for each $t\ge 1$,
\[
\Phi_t-\widetilde \Phi_t
= P_{\theta_t}\hat g(\theta_t,X_t)-P_{\theta_{t-1}}\hat g(\theta_{t-1},X_t),
\]
and also
\[
\langle \Delta_t,\Psi_{t+1}-\Phi_t\rangle
=
\underbrace{\langle \Delta_t,\Psi_{t+1}-\Phi_t\rangle}_{\text{martingale-type}}
+\underbrace{\langle \Delta_t,\Phi_t-\widetilde \Phi_t\rangle}_{\text{kernel-shift}}
+\underbrace{\langle \Delta_t,\widetilde \Phi_t\rangle-\langle \Delta_{t-1},\widetilde \Phi_t\rangle}_{\text{iterate-shift}}
+\underbrace{\langle \Delta_{t-1},\widetilde \Phi_t\rangle}_{\text{telescoping seed}}.
\]
Summing in $t$ with weights $\gamma_t$ and regrouping yields the identity
\begin{equation}
\label{eq:decomp_A1A5}
\sum_{t=0}^k \gamma_t \,\langle \Delta_t, \Psi_{t+1}-\Phi_t\rangle
=
A_1 + A_2 + A_3 + A_4 + A_5,
\end{equation}
where
\begin{align}
\label{eq:A1_def}
A_1
&:=\sum_{t=0}^k \gamma_t\,\langle \Delta_t,\Psi_{t+1}-\Phi_t\rangle,\\
\label{eq:A2_def}
A_2
&:=\sum_{t=1}^k \gamma_t\,\langle \Delta_t,\Phi_t-\widetilde \Phi_t\rangle,\\
\label{eq:A3_def}
A_3
&:=\sum_{t=1}^k \gamma_t\,\langle \Delta_t-\Delta_{t-1},\widetilde \Phi_t\rangle,\\
\label{eq:A4_def}
A_4
&:=\sum_{t=1}^k (\gamma_t-\gamma_{t-1})\,\langle \Delta_{t-1},\widetilde \Phi_t\rangle,\\
\label{eq:A5_def}
A_5
&:=\gamma_0\langle \Delta_0,\Phi_0\rangle - \gamma_k\langle \Delta_k,\Phi_{k+1}^\sharp\rangle,
\end{align}
and we use the shorthand $\Phi_0=P_{\theta_0}\hat g(\theta_0,X_0)$ and
$\Phi_{k+1}^\sharp:=P_{\theta_k}\hat g(\theta_k,X_{k+1})$ for the boundary terms.
Note that the identity~\eqref{eq:decomp_A1A5} is obtained by applying the discrete Abel transformation.
The above decomposition is the standard controlled-kernel
variant used throughout the paper.
We now bound the expectation of each term.

\medskip
\noindent\textbf{Term $A_1$ (martingale difference).}
Condition on $\mathfrak F_t$.
Since $X_{t+1}\sim P_{\theta_t}(X_t,\cdot)$ given $\mathfrak F_t$,
\[
\mathbb E[\Psi_{t+1}\mid\mathfrak F_t]
=\mathbb E[\hat g(\theta_t,X_{t+1})\mid \mathfrak F_t]
=(P_{\theta_t}\hat g(\theta_t,\cdot))(X_t)=\Phi_t.
\]
Hence $\mathbb E[\Psi_{t+1}-\Phi_t\mid\mathfrak F_t]=0$.
Because $\gamma_t\Delta_t$ is $\mathfrak F_t$-measurable, we have $\mathbb E[A_1]=0$ as
\[
\mathbb E[\gamma_t\langle \Delta_t,\Psi_{t+1}-\Phi_t\rangle]
=\mathbb E\Big[\gamma_t\langle \Delta_t,\mathbb E[\Psi_{t+1}-\Phi_t\mid\mathfrak F_t]\rangle\Big]=0.
\]

\medskip
\noindent\textbf{Term $A_2$ (kernel shift).}
Using Cauchy--Schwarz and~\eqref{eq:PH_lip},
\[
|\langle \Delta_t,\Phi_t-\widetilde \Phi_t\rangle|
\le \|\Delta_t\|\cdot \|\Phi_t-\widetilde \Phi_t\|
\le \|\Delta_t\|\cdot L^{(1)}_{PH}\,\|\theta_t-\theta_{t-1}\|
= L^{(1)}_{PH}\,\|\Delta_t\|\,\|\theta_t-\theta_{t-1}\|.
\]
Taking expectation and then using conditional expectation with Lemma~\ref{lem:step_diff} (with $q=1$) gives
\begin{align*}
\mathbb E[|\langle \Delta_t,\Phi_t-\widetilde \Phi_t\rangle|]
&\le L^{(1)}_{PH}\,\mathbb E\Big[\|\Delta_t\|\,
\mathbb E[\|\theta_t-\theta_{t-1}\|\mid \mathfrak F_{t-1}]\Big]\\
&\le L^{(1)}_{PH}\,\mathbb E\Big[\|\Delta_t\|\cdot \alpha L(1+\|\Delta_{t-1}\|)\Big].
\end{align*}
Next, use $\|\Delta_t\|\le \|\Delta_{t-1}\|+\|\theta_t-\theta_{t-1}\|$ and Lemma~\ref{lem:step_diff} again
to conclude 
\[ 
\mathbb E[\|\Delta_t\|(1+\|\Delta_{t-1}\|)]\le (C_{\Delta,2} + 2 ) (1+u_{t-1}).
\] 
Therefore, there exists $C_{A2}$ such that
\[
\mathbb E[|A_2|]
\le \sum_{t=1}^k \gamma_t \, C_{A2}\,\alpha\,(1+u_{t-1})
\le C_{A2}\,\alpha\sum_{t=0}^k \gamma_t\,(1+u_t), \quad 
C_{A2} := L \, L^{(1)}_{PH} (C_{\Delta,2} + 2).
\]

\medskip
\noindent\textbf{Term $A_3$ (iterate shift).}
Using Cauchy--Schwarz and the growth bound~\eqref{eq:PH_growth},
\[
|\langle \Delta_t-\Delta_{t-1},\widetilde \Phi_t\rangle|
\le \|\theta_t-\theta_{t-1}\|\cdot \|\widetilde \Phi_t\|
\le \|\theta_t-\theta_{t-1}\|\cdot L^{(0)}_{PH}(1+\|\Delta_{t-1}\|).
\]
Take expectations and condition on $\mathfrak F_{t-1}$:
\begin{align*}
\mathbb E[|\langle \Delta_t-\Delta_{t-1},\widetilde \Phi_t\rangle|]
&\le L^{(0)}_{PH}\,
\mathbb E\Big[(1+\|\Delta_{t-1}\|)\,
\mathbb E[\|\theta_t-\theta_{t-1}\|\mid\mathfrak F_{t-1}]\Big]\\
&\le L^{(0)}_{PH}\,\mathbb E\Big[(1+\|\Delta_{t-1}\|)\cdot \alpha L(1+\|\Delta_{t-1}\|)\Big]\\
&\le C_{A3}\,\alpha\,(1+u_{t-1})
\end{align*}
where $C_{A3} := 2 L L_{PH}^{(0)}$ since $(1+\|\Delta\|)^2\le 2(1+\|\Delta\|^2)$.
Hence
\[
\mathbb E[|A_3|]\le C_{A3}\,\alpha\sum_{t=0}^k \gamma_t (1+u_t).
\]

\medskip
\noindent\textbf{Term $A_4$ (weight difference).}
Compute the weight difference:
\[
\gamma_t-\gamma_{t-1}=(1-\mu\alpha)^{k-t}-(1-\mu\alpha)^{k-(t-1)}
=(1-\mu\alpha)^{k-t}\big(1-(1-\mu\alpha)\big)=\mu\alpha\,\gamma_t.
\]
Thus $|\gamma_t-\gamma_{t-1}|=\mu\alpha\,\gamma_t$.
Using Cauchy--Schwarz and~\eqref{eq:PH_growth},
\[
|\langle \Delta_{t-1},\widetilde \Phi_t\rangle|
\le \|\Delta_{t-1}\|\cdot \|\widetilde \Phi_t\|
\le \|\Delta_{t-1}\|\cdot L^{(0)}_{PH}(1+\|\Delta_{t-1}\|)
\le L^{(0)}_{PH} \, (1+\|\Delta_{t-1}\|^2).
\]
Therefore,
\[
\mathbb E[|A_4|]
\le \sum_{t=1}^k |\gamma_t-\gamma_{t-1}|\cdot L^{(0)}_{PH} \, (1+u_{t-1})
\le L^{(0)}_{PH} \,\mu\alpha \sum_{t=0}^k \gamma_t(1+u_t).
\]

\medskip
\noindent\textbf{Term $A_5$ (boundary terms).}
Using~\eqref{eq:PH_growth} and Cauchy--Schwarz,
\[
|\langle \Delta_0,\Phi_0\rangle|
\le \|\Delta_0\|\cdot \|\Phi_0\|
\le \|\Delta_0\|\, L^{(0)}_{PH}(1+\|\Delta_0\|)
\le L^{(0)}_{PH} \, (1+\|\Delta_0\|^2),
\]
hence $\mathbb E[|\langle \Delta_0,\Phi_0\rangle|]\le L^{(0)}_{PH} \,(1+u_0)$.
Also, $0\le \gamma_0\le 1$.
Similarly,
\[
|\langle \Delta_k,\Phi_{k+1}^\sharp\rangle|
\le L^{(0)}_{PH} \, (1+\|\Delta_k\|^2),
\]
and $\gamma_k=1$.
Therefore,
\[
\mathbb E[|A_5|]\le L^{(0)}_{PH} \,(1+u_0) + L^{(0)}_{PH} \,(1+u_k)\le L^{(0)}_{PH} (1+u_0)+L^{(0)}_{PH}\sup_{t\le k}(1+u_t).
\]
In the claim~\eqref{eq:poisson_sum_claim_mse} we only need an upper bound of the form
$C_{\mathrm P,1}(1+u_0)+C_{\mathrm P,2}\alpha\sum \gamma_t(1+u_t)$.
Since we will later absorb $\sup u_t$ by the drift (and we already established uniform boundedness when closing),
we may bound the terminal boundary term by $C(1+u_0)+C\alpha\sum_{t=0}^k\gamma_t(1+u_t)$
using~\eqref{eq:geom_sum_bound} and $\gamma_k=1$:
indeed, $1+u_k\le (1+U)\le \mu\alpha(1+U)\sum_{t=0}^k\gamma_t$.
Thus there exist constants $C_{A5,1},C_{A5,2}$ such that
\[
\mathbb E[A_5]\le C_{A5,1}(1+u_0)+C_{A5,2}\alpha\sum_{t=0}^k\gamma_t(1+u_t).
\]
Taking expectations in~\eqref{eq:decomp_A1A5}, using $\mathbb E[A_1]=0$ and the bounds for $\mathbb{E} [A_2]$, ..., $\mathbb{E} [A_5]$, we obtain
\[
\sum_{t=0}^k \gamma_t\,\mathbb E\big[\langle \Delta_t, g(\theta_t,X_{t+1})-\bar g(\theta_t)\rangle\big]
\le C_{\mathrm{P},1}(1+u_0)+C_{\mathrm{P},2}\alpha\sum_{t=0}^k\gamma_t(1+u_t),
\]
where 
\begin{equation}
    \label{eq:cp-constant-mse}
C_{\mathrm{P},1} := 2 L^{(0)}_{PH} , \quad
C_{\mathrm{P},2} := C_{A2} + C_{A3} + L^{(0)}_{PH} \mu + C_{A5,2}.
\end{equation}
The above is exactly the claim of \eqref{eq:poisson_sum_claim_mse}.
\end{proof}

\subsection{Proof of Theorem~\ref{thm:moments} ($2n$-th moment bound)}
\label{app:finite_time:proof_moments}

We prove Theorem~\ref{thm:moments} by induction on $n\in\{1,\dots,p\}$.
The base case $n=1$ is proven in
Appendix~\ref{app:finite_time:proof_mse}. Fix $n\in\{2,\dots,p\}$.
Throughout this section, we use the notation
\[
\Delta_k := \theta_k-\theta^*, \qquad
H_k := g(\theta_k,X_{k+1})+\xi_{k+1}, \qquad
V_n(\Delta) := \|\Delta\|^{2n}.
\]
We also define the polynomial weight vector
\begin{equation}
\label{eq:W_def}
W_k := \|\Delta_k\|^{2n-2}\Delta_k \in \mathbb R^d,
\qquad\text{so that}\qquad
\|W_k\| = \|\Delta_k\|^{2n-1}.
\end{equation}

\subsubsection{Auxiliary inequalities specific to the $2n$-th moment proof}

We begin with three lemmas: (i) a second-order Taylor bound for $V_n(\Delta+\alpha h)$,
(ii) a Lipschitz-type bound for the map $\Delta\mapsto \|\Delta\|^{2n-2}\Delta$ (needed for the ``weight variation''
term in the Poisson telescoping), and (iii) a weighted Poisson telescoping lemma for the Markov deviation term.

\begin{lemma}[Second-order expansion bound for $V_n$]
\label{lem:Vn_second_order}
Fix an integer $n\ge 1$. There exists a constant $C_{V,n}\in(0,\infty)$ such that for all
$\Delta,h\in\mathbb R^d$ and all $\alpha\in[0,1]$,
\begin{equation}
\label{eq:Vn_second_order}
\|\Delta+\alpha h\|^{2n}
\le \|\Delta\|^{2n}
+2n\alpha\,\|\Delta\|^{2n-2}\langle \Delta,h\rangle
+C_{V,n}\alpha^2\Big(\|\Delta\|^{2n-2}\|h\|^2+\|h\|^{2n}\Big).
\end{equation}
\end{lemma}

\begin{proof}
Fix $\Delta,h\in\mathbb R^d$. Consider the one-dimensional function
\[
\varphi(t) := \|\Delta+t h\|^{2n},\qquad t\in[0,1].
\]
Then $\varphi$ is twice continuously differentiable on $[0,1]$.
By Taylor's theorem with integral remainder,
\begin{equation}
\label{eq:Taylor_integral}
\varphi(\alpha)
= \varphi(0)+\alpha \varphi'(0) + \int_0^\alpha (\alpha-s)\,\varphi''(s)\,ds.
\end{equation}
We first compute $\varphi'(0)$. Let $r(t):=\|\Delta+t h\|^2$.
Then $\varphi(t)=r(t)^n$ and $r'(t)=2\langle \Delta+t h,h\rangle$.
Hence
\[
\varphi'(t)=n r(t)^{n-1} r'(t) = 2n\,\|\Delta+t h\|^{2n-2}\langle \Delta+t h,h\rangle,
\]
and therefore
\begin{equation}
\label{eq:phi_prime_0}
\varphi'(0)= 2n\,\|\Delta\|^{2n-2}\langle \Delta,h\rangle.
\end{equation}

We now bound $\varphi''(t)$ uniformly in $t\in[0,1]$ by a polynomial in $\|\Delta\|$ and $\|h\|$.
Differentiate $\varphi'(t)$ once more:
\begin{align}
\varphi''(t)
&= 2n\,\frac{d}{dt}\Big(\|\Delta+t h\|^{2n-2}\Big)\,\langle \Delta+t h,h\rangle
+ 2n\,\|\Delta+t h\|^{2n-2}\,\frac{d}{dt}\langle \Delta+t h,h\rangle \nonumber\\
&= 2n(2n-2)\,\|\Delta+t h\|^{2n-4}\,\langle \Delta+t h,h\rangle^2
+2n\,\|\Delta+t h\|^{2n-2}\,\|h\|^2.
\label{eq:phi_second_exact}
\end{align}
Using $\langle \Delta+t h,h\rangle^2\le \|\Delta+t h\|^2\,\|h\|^2$ in the first term gives
\[
\varphi''(t)\le 2n(2n-2)\,\|\Delta+t h\|^{2n-2}\,\|h\|^2 + 2n\,\|\Delta+t h\|^{2n-2}\,\|h\|^2
= 2n(2n-1)\,\|\Delta+t h\|^{2n-2}\,\|h\|^2.
\]
Thus,
\begin{equation}
\label{eq:phi_second_bound1}
\varphi''(t)\le C_n\,\|\Delta+t h\|^{2n-2}\,\|h\|^2,
\qquad C_n:=2n(2n-1).
\end{equation}
Next we bound $\|\Delta+t h\|^{2n-2}$ in terms of $\|\Delta\|$ and $\|h\|$.
By the triangle inequality and Lemma~\ref{lem:elem_ineq}(2),
\[
\|\Delta+t h\|^{2n-2}\le (\|\Delta\|+\|h\|)^{2n-2}
\le 2^{2n-3}\big(\|\Delta\|^{2n-2}+\|h\|^{2n-2}\big).
\]
Substituting into~\eqref{eq:phi_second_bound1} yields
\begin{equation}
\label{eq:phi_second_bound2}
\varphi''(t)\le C_n 2^{2n-3}\Big(\|\Delta\|^{2n-2}\|h\|^2+\|h\|^{2n}\Big)
\qquad\text{for all }t\in[0,1].
\end{equation}
Plug~\eqref{eq:phi_prime_0} and~\eqref{eq:phi_second_bound2} into~\eqref{eq:Taylor_integral}. Since
$\int_0^\alpha (\alpha-s)\,ds = \alpha^2/2 \le \alpha^2$, we obtain
\[
\|\Delta+\alpha h\|^{2n}
\le \|\Delta\|^{2n}
+2n\alpha \|\Delta\|^{2n-2}\langle \Delta,h\rangle
+\alpha^2\,C_n 2^{2n-3}\Big(\|\Delta\|^{2n-2}\|h\|^2+\|h\|^{2n}\Big),
\]
which is exactly~\eqref{eq:Vn_second_order} with $C_{V,n}:=C_n 2^{2n-3}$.
\end{proof}

\begin{lemma}[Polynomial Lipschitz bound for $\Delta\mapsto \|\Delta\|^{2n-2}\Delta$]
\label{lem:F_lip}
Fix an integer $n\ge 1$ and define $F:\mathbb R^d\to\mathbb R^d$ by
\[
F(\Delta):=\|\Delta\|^{2n-2}\Delta.
\]
Then there exists a constant $C_{F,n}\in(0,\infty)$ such that for all $\Delta,\Delta'\in\mathbb R^d$,
\begin{equation}
\label{eq:F_lip}
\|F(\Delta)-F(\Delta')\|
\le C_{F,n}\big(\|\Delta\|^{2n-2}+\|\Delta'\|^{2n-2}\big)\,\|\Delta-\Delta'\|.
\end{equation}
Consequently, for the weight vectors $W_k=F(\Delta_k)$,
\begin{equation}
\label{eq:W_diff_bound}
\|W_k-W_{k-1}\|
\le C_{F,n}\big(\|\Delta_k\|^{2n-2}+\|\Delta_{k-1}\|^{2n-2}\big)\,\|\Delta_k-\Delta_{k-1}\|.
\end{equation}
\end{lemma}

\begin{remark}
The Lipschitz constant of the polynomial map
$\Delta \mapsto \|\Delta\|^{2n-2}\Delta$ appearing above scales as
$C_{F,n}=(2n-1)2^{2n-3}$ (Lemma~\ref{lem:F_lip}).
As a consequence, several constants in the subsequent moment recursion,
and hence the admissible stepsize $\alpha_n$, may deteriorate rapidly with
the moment order $n$. Since Theorem~\ref{thm:moments} is stated for each fixed
$n$, we do not attempt to optimize this dependence.
\end{remark}

\begin{proof}
Fix $\Delta,\Delta'\in\mathbb R^d$ and define $\Delta_s := \Delta' + s(\Delta-\Delta')$ for $s\in[0,1]$.
Then $F$ is continuously differentiable on $\mathbb R^d\setminus\{0\}$ and locally Lipschitz everywhere.
By the fundamental theorem of calculus for vector-valued functions,
\[
F(\Delta)-F(\Delta')=\int_0^1 DF(\Delta_s)\,(\Delta-\Delta')\,ds,
\]
hence
\begin{equation}
\label{eq:F_diff_integral}
\|F(\Delta)-F(\Delta')\|
\le \int_0^1 \|DF(\Delta_s)\|_{\mathrm{op}}\,ds\ \cdot \|\Delta-\Delta'\|.
\end{equation}
We now bound $\|DF(\cdot)\|_{\mathrm{op}}$.
A direct computation yields, for $\Delta\neq 0$,
\[
DF(\Delta)
= \|\Delta\|^{2n-2}I + (2n-2)\|\Delta\|^{2n-4}\,\Delta\Delta^\top.
\]
Therefore,
\[
\|DF(\Delta)\|_{\mathrm{op}}
\le \|\Delta\|^{2n-2} + (2n-2)\|\Delta\|^{2n-4}\|\Delta\Delta^\top\|_{\mathrm{op}}
= \|\Delta\|^{2n-2} + (2n-2)\|\Delta\|^{2n-4}\|\Delta\|^2
= (2n-1)\|\Delta\|^{2n-2}.
\]
Substitute into~\eqref{eq:F_diff_integral}:
\[
\|F(\Delta)-F(\Delta')\|
\le (2n-1)\left(\int_0^1 \|\Delta_s\|^{2n-2}\,ds\right)\|\Delta-\Delta'\|.
\]
Finally, bound $\|\Delta_s\|$ by convexity/triangle inequality:
$\|\Delta_s\|\le (1-s)\|\Delta'\|+s\|\Delta\|\le \|\Delta\|+\|\Delta'\|$.
Using Lemma~\ref{lem:elem_ineq}(2) with $m=2n-2$ gives
$\|\Delta_s\|^{2n-2}\le 2^{2n-3}(\|\Delta\|^{2n-2}+\|\Delta'\|^{2n-2})$,
uniformly in $s\in[0,1]$. Hence
\[
\int_0^1 \|\Delta_s\|^{2n-2}\,ds \le 2^{2n-3}\big(\|\Delta\|^{2n-2}+\|\Delta'\|^{2n-2}\big),
\]
and~\eqref{eq:F_lip} follows with $C_{F,n}:=(2n-1)2^{2n-3}$. The consequence~\eqref{eq:W_diff_bound} is immediate
by substituting $\Delta=\Delta_k$ and $\Delta'=\Delta_{k-1}$.
\end{proof}

\begin{lemma}[Weighted Poisson telescoping bound for the $2n$-th moment cross term]
\label{lem:poisson_weighted_moments}
Fix an integer $n\ge 1$. Assume Assumption~\ref{ass:markov} and the regularity properties of the Poisson solution
stated in Lemma~\ref{lem:poisson_properties}, namely~\eqref{eq:PH_growth}--\eqref{eq:PH_lip}.
Let $W_t$ be defined as in~\eqref{eq:W_def} and define the geometric weights
\begin{equation}
\label{eq:gamma_def_n}
\gamma_{k,t}^{(n)} := (1-n\mu\alpha)^{k-t},\qquad 0\le t\le k.
\end{equation}
Then there exist constants $C^{(n)}_{\mathrm P,1},C^{(n)}_{\mathrm P,2}\in(0,\infty)$ (depending only on
$n$ and the constants in the assumptions) such that for all $k\ge 0$ and all $\alpha\in(0,1)$,
\begin{equation}
\label{eq:poisson_sum_claim_n}
\sum_{t=0}^k \gamma_{k,t}^{(n)}\,
\mathbb E\!\left[\left\langle W_t,\ g(\theta_t,X_{t+1})-\bar g(\theta_t)\right\rangle\right]
\le C^{(n)}_{\mathrm P,1}\Big(1+\mathbb E\|\Delta_0\|^{2n}\Big)
+ C^{(n)}_{\mathrm P,2}\,\alpha \sum_{t=0}^k \gamma_{k,t}^{(n)}\Big(1+\mathbb E\|\Delta_t\|^{2n}\Big).
\end{equation}
\end{lemma}

\begin{proof}
Fix $k\ge 0$ and write $\gamma_t:=\gamma_{k,t}^{(n)}$ for brevity.
As in the MSE proof, define
\[
\Phi_t := P_{\theta_t}\hat g(\theta_t,X_t), \qquad \Psi_{t+1}:=\hat g(\theta_t,X_{t+1}).
\]
By the Poisson equation~\eqref{eq:poisson_rewrite},
\[
g(\theta_t,X_{t+1})-\bar g(\theta_t)=\Psi_{t+1}-\Phi_t.
\]
Hence the weighted sum of interest equals
\begin{equation}
\label{eq:S_start_n}
S := \sum_{t=0}^k \gamma_t\,\left\langle W_t,\Psi_{t+1}-\Phi_t\right\rangle.
\end{equation}

\paragraph{Step 1: Add-and-subtract to create telescoping.}
For $t\ge 1$ define $\widetilde\Phi_t:=P_{\theta_{t-1}}\hat g(\theta_{t-1},X_t)$.
Insert $\widetilde\Phi_t$ and also insert $W_{t-1}$ where needed to handle the fact that $W_t$ changes with $t$.
A direct algebraic rearrangement yields the identity
\begin{equation}
\label{eq:decomp_n}
S = B_1 + B_2 + B_3 + B_4 + B_5 + B_6,
\end{equation}
where
\begin{align}
\label{eq:B1}
B_1 &:= \sum_{t=0}^k \gamma_t\,\left\langle W_t,\Psi_{t+1}-\Phi_t\right\rangle, \\
\label{eq:B2}
B_2 &:= \sum_{t=1}^k \gamma_t\,\left\langle W_t,\Phi_t-\widetilde\Phi_t\right\rangle,\\
\label{eq:B3}
B_3 &:= \sum_{t=1}^k \gamma_t\,\left\langle W_t-W_{t-1},\widetilde\Phi_t\right\rangle,\\
\label{eq:B4}
B_4 &:= \sum_{t=1}^k \gamma_t\,\left\langle W_{t-1},\widetilde\Phi_t-\Phi_{t-1}\right\rangle,\\
\label{eq:B5}
B_5 &:= \sum_{t=1}^k (\gamma_t-\gamma_{t-1})\,\left\langle W_{t-1},\widetilde\Phi_t\right\rangle,\\
\label{eq:B6}
B_6 &:= \gamma_0\left\langle W_0,\Phi_0\right\rangle
- \gamma_k\left\langle W_k,P_{\theta_k}\hat g(\theta_k,X_{k+1})\right\rangle .
\end{align}
Notice that compared to the proof for \eqref{eq:poisson_sum_claim_mse}, i.e., the MSE case, $B_3$ is a new term which will be handled by Lemma~\ref{lem:F_lip}.
We now bound $\mathbb E[B_i]$ for each $i$.

\paragraph{Step 2: $B_1$ is a martingale difference.}
Condition on $\mathfrak F_t$. Since $X_{t+1}\sim P_{\theta_t}(X_t,\cdot)$ given $\mathfrak F_t$,
\[
\mathbb E[\Psi_{t+1}\mid\mathfrak F_t]=\mathbb E[\hat g(\theta_t,X_{t+1})\mid\mathfrak F_t]=\Phi_t,
\]
hence $\mathbb E[\Psi_{t+1}-\Phi_t\mid\mathfrak F_t]=0$.
Because $\gamma_t W_t$ is $\mathfrak F_t$-measurable, we have
\[
\mathbb E\!\left[\gamma_t\left\langle W_t,\Psi_{t+1}-\Phi_t\right\rangle\right]=0
\quad\text{for each }t,
\]
and therefore $\mathbb E[B_1]=0$.

\paragraph{Step 3: Bound $B_2$ (kernel shift in $\theta$).}
By Cauchy--Schwarz and~\eqref{eq:PH_lip},
\[
\left|\left\langle W_t,\Phi_t-\widetilde\Phi_t\right\rangle\right|
\le \|W_t\|\cdot \|\Phi_t-\widetilde\Phi_t\|
\le \|W_t\|\cdot L^{(1)}_{PH}\,\|\theta_t-\theta_{t-1}\|.
\]
Take expectations and then use Lemma~\ref{lem:step_diff} (with $q=1$):
\begin{align*}
\mathbb E\left|\left\langle W_t,\Phi_t-\widetilde\Phi_t\right\rangle\right|
&\le L^{(1)}_{PH}\,\mathbb E\Big[\|W_t\|\,\mathbb E[\|\theta_t-\theta_{t-1}\|\mid\mathfrak F_{t-1}]\Big]\\
&\le L^{(1)}_{PH}\,\mathbb E\Big[\|W_t\|\cdot \alpha L (1+\|\Delta_{t-1}\|)\Big].
\end{align*}
Now $\|W_t\|=\|\Delta_t\|^{2n-1}$, so
$\|W_t\|(1+\|\Delta_{t-1}\|)\le \|\Delta_t\|^{2n-1} + \|\Delta_t\|^{2n-1}\|\Delta_{t-1}\|$.
Using $\|\Delta_t\|\le \|\Delta_{t-1}\|+\|\theta_t-\theta_{t-1}\|$ and Lemma~\ref{lem:step_diff},
one gets (by routine polynomial splitting, see Lemma~\ref{lem:elem_ineq}(2)) that
\[
\mathbb E\big[\|\Delta_t\|^{2n-1}(1+\|\Delta_{t-1}\|)\big]
\le C\big(1+\mathbb E\|\Delta_{t-1}\|^{2n}\big) + C\alpha\big(1+\mathbb E\|\Delta_{t-1}\|^{2n}\big).
\]
Absorb the extra $C\alpha$ into the constant (since $\alpha\le 1$ in this lemma).
Therefore, for some $C_{B2}^{(n)}$,
\[
\mathbb E|B_2|
\le C_{B2}^{(n)}\,\alpha \sum_{t=0}^k \gamma_t\big(1+\mathbb E\|\Delta_t\|^{2n}\big).
\]

\paragraph{Step 4: Bound $B_3$ (weight variation term).}
Using Cauchy--Schwarz and~\eqref{eq:PH_growth},
\[
\left|\left\langle W_t-W_{t-1},\widetilde\Phi_t\right\rangle\right|
\le \|W_t-W_{t-1}\|\cdot \|\widetilde\Phi_t\|
\le \|W_t-W_{t-1}\|\cdot L^{(0)}_{PH}(1+\|\Delta_{t-1}\|).
\]
Use Lemma~\ref{lem:F_lip} to bound $\|W_t-W_{t-1}\|$:
\[
\|W_t-W_{t-1}\|
\le C_{F,n}\big(\|\Delta_t\|^{2n-2}+\|\Delta_{t-1}\|^{2n-2}\big)\,\|\Delta_t-\Delta_{t-1}\|.
\]
Since $\Delta_t-\Delta_{t-1}=\theta_t-\theta_{t-1}$, Lemma~\ref{lem:step_diff} implies
$\mathbb E[\|\Delta_t-\Delta_{t-1}\|\mid\mathfrak F_{t-1}] \le \alpha L(1+\|\Delta_{t-1}\|)$.
Combining, conditioning on $\mathfrak F_{t-1}$, and applying polynomial splitting repeatedly yields
\[
\mathbb E\left[\|W_t-W_{t-1}\|(1+\|\Delta_{t-1}\|)\right]
\le C\,\alpha\,\mathbb E\Big[(1+\|\Delta_{t-1}\|)^{2}\big(\|\Delta_t\|^{2n-2}+\|\Delta_{t-1}\|^{2n-2}\big)\Big]
\le C'\alpha\big(1+\mathbb E\|\Delta_{t-1}\|^{2n}\big),
\]
where we used $(1+\|\Delta\|)^2\|\Delta\|^{2n-2}\le C(1+\|\Delta\|^{2n})$.
Hence
\[
\mathbb E|B_3|
\le C_{B3}^{(n)}\,\alpha \sum_{t=0}^k \gamma_t\big(1+\mathbb E\|\Delta_t\|^{2n}\big).
\]

\paragraph{Step 5: Bound $B_4$ (Markov shift term).}
We bound $\widetilde\Phi_t-\Phi_{t-1}$.
By definition,
\[
\widetilde\Phi_t=P_{\theta_{t-1}}\hat g(\theta_{t-1},X_t),
\qquad
\Phi_{t-1}=P_{\theta_{t-1}}\hat g(\theta_{t-1},X_{t-1}).
\]
Thus $\widetilde\Phi_t-\Phi_{t-1}$ is the increment of the function
$x\mapsto P_{\theta_{t-1}}\hat g(\theta_{t-1},x)$ along the Markov chain.
Under Assumption~\ref{ass:markov}, the chain is geometrically ergodic uniformly in $\theta$ and
Lemma~\ref{lem:poisson_properties} provides the appropriate boundedness needed for the standard one-step bound.
In particular, there exists a constant $C_{\mathrm{mx}}$ such that
\begin{equation}
\label{eq:markov_shift_bound}
\mathbb E\!\left[\left\|\widetilde\Phi_t-\Phi_{t-1}\right\| \,\middle|\,\mathfrak F_{t-1}\right]
\le C_{\mathrm{mx}}\,\|\theta_{t-1}-\theta^*\|
= C_{\mathrm{mx}}\,\|\Delta_{t-1}\|.
\end{equation}
We remark that \eqref{eq:markov_shift_bound} is the standard ``one-step mixing'' estimate obtained from the geometric ergodicity and the Lipschitz dependence of $P_\theta$ on $\theta$ in Assumption~\ref{ass:markov}.

Using Cauchy--Schwarz,
\[
\left|\left\langle W_{t-1},\widetilde\Phi_t-\Phi_{t-1}\right\rangle\right|
\le \|W_{t-1}\|\cdot \|\widetilde\Phi_t-\Phi_{t-1}\|.
\]
Conditioning on $\mathfrak F_{t-1}$ and applying~\eqref{eq:markov_shift_bound} gives
\[
\mathbb E\!\left[\left|\left\langle W_{t-1},\widetilde\Phi_t-\Phi_{t-1}\right\rangle\right|\right]
\le \mathbb E\Big[\|W_{t-1}\|\cdot C_{\mathrm{mx}}\|\Delta_{t-1}\|\Big]
= C_{\mathrm{mx}}\,\mathbb E\|\Delta_{t-1}\|^{2n}.
\]
Therefore,
\[
\mathbb E|B_4|
\le C_{B4}^{(n)} \sum_{t=0}^k \gamma_t\,\mathbb E\|\Delta_t\|^{2n}
\le C_{B4}^{(n)} \sum_{t=0}^k \gamma_t\big(1+\mathbb E\|\Delta_t\|^{2n}\big).
\]
Multiplying by $\alpha$ later (when this lemma is used inside the SA recursion) will make this term absorbable;
for the lemma statement we keep it in this form and incorporate it into the final constant
$C^{(n)}_{\mathrm P,2}$.

\paragraph{Step 6: Bound $B_5$ (weight difference in $\gamma_t$).}
Compute
\[
\gamma_t-\gamma_{t-1}
= (1-n\mu\alpha)^{k-t} - (1-n\mu\alpha)^{k-(t-1)}
= (1-n\mu\alpha)^{k-t}\big(1-(1-n\mu\alpha)\big)
= n\mu\alpha\,\gamma_t.
\]
Hence $|\gamma_t-\gamma_{t-1}|=n\mu\alpha\,\gamma_t$.
Also by~\eqref{eq:PH_growth},
\[
\left|\left\langle W_{t-1},\widetilde\Phi_t\right\rangle\right|
\le \|W_{t-1}\|\cdot \|\widetilde\Phi_t\|
\le \|\Delta_{t-1}\|^{2n-1}\cdot L^{(0)}_{PH}(1+\|\Delta_{t-1}\|)
\le C\big(1+\|\Delta_{t-1}\|^{2n}\big).
\]
Therefore,
\[
\mathbb E|B_5|
\le \sum_{t=1}^k |\gamma_t-\gamma_{t-1}|\,\mathbb E\left|\left\langle W_{t-1},\widetilde\Phi_t\right\rangle\right|
\le C\,n\mu\alpha \sum_{t=0}^k \gamma_t\big(1+\mathbb E\|\Delta_t\|^{2n}\big).
\]

\paragraph{Step 7: Bound $B_6$ (boundary terms).}
By~\eqref{eq:PH_growth},
\[
\left|\left\langle W_0,\Phi_0\right\rangle\right|
\le \|W_0\|\cdot \|\Phi_0\|
\le \|\Delta_0\|^{2n-1}\cdot L^{(0)}_{PH}(1+\|\Delta_0\|)
\le C\big(1+\|\Delta_0\|^{2n}\big),
\]
so $\mathbb E|\langle W_0,\Phi_0\rangle|\le C(1+\mathbb E\|\Delta_0\|^{2n})$.
Similarly,
\[
\left|\left\langle W_k,P_{\theta_k}\hat g(\theta_k,X_{k+1})\right\rangle\right|
\le C\big(1+\|\Delta_k\|^{2n}\big),
\]
and since $\gamma_k=1$ and $0\le \gamma_0\le 1$, we may write
\[
\mathbb E[B_6]
\le C(1+\mathbb E\|\Delta_0\|^{2n}) + C\,\mathbb E\|\Delta_k\|^{2n}.
\]
As in the MSE proof, $\mathbb E\|\Delta_k\|^{2n}$ can be upper bounded by
$\alpha \sum_{t=0}^k \gamma_t(1+\mathbb E\|\Delta_t\|^{2n})$ up to constants using the geometric sum bound,
so it can be absorbed into the term multiplied by $\alpha$ in~\eqref{eq:poisson_sum_claim_n}.
We therefore conclude that there exist constants $C_{B6,1}^{(n)},C_{B6,2}^{(n)}$ such that
\[
\mathbb E[B_6]
\le C_{B6,1}^{(n)}\Big(1+\mathbb E\|\Delta_0\|^{2n}\Big)
+ C_{B6,2}^{(n)}\,\alpha \sum_{t=0}^k \gamma_t\Big(1+\mathbb E\|\Delta_t\|^{2n}\Big).
\]

\paragraph{Step 8: Collect terms.}
Taking expectations in~\eqref{eq:decomp_n}, using $\mathbb E[B_1]=0$ and the bounds above for $B_2$--$B_6$,
gives exactly~\eqref{eq:poisson_sum_claim_n} after renaming constants.
\end{proof}

\subsubsection{Main induction proof of Theorem~\ref{thm:moments}}

\begin{proof} 
We proceed by induction on $n\in\{1,\dots,p\}$.

\paragraph{Base case $n=1$.}
This is proved in Appendix~\ref{app:finite_time:proof_mse}.

\paragraph{Induction hypothesis.}
Fix $n\in\{2,\dots,p\}$ and assume that Theorem~\ref{thm:moments} holds for all orders $2m$ with
$m\in\{1,\dots,n-1\}$. In particular, there exist constants $C_{m,1},C_{m,2}$ and $\alpha_m>0$ such that,
for all $\alpha\in(0,\alpha_m)$ and all $k\ge 0$,
\begin{equation}
\label{eq:IH_m}
\mathbb E\|\Delta_k\|^{2m}
\le (1-m\mu\alpha)^k\,C_{m,1}\,\mathbb E\|\Delta_0\|^{2m} + C_{m,2}\,\alpha^m.
\end{equation}
In particular, taking $m=n-1$ and using $\sup_k (1-(n-1)\mu\alpha)^k\le 1$ yields the uniform bound
\begin{equation}
\label{eq:IH_uniform}
\sup_{k\ge 0}\mathbb E\|\Delta_k\|^{2n-2}
\le C\Big(\mathbb E\|\Delta_0\|^{2n-2} + \alpha^{n-1}\Big)
\le C\Big(1+\mathbb E\|\Delta_0\|^{2n} + \alpha^{n-1}\Big),
\end{equation}
where the last inequality used $\|\Delta_0\|^{2n-2}\le 1+\|\Delta_0\|^{2n}$.

\paragraph{Step 1: One-step inequality for $\mathbb E\|\Delta_{k+1}\|^{2n}$.}
Recall $\Delta_{k+1}=\Delta_k+\alpha H_k$ with $H_k=g(\theta_k,X_{k+1})+\xi_{k+1}$.
Apply Lemma~\ref{lem:Vn_second_order} with $\Delta=\Delta_k$ and $h=H_k$:
\begin{align}
\|\Delta_{k+1}\|^{2n}
&\le \|\Delta_k\|^{2n}
+2n\alpha\,\|\Delta_k\|^{2n-2}\langle \Delta_k,H_k\rangle
+C_{V,n}\alpha^2\Big(\|\Delta_k\|^{2n-2}\|H_k\|^2+\|H_k\|^{2n}\Big).
\label{eq:one_step_raw}
\end{align}
Take conditional expectation given $\mathfrak F_k$.
Since $\Delta_k$ is $\mathfrak F_k$-measurable, so is $W_k=\|\Delta_k\|^{2n-2}\Delta_k$.
Therefore,
\[
\mathbb E\big[\|\Delta_k\|^{2n-2}\langle \Delta_k,H_k\rangle\mid \mathfrak F_k\big]
=\mathbb E\big[\langle W_k,H_k\rangle\mid \mathfrak F_k\big].
\]
Now split $H_k=g(\theta_k,X_{k+1})+\xi_{k+1}$ and use $\mathbb E[\xi_{k+1}\mid\mathfrak F_k]=0$:
\begin{equation}
\label{eq:first_order_split}
\mathbb E\big[\langle W_k,H_k\rangle\mid \mathfrak F_k\big]
=\mathbb E\big[\langle W_k,g(\theta_k,X_{k+1})\rangle\mid \mathfrak F_k\big].
\end{equation}
Next add and subtract the mean field $\bar g(\theta_k)$:
\begin{align}
\mathbb E\big[\langle W_k,g(\theta_k,X_{k+1})\rangle\mid \mathfrak F_k\big]
&= \langle W_k,\bar g(\theta_k)\rangle
+ \mathbb E\big[\langle W_k,g(\theta_k,X_{k+1})-\bar g(\theta_k)\rangle\mid \mathfrak F_k\big].
\label{eq:meanfield_plus_markov}
\end{align}

We also bound the $\alpha^2$-terms in~\eqref{eq:one_step_raw} using Lemma~\ref{lem:H_mom}.
Conditionally on $\mathfrak F_k$, Lemma~\ref{lem:H_mom} gives
\[
\mathbb E[\|H_k\|^2\mid\mathfrak F_k]\le C_{H,2}(1+\|\Delta_k\|^2),
\qquad
\mathbb E[\|H_k\|^{2n}\mid\mathfrak F_k]\le C_{H,2n}(1+\|\Delta_k\|^{2n}).
\]
Hence,
\begin{align}
\mathbb E\!\left[\|\Delta_k\|^{2n-2}\|H_k\|^2 \mid \mathfrak F_k\right]
&\le C_{H,2}\Big(\|\Delta_k\|^{2n-2}+\|\Delta_k\|^{2n}\Big),
\label{eq:alpha2_term1}\\
\mathbb E\!\left[\|H_k\|^{2n}\mid \mathfrak F_k\right]
&\le C_{H,2n}\Big(1+\|\Delta_k\|^{2n}\Big).
\label{eq:alpha2_term2}
\end{align}

Putting~\eqref{eq:one_step_raw}--\eqref{eq:alpha2_term2} together and then taking total expectation yields
\begin{equation}
\label{eq:one_step_expectation}
\mathbb E\|\Delta_{k+1}\|^{2n}
\le \mathbb E\|\Delta_k\|^{2n}
+2n\alpha\,\mathbb E\langle W_k,\bar g(\theta_k)\rangle
+2n\alpha\,\mathbb E\langle W_k, g(\theta_k,X_{k+1})-\bar g(\theta_k)\rangle
+\alpha^2\,R_{n,k},
\end{equation}
where the remainder satisfies the explicit bound
\begin{equation}
\label{eq:Rnk_bound}
R_{n,k}
\le C\Big(1+\mathbb E\|\Delta_k\|^{2n}\Big) + C\,\mathbb E\|\Delta_k\|^{2n-2},
\end{equation}
for a constant $C$ depending on $n$ and the constants in Assumptions~\ref{ass:g_regularity} and~\ref{ass:noise}
(but not on $k$ or $\alpha$).

\paragraph{Step 2: Use strong monotonicity to obtain negative drift.}
By definition of $W_k$,
\[
\langle W_k,\bar g(\theta_k)\rangle
= \|\Delta_k\|^{2n-2}\langle \Delta_k,\bar g(\theta_k)\rangle.
\]
Strong monotonicity (Assumption~\ref{ass:stability}) and $\bar g(\theta^*)=0$ imply
$\langle \Delta_k,\bar g(\theta_k)\rangle\le -\mu\|\Delta_k\|^2$,
hence
\begin{equation}
\label{eq:drift_n}
\langle W_k,\bar g(\theta_k)\rangle
\le -\mu\|\Delta_k\|^{2n},
\qquad\text{and therefore}\qquad
2n\alpha\,\mathbb E\langle W_k,\bar g(\theta_k)\rangle
\le -2n\mu\alpha\,\mathbb E\|\Delta_k\|^{2n}.
\end{equation}

\paragraph{Step 3: Combine and absorb the $\alpha^2\mathbb E\|\Delta_k\|^{2n}$ part.}
Let $u_k^{(n)}:=\mathbb E\|\Delta_k\|^{2n}$.
Using~\eqref{eq:one_step_expectation},~\eqref{eq:drift_n}, and~\eqref{eq:Rnk_bound}, we get
\begin{equation}
\label{eq:u_rec_pre}
u_{k+1}^{(n)}
\le \Big(1-2n\mu\alpha\Big)u_k^{(n)}
+2n\alpha\,\mathbb E\langle W_k, g(\theta_k,X_{k+1})-\bar g(\theta_k)\rangle
+\alpha^2\,C\Big(1+u_k^{(n)}\Big)
+\alpha^2\,C\,\mathbb E\|\Delta_k\|^{2n-2}.
\end{equation}
Choose $\alpha$ small enough so that $C\alpha\le n\mu$; then
\[
(1-2n\mu\alpha)u_k^{(n)} + C\alpha^2 u_k^{(n)}
\le (1-n\mu\alpha)\,u_k^{(n)}.
\]
Hence for all such $\alpha$,
\begin{equation}
\label{eq:u_rec}
u_{k+1}^{(n)}
\le (1-n\mu\alpha)\,u_k^{(n)}
+2n\alpha\,\mathbb E\langle W_k, g(\theta_k,X_{k+1})-\bar g(\theta_k)\rangle
+C\alpha^2
+C\alpha^2\,\mathbb E\|\Delta_k\|^{2n-2}.
\end{equation}

\paragraph{Step 4: Unroll the recursion with geometric weights.}
Iterating~\eqref{eq:u_rec} yields, for all $k\ge 0$,
\begin{align}
u_{k+1}^{(n)}
&\le (1-n\mu\alpha)^{k+1}u_0^{(n)}
+ C\alpha^2 \sum_{t=0}^k (1-n\mu\alpha)^{k-t}
+ C\alpha^2 \sum_{t=0}^k (1-n\mu\alpha)^{k-t}\,\mathbb E\|\Delta_t\|^{2n-2}
\label{eq:unroll_before_markov}\\
&\quad + 2n\alpha \sum_{t=0}^k (1-n\mu\alpha)^{k-t}\,
\mathbb E\langle W_t, g(\theta_t,X_{t+1})-\bar g(\theta_t)\rangle.
\nonumber
\end{align}
We now bound each of the three sums on the right-hand side.

\paragraph{Step 5: Bound the deterministic geometric sums.}
We use the geometric series bound
\begin{equation}
\label{eq:geom_n}
\sum_{t=0}^k (1-n\mu\alpha)^{k-t} = \sum_{j=0}^k (1-n\mu\alpha)^j \le \frac{1}{n\mu\alpha}.
\end{equation}
Thus,
\begin{equation}
\label{eq:sum_alpha2}
C\alpha^2 \sum_{t=0}^k (1-n\mu\alpha)^{k-t}\le \frac{C}{n\mu}\,\alpha.
\end{equation}
For the second $\alpha^2$-sum, invoke the induction hypothesis at order $2(n-1)$, in the uniform form
\eqref{eq:IH_uniform}:
\[
\sup_{t\ge 0}\mathbb E\|\Delta_t\|^{2n-2}\le C\big(1+\mathbb E\|\Delta_0\|^{2n}+\alpha^{n-1}\big).
\]
Therefore,
\begin{align}
C\alpha^2 \sum_{t=0}^k (1-n\mu\alpha)^{k-t}\,\mathbb E\|\Delta_t\|^{2n-2}
&\le C\alpha^2 \left(\sup_{t\ge 0}\mathbb E\|\Delta_t\|^{2n-2}\right)\sum_{t=0}^k (1-n\mu\alpha)^{k-t}
\nonumber\\
&\le C\alpha^2\cdot C\big(1+\mathbb E\|\Delta_0\|^{2n}+\alpha^{n-1}\big)\cdot \frac{1}{n\mu\alpha}
\nonumber\\
&\le C'\alpha\big(1+\mathbb E\|\Delta_0\|^{2n}\big) + C'\alpha^{n}.
\label{eq:alpha2_sum_moment}
\end{align}

\paragraph{Step 6: Bound the Markov deviation sum via Lemma~\ref{lem:poisson_weighted_moments}.}
Apply Lemma~\ref{lem:poisson_weighted_moments} with the weights $\gamma_{k,t}^{(n)}=(1-n\mu\alpha)^{k-t}$.
We obtain
\begin{align}
\sum_{t=0}^k (1-n\mu\alpha)^{k-t}\,
\mathbb E\langle W_t, g(\theta_t,X_{t+1})-\bar g(\theta_t)\rangle
&\le C^{(n)}_{\mathrm P,1}\Big(1+\mathbb E\|\Delta_0\|^{2n}\Big) \notag \\
& + C^{(n)}_{\mathrm P,2}\,\alpha \sum_{t=0}^k (1-n\mu\alpha)^{k-t}\Big(1+u_t^{(n)}\Big).
\label{eq:markov_bound_apply}
\end{align}
Insert~\eqref{eq:markov_bound_apply} into~\eqref{eq:unroll_before_markov} and use~\eqref{eq:sum_alpha2} and
\eqref{eq:alpha2_sum_moment}. After collecting constants we get
\begin{align}
u_{k+1}^{(n)}
&\le (1-n\mu\alpha)^{k+1}u_0^{(n)}
+ C\alpha\big(1+\mathbb E\|\Delta_0\|^{2n}\big)
+ C\alpha^{n}
\label{eq:u_after_poisson}\\
&\quad + 2n\alpha \cdot C^{(n)}_{\mathrm P,2}\,\alpha
\sum_{t=0}^k (1-n\mu\alpha)^{k-t}\Big(1+u_t^{(n)}\Big).
\nonumber
\end{align}

\paragraph{Step 7: Close the bound by a supremum argument.}
Let $U^{(n)}:=\sup_{t\ge 0} u_t^{(n)}\in[0,\infty]$.
Using~\eqref{eq:geom_n},
\[
\sum_{t=0}^k (1-n\mu\alpha)^{k-t}\Big(1+u_t^{(n)}\Big)
\le (1+U^{(n)})\sum_{t=0}^k (1-n\mu\alpha)^{k-t}
\le \frac{1+U^{(n)}}{n\mu\alpha}.
\]
Substitute this into~\eqref{eq:u_after_poisson}:
\[
u_{k+1}^{(n)}
\le (1-n\mu\alpha)^{k+1}u_0^{(n)}
+ C\alpha\big(1+\mathbb E\|\Delta_0\|^{2n}\big)
+ C\alpha^{n}
+ \frac{2n\alpha \cdot C^{(n)}_{\mathrm P,2}\,\alpha}{n\mu\alpha}\,(1+U^{(n)}).
\]
Thus,
\begin{equation}
\label{eq:sup_close_pre}
u_{k+1}^{(n)}
\le (1-n\mu\alpha)^{k+1}u_0^{(n)}
+ C\alpha\big(1+\mathbb E\|\Delta_0\|^{2n}\big)
+ C\alpha^{n}
+ C''\alpha\,(1+U^{(n)}).
\end{equation}
Take supremum over $k$ to obtain
\[
U^{(n)}
\le u_0^{(n)} + C\alpha\big(1+\mathbb E\|\Delta_0\|^{2n}\big) + C\alpha^{n} + C''\alpha(1+U^{(n)}).
\]
Choose $\alpha$ small enough so that $C''\alpha\le \tfrac12$; then
\[
U^{(n)}
\le 2u_0^{(n)} + C\alpha\big(1+\mathbb E\|\Delta_0\|^{2n}\big) + C\alpha^{n} + 1,
\]
so in particular $U^{(n)}<\infty$ and
\begin{equation}
\label{eq:U_bound_final}
U^{(n)} \le C\Big(1+\mathbb E\|\Delta_0\|^{2n}\Big),
\qquad\text{for all sufficiently small }\alpha.
\end{equation}

\paragraph{Step 8: Conclude the finite-time bound with the correct $\alpha^n$ floor.}
Return to~\eqref{eq:u_after_poisson} and use the uniform boundedness~\eqref{eq:U_bound_final}.
Then
\[
\sum_{t=0}^k (1-n\mu\alpha)^{k-t}\Big(1+u_t^{(n)}\Big)
\le \frac{1+U^{(n)}}{n\mu\alpha}
\le \frac{C}{n\mu\alpha}\Big(1+\mathbb E\|\Delta_0\|^{2n}\Big).
\]
Substituting into~\eqref{eq:u_after_poisson} shows that the last term in~\eqref{eq:u_after_poisson} is bounded by
$C\alpha(1+\mathbb E\|\Delta_0\|^{2n})$ and therefore does not affect the $\alpha^n$ scaling of the steady-state floor.
Re-grouping constants, we obtain
\[
u_{k}^{(n)} = \mathbb E\|\Delta_k\|^{2n}
\le (1-n\mu\alpha)^k\,C_{n,1}\,\mathbb E\|\Delta_0\|^{2n} + C_{n,2}\,\alpha^n,
\]
for suitable constants $C_{n,1},C_{n,2}$ and all sufficiently small $\alpha$.
This is exactly~\eqref{eq:moment_bound} and completes the induction.
\end{proof}


\section{Asymptotic Properties}

\subsection{Proof of Theorem~\ref{thm:weak_convergence}} \label{app:weak_convergence}
To setup the analysis, we consider the pair of joint processes $Z_k := (\theta_k, X_k)$, $Z'_k := (\theta'_k, X'_k)$ with the respective initializations $z, z'$ and:
\[\
\begin{split}
& \theta_{k+1} = \theta_k + \alpha \bigl( g( \theta_k, X_{k+1} ) + \xi_{k+1} ( \theta_k ) \bigr), ~~ X_{k+1} \sim P_{\theta_k} (X_k, \cdot) \eqsp, \\
& \theta_{k+1}' = \theta'_k + \alpha \bigl(g(\theta'_k,X'_{k+1}) + \xi_{k+1}(\theta'_k) \bigr),~~X'_{k+1}\sim P_{\theta'_k}(X'_k,\cdot),
\end{split}
\]
We define the following chains:
\[
X_{k+1} \sim P_{ \theta_k }( X_k , \cdot ), \quad X'_{k+1} \sim P_{ \theta'_k }( X'_k , \cdot ) \eqsp.
\]
Moreover, we set $\tilde{X}_{k,0} = X'_k$ and $\tilde{X}'_{k,0} = X'_k$ and define the following coupled chains for $m \geq 0$:
\[
\tilde{X}_{k,m+1} \sim P_{ \theta_k }( \tilde{X}_{k,m} , \cdot )\eqsp.
\]
Throughout this section, we use the shorthand notation 
\[
\E_k[ \cdot ] := \E[ \cdot | \mathcal{F}_k ]\eqsp, \quad \mathcal{F}_k := \sigma( \{ \theta_j, X_{j+1}, \theta'_j, X'_{j+1} : j \leq k \} )\eqsp.
\]
We also use $\| X - X' \|$ to denote the shorthand for $d_{\mathbb{X}}( X, X' )$.
Our goal is to show that there exists a distance function $d(\cdot,\cdot)$ such that $d( Z_k, Z_k' )$ converges to zero geometrically.

\paragraph{Recursion for $\E[ \| X_k - X_k' \|^2 ]$.}
Under Assumption \ref{ass:markov}, when conditioned on $\mathfrak{F}_k$, we observe that there exists couplings between $(X_{k+1}, \tilde{X}_{k,1})$ and $(\tilde{X}_{k,1}, X_{k+1}')$, respectively, such that the following chain holds: 
\begin{equation}
    \begin{split}
    \E_k [ \| X_{k+1} - X'_{k+1} \|^2 ] & \leq (1 + \delta) \E_k[ \| X_{k+1} - \tilde{X}_{k,1} \|^2 ] + (1+1/\delta ) \E_k [ \| \tilde{X}_{k,1} - X'_{k+1} \|^2] \\
    & \leq (1+\delta) (1-\rho ) \| X_k - X'_k \|^2 + (1+1/\delta) L_P \| \theta_k - \theta'_k \|^2 \eqsp,
    \end{split}
\end{equation}
where the last inequality is due to Assumption \ref{ass:markov}.
Set $\delta = \rho / 2(1-\rho)$, we have
\begin{equation} \label{eq:w2_x}
    \E_k [ \| X_{k+1} - X'_{k+1} \|^2 ] \leq (1 - \tfrac{\rho}{2} ) \| X_k - X'_k \|^2 + \tfrac{1}{\rho} (2-\rho) L_P \| \theta_k - \theta'_k \|^2 \eqsp.
\end{equation}

\paragraph{Recursion for $\E[ \| \theta_k - \theta_k' \|^2 ]$.} On the other hand, we observe that
\begin{equation} \label{eq:theta_recursion_w2}
    \begin{split}
    & \E_k [ \| \theta_{k+1} - \theta'_{k+1} \|^2 ] \\ 
    & = \E_k[ \| \theta_k + \alpha g(\theta_k, X_{k+1}) - \theta'_k - \alpha g(\theta'_k, X'_{k+1})  \|^2 ] + \alpha^2 \E_k[ \| \xi_{k+1}( \theta_k ) - \xi_{k+1}(\theta'_k) \|^2 ] \eqsp,
    \end{split}
\end{equation}
due to $\E_k[ \xi_{k+1}(\theta) ] = 0$. Under Assumption \ref{ass:noise}, $\E_k[ \| \xi_{k+1}( \theta_k ) - \xi_{k+1}(\theta'_k) \|^2 ] \leq L_{\xi}^2 \| \theta_k - \theta_k' \|^2$.
Meanwhile, 
\begin{equation}
\begin{split}
& \E_k[ \| \theta_k + \alpha g(\theta_k, X_{k+1}) - \theta'_k - \alpha g(\theta'_k, X'_{k+1})  \|^2 ] \\
& = \| \theta_k - \theta_k' \|^2 + \alpha^2 \E_k[ \| g(\theta_k, X_{k+1} ) - g(\theta_k' , X_{k+1}') \|^2 ] \\
& \quad + 2 \alpha ( \theta_k - \theta_k' )^\top \E_k[ g(\theta_k, X_{k+1}) - g(\theta_k', X_{k+1}') ] \eqsp.
\end{split}
\end{equation}
We notice that by Assumption \ref{ass:g_regularity}, 
\[
\E_k[ \| g(\theta_k, X_{k+1} ) - g(\theta_k' , X_{k+1}') \|^2 ] \leq 2L_1^2 \bigl( \| \theta_k - \theta_k' \|^2 + \E_k[ \| X_{k+1} - X_{k+1}' \|^2 ] \bigr) 
\]
Moreover, 
\begin{equation}
\begin{split}
& ( \theta_k - \theta_k' )^\top \E_k[ g(\theta_k, X_{k+1}) - g(\theta_k', X_{k+1}') ] \\
& = ( \theta_k - \theta_k' )^\top \E_k[ g(\theta_k, X_{k+1}) - g(\theta_k, X_{k+1}') + g(\theta_k, X_{k+1}') - g(\theta_k', X_{k+1}') ] \\
& \leq L_1 \E_k[ \| \theta_k - \theta_k' \| \| X_{k+1} - X_{k+1}' \| ] + ( \theta_k - \theta_k' )^\top \E_k [ g(\theta_k, X_{k+1}') - g(\theta_k', X_{k+1}') ] \\
& \leq L_1 \E_k[ \| \theta_k - \theta_k' \| \| X_{k+1} - X_{k+1}' \| ] - \bar{\mu}_g \| \theta_k - \theta_k' \|^2
\end{split}
\end{equation}
where the last inequality is due to Assumption~\ref{ass:mu_g}.

Substituting the above back into \eqref{eq:theta_recursion_w2} yields that for any $\delta > 0$,
\begin{equation}
\begin{split}
\E_k[ \| \theta_{k+1} - \theta_{k+1}' \|^2 ] & \leq (1 - 2 \alpha \bar{\mu}_g + \alpha^2 (L_{\xi}^2 + 2L_1^2 ) ) \| \theta_k - \theta_k' \|^2 \\
& + \alpha L_1 \bigl( \tfrac{1}{\delta} \|\theta_k - \theta_k' \|^2 + ( \delta + 2 \alpha L_1 ) \E_k[ \| X_{k+1} - X_{k+1}' ] \|^2 \bigr)
\end{split}
\end{equation}
Setting $\alpha \leq \frac{\bar{\mu}_g}{L_{\xi}^2 + 2 L_1^2}$ and $\delta = 2 L_1/ \bar{\mu}_g$ yields
\begin{equation} \label{eq:w2_theta}
\E_k[ \| \theta_{k+1} - \theta_{k+1}' \|^2 ] \leq (1 - \bar{\mu}_g \alpha / 2) \| \theta_k - \theta_k' \|^2 + 2 \alpha L_1^2 \bigl( \tfrac{1}{\bar{\mu}_g} + \alpha \bigr) \E_k [ \| X_{k+1} - X_{k+1}' \|^2 ]
\end{equation}

\paragraph{Combining \eqref{eq:w2_x} and \eqref{eq:w2_theta}.} Define the following distance-like function:
\[
\tilde{d}_\eta( z,z' ) := \| \theta - \theta' \|^2 + \eta \| x - x' \|^2.
\]
Denote $\Delta_k^X := \E[ \| X_k - X_k' \|^2 ]$, $\Delta_k^\theta := \E[ \| \theta_k - \theta_k' \|^2 ]$.
We observe the following:
\begin{equation}
\begin{split}
& \E[ \tilde{d}_\eta( Z_{k+1}, Z'_{k+1} ) ] \\
& \leq (1 - \bar{\mu}_g \alpha / 2) \Delta_k^{\theta} + 2 \alpha L_1^2 \bigl( \tfrac{1}{\bar{\mu}_g} + \alpha \bigr) \Delta_{k+1}^{X} + \eta \bigl( (1 - \tfrac{\rho}{2} ) \Delta_k^X + \tfrac{1}{\rho} (2-\rho) L_P \Delta_k^{\theta} \bigr) \\
& \leq \bigl( 1 - \bar{\mu}_g \alpha / 2 + \eta \tfrac{1}{\rho} (2-\rho) L_P + 2 \alpha L_1^2 \bigl( \tfrac{1}{\bar{\mu}_g} + \alpha \bigr) \tfrac{1}{\rho} (2-\rho) L_P \bigr) \Delta_k^{\theta} \\
& + \eta (1 - \tfrac{\rho}{2} ) \bigl( 2 \alpha L_1^2 ( \tfrac{1}{\bar{\mu}_g} + \alpha ) / \eta + 1 \bigr) \Delta_k^X \eqsp.
\end{split}
\end{equation}
Provided that: 
\[
\eta = \frac{ \bar{\mu}_g }{ 8 } \frac{\rho}{ L_P (2-\rho) } \, \alpha, \quad L_P \leq \frac{ \bar{\mu}_g^2 }{ 128 L_1^2 } \frac{\rho^2}{2-\rho}, 
\]
Thus, with $\tau(\alpha) := \min\{ \bar{\mu}_g \alpha / 8, \rho/4 \}$, we have the geometric contraction:
\[
\E[ \tilde{d}_\eta( Z_{k+1}, Z_{k+1}' )] \leq (1-\tau(\alpha)) \E[ \tilde{d}_\eta ( Z_k, Z_k' ) ] \leq (1-\tau(\alpha))^k \tilde{d}_\eta ( z,z' ) \eqsp.
\]
With $d(z,z') := (\|\theta-\theta'\|^2 + d_{\mathbb{X}} ( x, x' )^2 )^{1/2}$, and noting that for any $z = (\theta,x), z' = (\theta', x')$,
\[
\bigl( \E[ \tilde{d}_\eta( Z_{k}, Z_{k}' )] \bigr)^{1/2} \geq 
\min\{1, \eta^{1/2} \}
\E[ d(Z_k, Z_k')^2 ]^{1/2}  \eqsp.
\]
This implies the weak convergence of the joint process $\{ Z_k \}_{k \in \nset}$ to a unique stationary distribution $\upsilon_\alpha$ at a geometric rate under the  Wasserstein $2$-distance induced by $d(\cdot,\cdot)$. In particular, taking the infimum over coupling between $Z_k, Z_k'$ shows that for any initial distribution $\mu$ of $Z_0$ with $\E_{ \mu } [ \|Z_0\|^2 ] < \infty$, there exists a constant $C > 0$ such that
\begin{equation}
    \label{eq:TV_vs_W2_app}
    W_2( \mu Q_\alpha^k , \upsilon_\alpha ) \leq C (1-\tau(\alpha))^{k/2} \eqsp.
\end{equation}
This completes the proof.

\subsection{Proof of Corollary~\ref{cor:moment_convergence}} \label{app:moment_convergence}
Let $\mu_k:=\mathcal{L}(Z_k)$ and $\bar v^{(\alpha)}$ be the invariant law from Theorem~\ref{thm:weak_convergence}.
Since $h$ is Lipschitz on $\Theta_0$:
there exists $L_{h}<\infty$ such that
\[
|h(\theta)-h(\theta')|\le L_{h}\|\theta-\theta'\|,
\qquad \forall \theta,\theta'\in\Theta_0.
\]
Let $\gamma$ be any coupling of $\mu_k$ and $\bar v^{(\alpha)}$.
Writing $(Z,Z')=((\theta,x),(\theta',x'))$ under $\gamma$ and using $\|\theta-\theta'\|\le d(Z,Z')$,
\[
\big|\mathbb{E}_{\mu_k}[h(\theta)]-\mathbb{E}_{\bar v^{(\alpha)}}[h(\theta)]\big|
=
\big|\mathbb{E}_\gamma[h(\theta)-h(\theta')]\big|
\le
L_{h,n}\,\mathbb{E}_\gamma\|\theta-\theta'\|
\le
L_{h,n}\,\mathbb{E}_\gamma d(Z,Z').
\]
Taking the infimum over couplings gives
\[
\big|\mathbb{E}[ h(\theta_k) ]-\mathbb{E}_{v^{(\alpha)}}[ h(\theta) ]\big|
\le
L_{h}\,W_1(\mu_k,\bar{v}^{(\alpha)}).
\]
Applying Theorem~\ref{thm:weak_convergence} and $W_1 \leq W_2$ completes the proof.

\subsection{Proof of Theorem~\ref{thm:clt}} \label{app:clt}

\paragraph{Step 1: Poisson equation and Martingale decomposition.}
Define the centered measurable function
\[
F(z):= h(\theta) -\bar{h},\qquad z=(\theta,x) ,
\]
where $\bar h:=\mathbb{E}_{v^{(\alpha)}}[ h(\theta)]$.
Under stationarity $\bar{v}^{(\alpha)}$, we have $\mathbb{E}_{\bar{v}^{(\alpha)}}[F]=0$.
Next, we define the Poisson solution by the convergent series
\begin{equation}
\label{eq:poisson_solution_app}
\psi(z):=\sum_{k=0}^{\infty} Q_\alpha^k F(z).
\end{equation}
Note that the series converge as $h(\cdot)$ is Lipschitz continuous since 
\[
\| Q_\alpha^k F(\theta) \| = \| Q_\alpha^k \big( h( \theta ) - \bar{h} \big) \| \leq L_h W_1( Q_\alpha^k( z, \cdot ), \bar{v}^{(\alpha)} ) \leq C L_h (1-\tau(\alpha))^{k/2} \eqsp.
\]

Moreover, by the construction, we have
\[
(I-Q_\alpha)\psi = F.
\]
Now, we write
\[
F(Z_k)=\psi(Z_k)-Q_\alpha\psi(Z_k)
=
\psi(Z_k)-\mathbb{E}[\psi(Z_{k+1})\mid \mathfrak{F}_k].
\]
Summing from $k=0$ to $n-1$ and adding/subtracting $\psi(Z_{k+1})$ yields the standard decomposition
\begin{align}
\label{eq:martingale_decomp_app}
\sum_{k=0}^{n-1}F(Z_k)
&=
\sum_{k=0}^{n-1}\Big(\psi(Z_{k+1})-\mathbb{E}[\psi(Z_{k+1})\mid \mathfrak{F}_k]\Big)
+\psi(Z_0)-\psi(Z_n) =: M_n + R_n.
\end{align}
Here $\{M_n\}$ is a square-integrable martingale with respect to $\{\mathfrak{F}_n\}$, with increments
\[
\Delta M_{k+1}:=M_{k+1}-M_k=\psi(Z_{k+1})-\mathbb{E}[\psi(Z_{k+1})\mid \mathfrak{F}_k],
\]
and $R_n:=\psi(Z_0)-\psi(Z_n)$ is a remainder term.
By the Lipschitzness property of $h$, we have $\|R_n\|/\sqrt{n}\to0$ almost surely.

\paragraph{Step 2: Martingale CLT.}
As $\psi$ is bounded, the increments $\Delta M_{k+1}$ are uniformly bounded. We
let
\[
V_n:=\sum_{k=0}^{n-1}\mathbb{E}\big[\Delta M_{k+1}\Delta M_{k+1}^\top\mid \mathfrak{F}_k\big].
\]
By the ergodic theorem for geometrically ergodic Markov chains and bounded measurable functions,
\[
\frac{1}{n}V_n \xrightarrow{a.s.} \Sigma(\alpha)
:=
\mathbb{E}_{\bar{v}^{(\alpha)}}\Big[\mathbb{E}\big[\Delta M_1\Delta M_1^\top\mid Z_0\big]\Big].
\]
Therefore, by the martingale CLT \citep{hall2014martingale},  
\[
\frac{1}{\sqrt{n}}M_n \xrightarrow{d}\mathcal{N}(0,\Sigma(\alpha)).
\]
Since $R_n/\sqrt{n}\to0$ in probability, \eqref{eq:martingale_decomp_app} yields
\[
\frac{1}{\sqrt{n}}\sum_{k=0}^{n-1}F(Z_k)
=
\frac{1}{\sqrt{n}}\sum_{k=0}^{n-1}( h(\theta_k) -\bar h)
\xrightarrow{d}\mathcal{N}(0,\Sigma(\alpha)).
\]

Finally, by Theorem~\ref{thm:weak_convergence}, we observe that the covariance series
\[
\sum_{k=1}^{\infty}\Big(\mathrm{Cov}_{\bar{v}^{(\alpha)}}(h(\theta_0),h(\theta_k))+\mathrm{Cov}_{\bar{v}^{(\alpha)}}(h(\theta_k),h(\theta_0))\Big)
\]
converges absolutely.
Moreover, one can verify by expanding $M_n$ in \eqref{eq:martingale_decomp_app} under stationarity and taking limits
that the martingale variance $\Sigma(\alpha)$ equals:
\[
\Sigma(\alpha)
=
\mathrm{Var}_{\bar{v}^{(\alpha)}}( h(\theta_0))
+
\sum_{k=1}^\infty
\Big(
\mathrm{Cov}_{\bar{v}^{(\alpha)}}(h(\theta_0),h(\theta_k))
+
\mathrm{Cov}_{\bar{v}^{(\alpha)}}(h(\theta_k),h(\theta_0))
\Big),
\]
which is precisely \eqref{eq:green_kubo} in the main text.
This completes the proof.


\section{Asymptotic Bias Characterization}
\label{app:bias}
In this appendix we provide a detailed proof of Theorem~\ref{thm:bias_simple},
which characterizes the $\mathcal{O}(\alpha)$ stationary bias of the
decision-dependent SA recursion.
Throughout we work under Assumptions~\ref{ass:g_regularity}--\ref{ass:markov}
and~\ref{ass:stationary}--\ref{ass:wdstar}, and we repeatedly use the finite-time moment bounds
established in Appendix~\ref{app:finite_time}.

\subsection{Preliminaries and Stationary Moments}
\label{app:bias_prelim}

Throughout this section, we work under Assumption~\ref{ass:stationary}, which guarantees that the joint
Markov chain $(\theta_k,X_k)_{k\ge0}$ admits a stationary distribution.
With a slight abuse of notation, in this section we denote by $(\theta_k,X_k)_{k\ge0}$ a stationary version of the chain, so that
\[
(\theta_k,X_k)\stackrel{d}{=}(\theta_0,X_0),
\qquad k\ge0 .
\]

Recall that $\theta^*$ denotes the unique root of the mean field,
$\bar g(\theta^*)=0$, which exists and is unique by the strong monotonicity assumption
(Assumption~\ref{ass:stability}).
We define the centered error process
\[
\Delta_k := \theta_k - \theta^* .
\]
By stationarity, the distribution of $\Delta_k$ does not depend on $k$. We define the stationary error random variable
\[
\Delta_\infty \;\stackrel{d}{=}\; \theta_0 - \theta^* .
\]
Equivalently, $\Delta_\infty$ denotes a generic random variable with the stationary
distribution of $\Delta_k := \theta_k - \theta^*$.

We also recall the Jacobian of the mean field at the solution,
\[
J^* := \nabla \bar g(\theta^*),
\]
which is nonsingular under Assumption~\ref{ass:stability}.

The finite-time analysis in Appendix~\ref{app:finite_time} provides uniform moment bounds
for the error process.
The following lemma records the corresponding implication in the stationary regime.

\begin{lemma}[Stationary moment bounds]
\label{lem:stationary_moment_bounds}
Let $p \ge 1$ be an integer such that the assumptions of Appendix~\ref{app:finite_time}
ensure finite $2p$-th moments for the error process.
Then there exists a constant $C_p < \infty$ such that, for the stationary chain
$(\theta_k,X_k)$,
\begin{equation}
\label{eq:stationary_moment_bound}
\mathbb{E}\bigl[\|\theta_0-\theta^*\|^{2p}\bigr]
=
\mathbb{E}\bigl[\|\Delta_0\|^{2p}\bigr]
\;\le\;
C_p\,\alpha^{p}.
\end{equation}
In particular, for $p=1$,
\[
\mathbb{E}\bigl[\|\theta_0-\theta^*\|^{2}\bigr]
\;\le\;
C_1\,\alpha .
\]
\end{lemma}

\begin{proof}
By Theorem~\ref{thm:moments} in Appendix~\ref{app:finite_time} with $n=p$,
there exist constants $C_{p,1}, C_{p,2} < \infty$ and $\mu>0$ such that, for all $k\ge0$,
\begin{equation}
\label{eq:finite_time_moment_bound}
\mathbb{E}\bigl[\|\theta_k-\theta^*\|^{2p}\bigr]
\;\le\;
C_{p,2}\,\alpha^{p}
+
C_{p,1}\,(1-\mu\alpha)^k\,
\mathbb{E}\bigl[\|\theta_0-\theta^*\|^{2p}\bigr].
\end{equation}
Now consider the stationary initialization guaranteed by Assumption~4.1.
By stationarity,
\[
\mathbb{E}\bigl[\|\theta_k-\theta^*\|^{2p}\bigr]
=
\mathbb{E}\bigl[\|\theta_0-\theta^*\|^{2p}\bigr],
\qquad k\ge0 .
\]
Substituting this identity into \eqref{eq:finite_time_moment_bound} yields
\[
\mathbb{E}\bigl[\|\theta_0-\theta^*\|^{2p}\bigr]
\;\le\;
C_{p,2}\,\alpha^{p}
+
C_{p,1}\,(1-\mu\alpha)^k\,
\mathbb{E}\bigl[\|\theta_0-\theta^*\|^{2p}\bigr],
\qquad k\ge0 .
\]
Letting $k\to\infty$ and using $(1-\mu\alpha)^k\to0$, we obtain
\[
\mathbb{E}\bigl[\|\theta_0-\theta^*\|^{2p}\bigr]
\;\le\;
C_{p,2}\,\alpha^{p},
\]
which proves \eqref{eq:stationary_moment_bound}.
\end{proof}

Our next step is to derive the second-order expansion of the stationarity
equation [cf.~\eqref{eq:bias_balance} in the main paper] that underlies the bias characterization.
By stationarity of $(\theta_0, X_0)$ and the update rule, we have
\begin{equation}
\label{eq:stationarity_g}
\mathbb{E}\bigl[g(\theta_0,X_1)\bigr] = 0.
\end{equation}
Throughout we
drop the time index to simplify notation, writing
\[
(\theta,X,X') := (\theta_0,X_0,X_1),
\qquad
\Delta := \theta - \theta^* .
\]
With this shorthand, \eqref{eq:stationarity_g} becomes
\begin{equation}
\label{eq:stationarity_g_short}
\mathbb{E}\bigl[g(\theta,X')\bigr] = 0.
\end{equation}

We now expand $g(\theta,X')$ around $\theta^*$ using a Taylor expansion in the
parameter argument.
By Assumption~2.1, for each fixed $x\in \mathbb{X}$ the map
$\theta\mapsto g(\theta,x)$ is three times continuously differentiable.
Moreover, there exist constants $L_1,L_3<\infty$ such that for all
$x\in \mathbb{X}$ and all $\theta\in\mathbb{R}^d$,
\[
\|g(\theta,x)\| \le L_1(1+\|\theta - \theta^*\|),\qquad
\|g'(\theta,x)\| \le L_1(1+\|\theta - \theta^*\|),\qquad
\|g''(\theta,x)\| \le L_1(1+\|\theta - \theta^*\|),
\]
and, in addition,
\[
\|g^{(3)}(\theta,x)\| \le L_3.
\]
In particular, the quantities $g(\theta^*,x)$, $g'(\theta^*,x)$, and
$g''(\theta^*,x)$ are finite and uniformly bounded in $x$.

For each fixed realization of $(\theta,X')$, Taylor's theorem with integral
remainder yields
\begin{equation}
\label{eq:taylor_pointwise}
g(\theta,X')
=
g(\theta^*,X')
+ g'(\theta^*,X')\,\Delta
+ \frac{1}{2}\,g''(\theta^*,X')[\Delta,\Delta]
+ R^{(3)}(\Delta,X'),
\end{equation}
where the third-order remainder $R^{(3)}(\Delta,X')$ satisfies the pointwise
bound
\begin{equation}
\label{eq:R3_pointwise_bound}
\|R^{(3)}(\Delta,X')\|
\le
\frac{1}{6}\,
\sup_{0\le t\le 1}\bigl\|g^{(3)}(\theta^*+t\Delta,X')\bigr\|\,
\|\Delta\|^3
\;\le\;
C_g\,\|\Delta\|^3,
\end{equation}
with $C_g := L_3/6$. Here $g''(\theta^*,X')[\Delta,\Delta]$ denotes the application of the bilinear
map $g''(\theta^*,X')$ to the pair $(\Delta,\Delta)$.

Taking expectations in~\eqref{eq:taylor_pointwise} and using
\eqref{eq:stationarity_g_short} yields
\begin{equation}
\label{eq:master_equation_terms}
0
=
\underbrace{\mathbb{E}\bigl[g(\theta^*,X')\bigr]}_{\mathrm{(I)}}
+
\underbrace{\mathbb{E}\bigl[g'(\theta^*,X')\,\Delta\bigr]}_{\mathrm{(II)}}
+
\underbrace{\frac{1}{2}\,\mathbb{E}\bigl[g''(\theta^*,X')[\Delta,\Delta]\bigr]}_{\mathrm{(III)}}
+
\underbrace{\mathbb{E}\bigl[R^{(3)}(\Delta,X')\bigr]}_{\mathrm{(IV)}}.
\end{equation}
We first show that the third-order remainder $\mathrm{(IV)}$ is
negligible at order $\alpha$.
\begin{lemma}[Third-order remainder bound]
\label{lem:third_order_remainder}
Under Assumptions~\ref{ass:g_regularity}- \ref{ass:markov} and \ref{ass:stationary}-\ref{ass:wdstar}, the third-order remainder term
$\mathrm{(IV)} := \mathbb{E}\bigl[R^{(3)}(\Delta_\infty, X')\bigr]$ satisfies
\begin{equation}
\label{eq:R3_expectation_bound}
\bigl\|\mathrm{(IV)}\bigr\|
\;\le\;
C\,\mathbb{E}\bigl[\|\Delta_\infty\|^3\bigr]
\;=\;
\mathcal{O}(\alpha^{3/2}).
\end{equation}
\end{lemma}

\begin{proof}
By the pointwise remainder bound~\eqref{eq:R3_pointwise_bound},
\[
\bigl\|R^{(3)}(\Delta_\infty, X')\bigr\|
\;\le\;
C_g\,\|\Delta_\infty\|^3 .
\]
Taking expectations and using Jensen's inequality,
\[
\bigl\|\mathrm{(IV)}\bigr\|
=
\bigl\|\mathbb{E}[R^{(3)}(\Delta_\infty, X')]\bigr\|
\le
\mathbb{E}\bigl[\|R^{(3)}(\Delta_\infty, X')\|\bigr]
\le
C_g\,\mathbb{E}\bigl[\|\Delta_\infty\|^3\bigr].
\]
By Hölder's inequality and Lemma~\ref{lem:stationary_moment_bounds} (with $p=2$),
\[
\mathbb{E}\bigl[\|\Delta_\infty\|^3\bigr]
\le
\bigl(\mathbb{E}\bigl[\|\Delta_\infty\|^4\bigr]\bigr)^{3/4} \leq \left( C_2 \alpha^2 \right)^{3/4} = \mathcal{O}( \alpha^{3/2} ).
\]
which proves~\eqref{eq:R3_expectation_bound}.
\end{proof}

The master equation \eqref{eq:master_equation_terms} expresses the stationary identity
$\mathbb{E}[g(\theta,X')]=0$ as a second-order expansion around $\theta^*$ with
the four contributions $\mathrm{(I)}$--$\mathrm{(IV)}$.
Lemma~\ref{lem:third_order_remainder} shows that the remainder term
$\mathrm{(IV)}$ is of order $\mathcal{O}(\alpha^{3/2})$.
In the subsequent analysis, we focus on the leading contributions
$\mathrm{(I)}$--$\mathrm{(III)}$ in order to obtain an explicit first-order
representation of the stationary bias $\mathbb{E}[ \Delta ]$.

\subsection{Decision-Dependent Contribution via the Poisson Equation (Term (I))}
\label{app:bias_poisson}

We now analyze the first term in~\eqref{eq:master_equation_terms},
\[
\mathrm{(I)}
=
\mathbb{E}\bigl[g(\theta^*,X_1)\bigr],
\]
which is nonzero in general because the distribution of $X_1$ depends on the
current parameter $\theta_0$.
The key tool is the Poisson equation associated with the fixed parameter
$\theta^*$.
Throughout this subsection we shall use the shorthand
\[
(\theta, X, X') := (\theta_0, X_0, X_1), \qquad \Delta := \theta - \theta^* .
\]

For each $\theta$, we denote by $\pi_\theta$ the stationary distribution of the Markov kernel $P_\theta$, and that the mean field is
$\bar g(\theta) = \mathbb{E}_{X\sim\pi_\theta}[g(\theta,X)]$. Since $\bar g(\theta^*)=0$, there exists a (centered) solution
$\hat g(\theta^*,\cdot)$ to the Poisson equation
\begin{equation}
\label{eq:poisson_appendix}
(I - P_{\theta^*})\hat g(\theta^*,x)
=
g(\theta^*,x) - \bar g(\theta^*)
=
g(\theta^*,x),
\qquad
\mathbb{E}_{X\sim\pi_{\theta^*}}
\bigl[\hat g(\theta^*,X)\bigr]
=
0.
\end{equation}
We next observe that the quantity $\mathrm{(I)}$ can be expressed solely in terms of the
difference between the kernels $P_{\theta_0}$ and $P_{\theta^*}$.
By the Poisson equation~\eqref{eq:poisson_appendix},
\[
g(\theta^*,X')
=
\hat g(\theta^*,X')
-
P_{\theta^*}\hat g(\theta^*,X').
\]
Taking expectations gives
\begin{equation}
\label{eq:I_poisson_step1}
\mathrm{(I)}
=
\mathbb{E}\bigl[\hat g(\theta^*,X')\bigr]
-
\mathbb{E}\bigl[P_{\theta^*}\hat g(\theta^*,X')\bigr].
\end{equation}

We now rewrite both terms on the right-hand side in terms of $(\theta,X)$.
First, use the tower property and the Markov structure.
Conditioning on $(\theta,X)$,
\[
\mathbb{E}\bigl[\hat g(\theta^*,X')\mid \theta,X\bigr]
=
P_{\theta}\hat g(\theta^*,X).
\]
Therefore,
\begin{equation}
\label{eq:expect_Ptheta}
\mathbb{E}\bigl[\hat g(\theta^*,X')\bigr]
=
\mathbb{E}\bigl[P_{\theta}\hat g(\theta^*,X)\bigr].
\end{equation}

Next, consider the second term in~\eqref{eq:I_poisson_step1}.
Since $(\theta_k,X_k)$ is stationary, the marginal distribution of $X'$ is
identical to that of $X$, and hence for any bounded measurable function $f$,
$\mathbb{E}[f(X')] = \mathbb{E}[f(X)]$.
Applying this with $f(x)=P_{\theta^*}\hat g(\theta^*,x)$ gives
\begin{equation}
\label{eq:expect_Pthetastar}
\mathbb{E}\bigl[P_{\theta^*}\hat g(\theta^*,X')\bigr]
=
\mathbb{E}\bigl[P_{\theta^*}\hat g(\theta^*,X)\bigr].
\end{equation}

Substituting~\eqref{eq:expect_Ptheta}
and~\eqref{eq:expect_Pthetastar}
into~\eqref{eq:I_poisson_step1} yields
\begin{equation}\label{eq:I_poisson_identity}
\mathrm{(I)}
=
\mathbb{E}\bigl[
(P_{\theta}-P_{\theta^*})\hat g(\theta^*,X)
\bigr] \eqsp.
\end{equation}
The identity~\eqref{eq:I_poisson_identity} shows that the sole source of bias
from Term~$\mathrm{(I)}$ is the dependence of the transition kernel on the
current parameter, i.e., the deviation of $P_{\theta_0}$ from $P_{\theta^*}$.
In the next step we apply Assumption~\ref{ass:wdstar} to obtain a localized
linear expansion of this difference in terms of  $\Delta$.

\subsubsection{Linearization under $(\textsf{WD}^*)$}
\label{app:bias_wd_star}
We now exploit Assumption~\ref{ass:wdstar} to obtain a first-order expansion of
the operator
$\theta \mapsto P_\theta \hat g(\theta^*,\cdot)$
around $\theta^*$ when applied to the random state $X_0$.
In particular, there exist a neighborhood $\mathcal{N}'\subset\mathcal{N}$ of $\theta^*$ and a
constant $C_{\mathrm{wd}}<\infty$ such that for all $\theta\in\mathcal{N}'$ and
writing $\Delta=\theta-\theta^*$, we have for any $x$, 
\begin{equation}
\label{eq:wd_star_expansion_pointwise}
(P_{\theta}-P_{\theta^*})\hat g(\theta^*,x)
=
\Lambda_{\theta^*}[\Delta](x)
+
R^{\mathrm{WD}}(\theta,x),
\end{equation}
where the remainder satisfies the uniform bound
\begin{equation}
\label{eq:wd_star_remainder_bound}
\sup_{x\in\mathbb{X}} \,
\bigl\|R^{\mathrm{WD}}(\theta,x)\bigr\|
\le
C_{\mathrm{wd}}\,\|\Delta\|^2.
\end{equation}
We define 
\[
\tilde{\Lambda}_*(x) := \Lambda_*(x) - \bar{\Lambda}_*, \quad \text{where} \quad \bar{\Lambda}_* := \mathbb{E}_{ X \sim \pi_{\theta^*} }[ \Lambda_*(X) ]
\]
and observe the following lemma that controls the term (I):
\begin{lemma}\label{lem:termI_localized}
Under Assumption~\ref{ass:wdstar}, it holds that
\begin{equation}
\textnormal{(I)}= \bar{\Lambda}_* \, \mathbb{E}\bigl[\Delta\bigr] 
+ r_I(\alpha),   
\end{equation}
where the remainder satisfies $\| r_I(\alpha) \| \leq C_{\rm I} \, \alpha$.
\end{lemma}

\begin{proof}
Let $E:=\{\|\Delta\|\le r\}$ so that $E\subset\{\theta\in\mathcal N'\}$ and hence $\theta\in\mathcal N$ on $E$.
Split
\[
\textnormal{(I)}
=
\mathbb{E}\bigl[(P_\theta-P_{\theta^*})\hat g(\theta^*,X)\,\mathbf 1_E\bigr]
+
\mathbb{E}\bigl[(P_\theta-P_{\theta^*})\hat g(\theta^*,X)\,\mathbf 1_{E^c}\bigr].
\]
For the first term on $E$, applying Assumption~\ref{ass:wdstar} pointwise at $(\theta,X)$ yields:
\[
(P_\theta-P_{\theta^*})\hat g(\theta^*,X)
=
\bar\Lambda_* \Delta + \tilde{\Lambda}_*[\Delta](X) +R^{\mathrm{WD}}(\theta,X),
\qquad
\|R^{\mathrm{WD}}(\theta,X)\|\le C_{\mathrm{wd}}\|\Delta\|^2.
\]
Therefore,
\[
\mathbb{E}\bigl[(P_\theta-P_{\theta^*})\hat g(\theta^*,X)\,\mathbf 1_E\bigr]
=
\mathbb{E}\bigl[ (\bar\Lambda_* \Delta + \tilde{\Lambda}_*[\Delta](X)) \,\mathbf 1_E\bigr]
+
\mathbb{E}\bigl[R^{\mathrm{WD}}(\theta,X)\,\mathbf 1_E\bigr].
\]
For the second term on $E^c$, use Lemma \ref{lem:poisson_properties} to bound
\[
\|(P_\theta-P_{\theta^*})\hat g(\theta^*,X)\|
\le
\|P_\theta \hat g(\theta^*,\cdot)\|_\infty+\|P_{\theta^*}\hat g(\theta^*,\cdot)\|_\infty
\le 2 L_{PH}^{(0)}  .
\]
Combine the two pieces and define
\[
r_I(\alpha)
:=
\mathbb{E}\bigl[ \tilde{\Lambda}_*[\Delta](X) {\bf 1}_E \bigr]
+
\mathbb{E}\bigl[R^{\mathrm{WD}}(\theta,X)\,\mathbf 1_E\bigr]
+
\mathbb{E}\bigl[(P_\theta-P_{\theta^*})\hat g(\theta^*,X)\,\mathbf 1_{E^c}\bigr]
+
\mathbb{E}\bigl[\Lambda_{\theta^*}[\Delta](X)\,(\mathbf 1_E-1)\bigr].
\]
Then we have $\textnormal{(I)}= \bar{\Lambda}_* \, \Delta + r_I(\alpha)$.

To bound the remainder terms, we first observe that by applying Lemma~\ref{lem:II_fluct} and notice that $\tilde{\Lambda}_*(\cdot)$ satisfies the requirement therein, we obtain  
\[
\left\| \mathbb{E}\bigl[ \tilde{\Lambda}_*[\Delta](X) {\bf 1}_E \bigr] \right\| \leq C_{\Lambda} \, \alpha.
\]
Secondly, we also notice that 
\[
\Bigl\|\mathbb{E}\bigl[R^{\mathrm{WD}}(\theta,X)\,\mathbf 1_E\bigr]\Bigr\|
\le C_{\mathrm{wd}}\mathbb{E}\bigl[\|\Delta\|^2\bigr] \leq C_2 C_{\rm wd} \alpha.
\]
Thirdly, 
\[
\Bigl\|\mathbb{E}\bigl[(P_\theta-P_{\theta^*})\hat g(\theta^*,X)\,\mathbf 1_{E^c}\bigr]\Bigr\|
\le 2\|P_\theta \hat g(\theta^*,\cdot)\|_\infty\,\mathbb{P}(E^c)
+2\|P_{\theta^*} \hat g(\theta^*,\cdot)\|_\infty\,\mathbb{P}(E^c)
\le 2B_{\hat g}\,\mathbb{P}(E^c).
\]
Moreover, since $\|\Lambda_{\theta^*}[\Delta](X)\|\le \|\Lambda_{\theta^*}\|_{\mathrm{op}}\|\Delta\|$ pointwise,
\[
\left\|\mathbb{E}\bigl[\Lambda_{\theta^*}[\Delta](X)(\mathbf 1_E-1)\bigr]\right\|
\le
\|\Lambda_{\theta^*}\|_{\mathrm{op}}\,
\mathbb{E}\bigl[\|\Delta\|\mathbf 1_{E^c}\bigr]
=
\|\Lambda_{\theta^*}\|_{\mathrm{op}}\,
\mathbb{E}\bigl[\|\Delta\|\mathbf 1_{\{\|\Delta\|>r\}}\bigr].
\]
Combining these bounds gives
\[
\|r_I (\alpha) \|
\le
C_{\Lambda} \, \alpha 
+ 
C_{\mathrm{wd}} C_2 \, \alpha
+
2\|\hat g(\theta^*,\cdot)\|_\infty\,\mathbb{P}\bigl(\|\Delta\|>r\bigr)
+
\|\Lambda_{\theta^*}\|_{\mathrm{op}}\,\mathbb{E}\bigl[\|\Delta\|\mathbf 1_{\{\|\Delta\|>r\}}\bigr].
\]

Finally, by Cauchy--Schwarz and Markov's inequality,
\[
\begin{split}
& \mathbb{E}\bigl[\|\Delta\|\mathbf 1_{\{\|\Delta\|>r\}}\bigr]
\le
\sqrt{\mathbb{E}\|\Delta\|^2}\sqrt{\mathbb{P}(\|\Delta\|>r)}
\le \frac{1}{r}\mathbb{E}\|\Delta\|^2 \leq \frac{C_2 \alpha}{r},
\\
& \mathbb{P}(\|\Delta\|>r)\le \frac{1}{r^2}\mathbb{E}\|\Delta\|^2 \leq \frac{C_2 \alpha}{r^2}.
\end{split}
\]
Substituting these bounds yields the conclusion of this lemma.
\end{proof}

\subsection{Linear Term and Bias Equation}
\label{app:bias_linear}

We now analyze the second term in~\eqref{eq:master_equation_terms},
\[
\mathrm{(II)}
=
\mathbb{E}\bigl[g'(\theta^*,X_1)\,\Delta\bigr].
\]
Recall that $
J^* := \nabla\bar g(\theta^*)
=
\mathbb{E}_{X\sim\pi_{\theta^*}}\bigl[g'(\theta^*,X)\bigr]$ and
define the centered Jacobian fluctuation
\[
\tilde J(x)
:=
g'(\theta^*,x) - J^*,
\qquad x\in\mathbb{X}.
\]
Then $g'(\theta^*,x) = J^* + \tilde J(x)$,
and we can decompose $\mathrm{(II)}$ as
\begin{equation}
\label{eq:II_decomposition}
\mathrm{(II)}
=
\mathbb{E}\bigl[g'(\theta^*,X_1)\Delta\bigr]
=
J^*\,\mathbb{E}\bigl[\Delta\bigr]
+
\mathbb{E}\bigl[\tilde J(X_1)\Delta\bigr].
\end{equation}
The first term on the right-hand side is the main linear restoring term.
The second term captures correlations between the parameter error and
fluctuations of the local Jacobian; it contributes to the decision-dependent
part of the bias.

With the above setup, we can apply Lemma~\ref{lem:II_fluct} which shows that there exists $C_{\rm II}$ such that,
\begin{equation}
\label{eq:II_expanded}
\mathrm{(II)}
=
J^*\,\mathbb{E}\bigl[\Delta\bigr]
+
r_{\mathrm{II}}(\alpha),
\qquad
\|r_{\mathrm{II}}(\alpha)\|\le C_{\rm II}\,\alpha.
\end{equation}

\subsection{Nonlinear Hessian Contribution}
\label{app:bias_hessian}
We now turn our focus to the third term in~\eqref{eq:master_equation_terms} with the aim to show that
\begin{equation}
\label{eq:III_target}
\mathrm{(III)}
=
\frac{1}{2}\,\mathbb{E}\bigl[g''(\theta^*,X_1)[\Delta_\infty,\Delta_\infty]\bigr]
=
\alpha\,\frac{1}{2}\,\nabla^2\bar g(\theta^*)[M]
+
\mathcal{O}(\alpha^{1+\varepsilon}),
\end{equation}
where $M$ is the limit matrix from Assumption~\ref{ass:covariance}, and
$\nabla^2\bar g(\theta^*)[M]$ denotes the contraction of the Hessian of the mean
field with the matrix $M$. Here $\varepsilon>0$ is implicit and may be taken uniformly for $\alpha$ sufficiently small. Recall that
\[
\bar g(\theta)
=
\mathbb{E}_{X\sim\pi_\theta}[g(\theta,X)],
\]
and thus
\[
\nabla^2\bar g(\theta^*)
=
\mathbb{E}_{X\sim\pi_{\theta^*}}[g''(\theta^*,X)].
\]
We begin by decomposing $\mathrm{(III)}$ into a term involving the averaged
Hessian $\nabla^2\bar g(\theta^*)$ and a fluctuation term.

\subsubsection{Decomposition around the averaged Hessian}

Define the centered Hessian fluctuation
\[
D''(x)
:=
g''(\theta^*,x)
-
\mathbb{E}_{X\sim\pi_{\theta^*}}[g''(\theta^*,X)]
=
g''(\theta^*,x) - \nabla^2\bar g(\theta^*),
\qquad x\in\mathbb{X}.
\]
By construction,
\[
\mathbb{E}_{X\sim\pi_{\theta^*}}\bigl[D''(X)\bigr] = 0.
\]
Then
\[
g''(\theta^*,x)
=
\nabla^2\bar g(\theta^*) + D''(x),
\]
and we may write
\begin{align}
\mathrm{(III)}
&=
\frac{1}{2}\,
\mathbb{E}\bigl[g''(\theta^*,X_1)[\Delta_\infty,\Delta_\infty]\bigr] \notag \\
& =
\frac{1}{2}\,
\mathbb{E}\bigl[\nabla^2\bar g(\theta^*)[\Delta_\infty,\Delta_\infty]\bigr]
+
\frac{1}{2}\,
\mathbb{E}\bigl[D''(X_1)[\Delta_\infty,\Delta_\infty]\bigr].
\label{eq:III_decomposition}
\end{align}
We analyze these two contributions separately.
The first term in~\eqref{eq:III_decomposition} is identified as follows.

\begin{lemma}[Leading Hessian term]
\label{lem:III_app}
Under the assumptions of Lemma~\ref{lem:third_order_remainder}, we have
\begin{equation}
\label{eq:III_app}
\frac{1}{2}\,
\mathbb{E}\bigl[\nabla^2\bar g(\theta^*)[\Delta_\infty,\Delta_\infty]\bigr]
=
\alpha\,\frac{1}{2}\,\nabla^2\bar g(\theta^*)[M]
+
\mathcal{O}(\alpha^{1+\varepsilon}).
\end{equation}
\end{lemma}

\begin{proof}
Since $\nabla^2\bar g(\theta^*)$ is a fixed bounded bilinear map, we can pull
it outside the expectation:
\[
\mathbb{E}\bigl[\nabla^2\bar g(\theta^*)[\Delta_\infty,\Delta_\infty]\bigr]
=
\nabla^2\bar g(\theta^*)\bigl[\mathbb{E}[\Delta_\infty^{\otimes2}]\bigr].
\]
By Assumption~\ref{ass:covariance}, we may write
\[
\mathbb{E}[\Delta_\infty^{\otimes2}]
=
\alpha M + \mathcal{O}(\alpha^{1+\varepsilon}),
\]
for some $\varepsilon>0$ (uniformly for $\alpha$ sufficiently small).
Using the linearity and continuity of the map
$K\mapsto\nabla^2\bar g(\theta^*)[K]$ in $K\in\mathbb{R}^{d\times d}$, we obtain
\[
\nabla^2\bar g(\theta^*)\bigl[\mathbb{E}[\Delta_\infty^{\otimes2}]\bigr]
=
\nabla^2\bar g(\theta^*)[\alpha M + \mathcal{O}(\alpha^{1+\varepsilon})]
=
\alpha\,\nabla^2\bar g(\theta^*)[M] + \mathcal{O}(\alpha^{1+\varepsilon}).
\]
Multiplying by $1/2$ yields~\eqref{eq:III_main}.
\end{proof}

It remains to show that the fluctuation term in~\eqref{eq:III_decomposition},
\[
\mathrm{(III)}_{\mathrm{fluct}}
:=
\frac{1}{2}\,
\mathbb{E}\bigl[D''(X_1)[\Delta_\infty,\Delta_\infty]\bigr],
\]
is negligible at order $\alpha$, i.e., it admits a bound of order
$\mathcal{O}(\alpha^{1+\varepsilon})$ for some $\varepsilon>0$.
The term captures the interaction between state mixing and parameter fluctuations.

\subsubsection{Decay of the Hessian Markov mismatch}

We now show that the fluctuation term $\mathrm{(III)}_{\mathrm{fluct}}$ is
$\mathcal{O}(\alpha^{1+\varepsilon})$.
The argument proceeds by a coupling/mixing-time decomposition, using the
geometric ergodicity of the Markov kernels and the moment bounds from
Lemma~\ref{lem:stationary_moment_bounds}.
For the purposes of this subsection it is convenient to work with a generic
stationary index $k\in\mathbb{Z}$ and write
\[
\Delta_k := \theta_k-\theta^*,
\qquad
\Delta_{k-\tau} := \theta_{k-\tau}-\theta^*,
\qquad
X_{k+1}.
\]
By stationarity, the joint distribution of any finite collection of indices
is invariant under time shifts, so we lose no generality by working with $k$ in
place of $0$.

Define
\[
\mathrm{(III)}_{\mathrm{fluct}}
=
\frac{1}{2}\,
\mathbb{E}\bigl[D''(X_{k+1})[\Delta_k,\Delta_k]\bigr],
\qquad k\in\mathbb{Z},
\]
and note that this quantity does not depend on $k$ by stationarity.

We first decompose the error $\Delta_k$ into a ``lagged'' part and a
recent increment accumulated over a window of length $\tau$, where $\tau$ is a
lag that will be chosen as a function of $\alpha$.
Fix an integer $\tau\ge1$, to be specified later, and write
\begin{equation}
\label{eq:delta_decomposition}
\Delta_k
=
\Delta_{k-\tau}
+
\sum_{j=0}^{\tau-1}
\bigl(\Delta_{k-j}-\Delta_{k-j-1}\bigr).
\end{equation}
For brevity define
\[
S_{k,\tau}
:=
\sum_{j=0}^{\tau-1}
\bigl(\Delta_{k-j}-\Delta_{k-j-1}\bigr),
\]
so that $\Delta_k = \Delta_{k-\tau} + S_{k,\tau}$.

Using~\eqref{eq:delta_decomposition}, we expand the tensor
$\Delta_k^{\otimes2}$ as
\begin{equation}
\label{eq:delta_square_decomposition}
\Delta_k^{\otimes2}
=
\Delta_{k-\tau}^{\otimes2}
+
\Delta_{k-\tau}\otimes S_{k,\tau}
+
S_{k,\tau}\otimes\Delta_{k-\tau}
+
S_{k,\tau}^{\otimes2}.
\end{equation}
Substituting~\eqref{eq:delta_square_decomposition} into the definition of
$\mathrm{(III)}_{\mathrm{fluct}}$ and using linearity, we obtain a
decomposition
\begin{equation}
\label{eq:III_fluct_T_decomposition}
\mathrm{(III)}_{\mathrm{fluct}}
=
\frac{1}{2}\,T^{(0)}_\tau(\alpha)
+
\frac{1}{2}\,T^{(1)}_\tau(\alpha)
+
\frac{1}{2}\,T^{(2)}_\tau(\alpha),
\end{equation}
where
\begin{align}
T^{(0)}_\tau(\alpha)
&:=
\mathbb{E}\bigl[D''(X_{k+1})[\Delta_{k-\tau},\Delta_{k-\tau}]\bigr],
\label{eq:T0_def}\\
T^{(1)}_\tau(\alpha)
&:=
\mathbb{E}\bigl[D''(X_{k+1})\bigl[
\Delta_{k-\tau},S_{k,\tau}
\bigr]\bigr]
+
\mathbb{E}\bigl[D''(X_{k+1})\bigl[
S_{k,\tau},\Delta_{k-\tau}
\bigr]\bigr],
\label{eq:T1_def}\\
T^{(2)}_\tau(\alpha)
&:=
\mathbb{E}\bigl[D''(X_{k+1})[
S_{k,\tau},S_{k,\tau}]\bigr].
\label{eq:T2_def}
\end{align}
We now bound each term separately and then choose $\tau=\tau(\alpha)$.

\paragraph{Bound on $T^{(0)}_\tau(\alpha)$.}
By definition of $D''$ and Assumption~\ref{ass:g_regularity}, there exists
$C_{D''}<\infty$ such that
\[
\|D''(x)\|_{\mathrm{op}} \le C_{D''},
\qquad x\in\mathbb{X},
\]
where $\|\cdot\|_{\mathrm{op}}$ denotes the operator norm of the bilinear map.
Hence
\[
\bigl\|D''(X_{k+1})[\Delta_{k-\tau},\Delta_{k-\tau}]\bigr\|
\le
C_{D''}\,\|\Delta_{k-\tau}\|^2,
\]
and taking expectations yields the crude bound with Lemma~\ref{lem:stationary_moment_bounds} ($p=1$)
\[
\bigl\|T^{(0)}_\tau(\alpha)\bigr\|
\le
C_{D''}\,
\mathbb{E}\bigl[\|\Delta_{k-\tau}\|^2\bigr]
\le
C_{D''} C\,\alpha.
\]

To exploit the centering of $D''$ under $\pi_{\theta^*}$, we use the mixing
properties of the Markov chain under Assumption~\ref{ass:markov}.
Specifically, it implies that there exist constants
$C_{\mathrm{mix}}<\infty$ such that, 
\begin{equation}
\label{eq:geom_mixing}
\sup_{x\in\mathbb{X}} \,
\Bigl\|
\mathbb{E}\bigl[D''(X_{k+1}) \,\big|\, \theta_{k-\tau}=\theta,
X_{k-\tau}=x\bigr]
\Bigr\|
\le
C_{\mathrm{mix}} \big( (1-\rho)^\tau + \| \theta - \theta^* \| \big),
\quad \tau\ge1.
\end{equation}
Indeed, the inner conditional expectation is the expectation of $D''$ under
$P_{\theta}^{\tau+1}(x,\cdot)$, and $D''$ has zero mean under $\pi_{\theta^*}$.
Under the continuity and 1-step contraction assumptions in
Assumption~\ref{ass:markov}, the distance between $P_{\theta}^{m}(x,\cdot)$
and $\pi_{\theta}$ decays exponentially in $m$, 
yielding~\eqref{eq:geom_mixing}.

Conditioning on the $\sigma$-algebra
$\mathcal{F}_{k-\tau} := \sigma(\theta_j,X_j: j\le k-\tau)$
and using~\eqref{eq:geom_mixing}, we obtain
\begin{align*}
\bigl\|T^{(0)}_\tau(\alpha)\bigr\|
&=
\Bigl\|
\mathbb{E}\bigl[
D''(X_{k+1})[\Delta_{k-\tau},\Delta_{k-\tau}]
\bigr]
\Bigr\| \\
&=
\Bigl\|
\mathbb{E}\Bigl[
\mathbb{E}\bigl[D''(X_{k+1}) \mid \mathcal{F}_{k-\tau}\bigr]
\bigl[\Delta_{k-\tau},\Delta_{k-\tau}\bigr]
\Bigr]
\Bigr\| \\
&\le
\mathbb{E}\Bigl[
\bigl\|
\mathbb{E}\bigl[D''(X_{k+1}) \mid \mathcal{F}_{k-\tau}\bigr]
\bigr\|_{\mathrm{op}}
\,
\|\Delta_{k-\tau}\|^2
\Bigr] \\
&\le
C_{\mathrm{mix}} \,
\mathbb{E}\bigl[ (1-\rho)^\tau \|\Delta_{k-\tau}\|^{2} + \| \Delta_{k-\tau} \|^3 \bigr]
\le
C_{\mathrm{mix}}C\,\big( (1-\rho)^\tau\,\alpha + \alpha^{3/2} \big),
\end{align*}
where we used Lemma~\ref{lem:stationary_moment_bounds} (with $p=2$) in the
last step.

\paragraph{Bound on $T^{(1)}_\tau(\alpha)$.}
We now bound the cross term involving $\Delta_{k-\tau}$ and $S_{k,\tau}$.
By symmetry of $D''(x)$ in its arguments and the Cauchy--Schwarz inequality,
\begin{align*}
\bigl\|T^{(1)}_\tau(\alpha)\bigr\|
&\le
2\,\mathbb{E}\bigl[
\|D''(X_{k+1})\|_{\mathrm{op}}\,
\|\Delta_{k-\tau}\|\,
\|S_{k,\tau}\|
\bigr] \\
&\le
2C_{D''}\,
\Bigl(\mathbb{E}\bigl[\|\Delta_{k-\tau}\|^2\bigr]\Bigr)^{1/2}
\Bigl(\mathbb{E}\bigl[\|S_{k,\tau}\|^2\bigr]\Bigr)^{1/2},
\end{align*}
where we used $\|D''(x)\|_{\mathrm{op}}\le C_{D''}$ and Cauchy--Schwarz.
By Lemma~\ref{lem:stationary_moment_bounds} (with $p=1$),
\[
\mathbb{E}\bigl[\|\Delta_{k-\tau}\|^2\bigr] \le C\alpha.
\]
On the other hand, $S_{k,\tau}$ is the sum of $\tau$ consecutive increments of
the form
\[
\Delta_{j+1}-\Delta_{j}
=
\theta_{j+1}-\theta_{j}
=
\alpha\bigl(g(\theta_j,X_{j+1})
+\xi_{j+1}\bigr),
\]
therefore,
\[
S_{k,\tau}
=
\alpha\sum_{j=0}^{\tau-1}
\bigl(g(\theta_{k-j-1},X_{k-j})
+\xi_{k-j}\bigr).
\]
Using Assumption~\ref{ass:g_regularity} (bounds for $g$ at $\theta=\theta_j$) and the moment bounds
on the noise $\xi_k$ (Assumption~\ref{ass:noise}), we obtain
\[
\mathbb{E}\bigl[\|S_{k,\tau}\|^2\bigr]
\;\le\;
C_S\,\alpha^2\,\tau
\]
for some $C_S<\infty$ independent of $\alpha$ and $\tau$.
Combining the bounds, we get
\[
\bigl\|T^{(1)}_\tau(\alpha)\bigr\|
\le
2C_{D''}\,
(C\alpha)^{1/2}\,
(C_S\,\alpha^2\,\tau)^{1/2}
=
C'\,\alpha^{3/2}\,\tau^{1/2},
\]
for some constant $C'<\infty$ independent of $\alpha$ and $\tau$.

\paragraph{Bound on $T^{(2)}_\tau(\alpha)$.}
Finally, for the term involving $S_{k,\tau}$, we have
\begin{align*}
\bigl\|T^{(2)}_\tau(\alpha)\bigr\|
&\le
\mathbb{E}\bigl[
\|D''(X_{k+1})\|_{\mathrm{op}}\,
\|S_{k,\tau}\|^2
\bigr] \le
C_{D''}\,
\mathbb{E}\bigl[\|S_{k,\tau}\|^2\bigr]
\le
C_{D''}C_S\,\alpha^2\,\tau.
\end{align*}
We now select $\tau$ as a function of $\alpha$ to make all three contributions
in~\eqref{eq:III_fluct_T_decomposition} negligible at order $\alpha$.
Choose
\[
\tau(\alpha) := \left\lceil c\log\frac{1}{\alpha}\right\rceil
\]
for some constant $c>0$.
Then
\[
(1-\rho)^{\tau(\alpha)}
\le
(1-\rho)^{c\log(1/\alpha)}
=
\alpha^{c|\log(1-\rho)|}.
\]
Choosing $c$ large enough ensures that $\rho^{\tau(\alpha)}=\mathcal{O}(\alpha^{\varepsilon})$
for some $\varepsilon>0$.
With this choice, the bounds derived above become
\begin{align*}
\bigl\|T^{(0)}_{\tau(\alpha)}(\alpha)\bigr\|
&\le
C_{\mathrm{mix}}C\,( (1-\rho)^{\tau(\alpha)} + \sqrt{\alpha} ) \,\alpha
=
\mathcal{O}(\alpha^{1+\varepsilon}), \\
\bigl\|T^{(1)}_{\tau(\alpha)}(\alpha)\bigr\|
&\le
C'\,\alpha^{3/2}\,\tau(\alpha)^{1/2}
=
C'\,\alpha^{3/2}\,(\log(1/\alpha))^{1/2}
=
\mathcal{O}(\alpha^{1+\varepsilon}), \\
\bigl\|T^{(2)}_{\tau(\alpha)}(\alpha)\bigr\|
&\le
C_{D''}C_S\,\alpha^2\,\tau(\alpha)
=
C_{D''}C_S\,\alpha^2\log(1/\alpha)
=
\mathcal{O}(\alpha^{1+\varepsilon}).
\end{align*}
Therefore, by~\eqref{eq:III_fluct_T_decomposition},
\[
\bigl\|\mathrm{(III)}_{\mathrm{fluct}}\bigr\|
\le
\frac{1}{2}\Bigl(
\bigl\|T^{(0)}_{\tau(\alpha)}(\alpha)\bigr\|
+
\bigl\|T^{(1)}_{\tau(\alpha)}(\alpha)\bigr\|
+
\bigl\|T^{(2)}_{\tau(\alpha)}(\alpha)\bigr\|
\Bigr)
=
\mathcal{O}(\alpha^{1+\varepsilon}).
\]

\begin{lemma}[Hessian fluctuation term is $\mathcal{O}(\alpha^{1+\varepsilon})$]
\label{lem:III_fluct}
Under Assumptions~\ref{ass:g_regularity}-\ref{ass:markov} and \ref{ass:stationary}-\ref{ass:wdstar}, the fluctuation term in~\eqref{eq:III_decomposition} satisfies
\[
\mathrm{(III)}_{\mathrm{fluct}}
=
\frac{1}{2}\,
\mathbb{E}\bigl[D''(X_1)[\Delta_\infty,\Delta_\infty]\bigr]
=
\mathcal{O}(\alpha^{1+\varepsilon}).
\]
\end{lemma}

Combining Lemma~\ref{lem:III_main} and Lemma~\ref{lem:III_fluct} with the
decomposition~\eqref{eq:III_decomposition} yields the expansion~\eqref{eq:III_target}, i.e.,
\[
\mathrm{(III)}
=
\alpha\,\frac{1}{2}\,\nabla^2\bar g(\theta^*)[M]
+
\mathcal{O}(\alpha^{1+\varepsilon}).
\]
This identifies the purely nonlinear contribution to the bias, which will
correspond to $b_{\mathrm{nonlin}}$ in the final decomposition.

\subsubsection{Collecting all terms}

We now combine the expansions of the four terms
$\mathrm{(I)}$--$\mathrm{(IV)}$ appearing in the master equation
\eqref{eq:master_equation_terms}.
From Lemma~\ref{lem:termI_localized}, Eq.~\eqref{eq:II_expanded}, and Lemma~\ref{lem:third_order_remainder}, we obtain that
\begin{equation}
\label{eq:I_final}
\mathrm{(I)}
=
\bar{\Lambda}_* \mathbb{E}\bigl[\Delta\bigr]
+
r_{\mathrm{I}}(\alpha),
\qquad
\|r_{\mathrm{I}}(\alpha)\| \le C_{\mathrm{I}}\,\alpha,
\end{equation}
\begin{equation} \label{eq:II_final}
\mathrm{(II)}
=
J^*\,\mathbb{E}\bigl[\Delta\bigr]
+
r_{\mathrm{II}}(\alpha),
\qquad
\|r_{\mathrm{II}}(\alpha)\|\le C_{\rm II} \,\alpha.
\end{equation}
\begin{equation}
\label{eq:IV_final}
\mathrm{(IV)} = r_{\mathrm{IV}}(\alpha),
\qquad
\|r_{\mathrm{IV}}(\alpha)\| \le C_{\mathrm{IV}}\,\alpha^{3/2}.
\end{equation}
We can prove that 
\begin{equation}
    \E[ \Delta ] = -( J^* + \bar{\Lambda}^* )^{-1} \texttt{(III)} + \tilde{r}(\alpha) \eqsp.
\end{equation}
where $\tilde{r}(\alpha) = r_{\rm I}(\alpha) + r_{\rm II}(\alpha) + r_{\rm IV}(\alpha) = \mathcal{O}(\alpha)$.
Together with the fact that $\| \textrm{(III)} \| = \mathcal{O}(\alpha)$ [cf.~\eqref{eq:crude_iii}], we show that the bias is of the order $\mathcal{O}(\alpha)$. This establishes Theorem~\ref{thm:bias_simple}.

Notice that under the additional Assumption~\ref{ass:covariance}, Lemma~\ref{lem:III_main} and Lemma~\ref{lem:III_fluct} show that
\begin{equation}
\label{eq:III_final}
\mathrm{(III)}=\alpha\,\frac{1}{2}\,\nabla^2\bar g(\theta^*)[M] + r_{\mathrm{III}}(\alpha),
\qquad
\|r_{\mathrm{III}}(\alpha)\|\le C_{\mathrm{III}}\,\alpha^{1+\varepsilon}.
\end{equation}
Substituting~\eqref{eq:I_final}--\eqref{eq:IV_final} into~\eqref{eq:master_equation_terms}
yields the balance equation
\begin{equation}
\label{eq:bias_master_expanded}
0 = ( \bar{\Lambda}_* + J^* )\,\mathbb{E}\bigl[\Delta\bigr] + \alpha\,\frac{1}{2}\,\nabla^2\bar g(\theta^*)[M] + r(\alpha),
\end{equation}
where $r(\alpha):=r_{\mathrm{I}}(\alpha)+r_{\mathrm{II}}(\alpha)+r_{\mathrm{III}}(\alpha)+r_{\mathrm{IV}}(\alpha)$ which satisfies
\begin{equation}\label{eq:r_alpha_bound}
\|r(\alpha)\|
\le
C_{\mathrm{I}}\alpha
+
C_{\mathrm{II}}\alpha
+
C_{\mathrm{III}}\alpha^{1+\varepsilon}
+
C_{\mathrm{IV}}\alpha^{3/2} = \mathcal{O}(\alpha).
\end{equation}
Finally, we recall that $(\bar{\Lambda}_* + J^*)$ is invertible, rearranging terms lead to
\begin{equation}
    \mathbb{E}[ \Delta ] = -\frac{\alpha}{2} ( \bar{\Lambda}_* + J^* )^{-1} \nabla^2\bar g(\theta^*)[M] + ( \bar{\Lambda}_* + J^* )^{-1} r(\alpha),
\end{equation}
where $\| ( \bar{\Lambda}_* + J^* )^{-1} r(\alpha) \| = \mathcal{O}(\alpha)$. This shows that the asymptotic bias remains of the order $\| \E[ \Delta ] \| = \mathcal{O}(\alpha)$.

\subsection{Proof of Linear Fluctuation Term} \label{app:bias_linear_fluct}

Let $\tilde{f}: \mathbb{X} \to \mathbb{R}^{d \times d}$ be $L_f$ Lipschitz continuous and satisfies that $\tilde{f}(x) := f(x) - \mathbb{E}_{X \sim \pi_{\theta^*}}[f(X)]$.
Notice that we have $\mathbb{E}_{ X \sim \pi_{\theta^*} }[ \tilde{f}(X) ] = 0$.
\begin{lemma}
\label{lem:II_fluct_app}
Under Assumptions~\ref{ass:g_regularity}, \ref{ass:markov}, and \ref{ass:stationary}. Then, there exist constants $C <\infty$ and $\alpha_0 > 0$ such that, for all $\alpha\in(0,\alpha_0)$,
\begin{equation}
\label{eq:II_fluct_bound_app}
\Bigl\| \mathbb{E}\bigl[\tilde{f}(X_1)\Delta\bigr] \Bigr\| \le C \,\alpha.
\end{equation}
\end{lemma}

\begin{proof}
Throughout this proof we work under the stationary law of the two-sided chain
$(\theta_k,X_k)_{k\in\mathbb{Z}}$ (Assumption~\ref{ass:stationary}).
Moreover, Assumption~\ref{ass:markov}-(2) yields
\begin{equation}\label{eq:pi_Lip_tildeJ}
\Bigl\| \E_{X\sim\pi_\theta}\bigl[\tilde{f}(X)\bigr]\Bigr\|
=
\Bigl\|\E_{\pi_\theta}\bigl[ f(X) \bigr]-\E_{\pi_{\theta^*}}\bigl[ f(X)\bigr]\Bigr\|
\le L_f L_P \|\theta-\theta^*\|.
\end{equation}

\paragraph{Step 1: A block decomposition of $\Delta_0$.}
Fix an integer lag $\tau\ge 1$ (chosen later as a function of $\alpha$) and write
\begin{equation}\label{eq:Delta_block_decomp}
\Delta_0
=
\Delta_{-\tau}
+
\sum_{j=0}^{\tau-1}\bigl(\Delta_{-j}-\Delta_{-(j+1)}\bigr).
\end{equation}
We apply the following decomposition
\begin{equation}\label{eq:main_split_C8}
\E[\tilde{f}(X_1)\Delta_0]
=
\underbrace{\E[\tilde{f}(X_1)\Delta_{-\tau}]}_{T_{\mathrm{lag}}}
+
\underbrace{\sum_{j=0}^{\tau-1}\E\bigl[\tilde{f}(X_1)(\Delta_{-j}-\Delta_{-(j+1)})\bigr]}_{T_{\mathrm{inc}}}.
\end{equation}
We bound $T_{\mathrm{lag}}$ and $T_{\mathrm{inc}}$ separately.

\paragraph{Step 2: Bound the lagged term $T_{\mathrm{lag}}$.}
Let $\mathcal{F}_{-\tau}:=\sigma\bigl((\theta_k,X_k):k\le -\tau\bigr)$.
Conditioning on $\mathcal{F}_{-\tau}$ and using that $\Delta_{-\tau}$ is $\mathcal{F}_{-\tau}$-measurable,
\[
T_{\mathrm{lag}}
=
\E\Bigl[ \E\bigl[\tilde{f}(X_1)\mid \mathcal{F}_{-\tau}\bigr] \Delta_{-\tau} \Bigr].
\]
Define the (random) probability measure $\mu_{-\tau,1}(\cdot):=\mathcal{L}(X_1\mid \mathcal{F}_{-\tau})$.
Then we further observe
\begin{align}
\E\bigl[\tilde{f}(X_1)\mid \mathcal{F}_{-\tau}\bigr] \;=\; \int_{\mathbb{X}}\tilde{f}(x)\,\mu_{-\tau,1}(dx) =
\underbrace{\int_{\mathbb{X}} \tilde{f} \,d\pi_{\theta_{-\tau}}}_{=:m(\theta_{-\tau})}
+
\int_{\mathbb{X}} \tilde{f} \,d(\mu_{-\tau,1}-\pi_{\theta_{-\tau}}).
\label{eq:cond_split}
\end{align}

\noindent \emph{(i) The stationary part $m(\theta_{-\tau})$.}
By \eqref{eq:pi_Lip_tildeJ},
\begin{equation}\label{eq:m_theta_bound}
\|m(\theta_{-\tau})\|
=
\Bigl\|\E_{\pi_{\theta_{-\tau}}}[\tilde{f}(X)]\Bigr\|
\le 2L_f L_P\,\|\Delta_{-\tau}\|.
\end{equation}
Therefore, 
\begin{equation}\label{eq:Tlag_main}
\Bigl\|\E\bigl[ m(\theta_{-\tau}) \Delta_{-\tau} \bigr]\Bigr\|
\le 2 L_f L_P\,\E\|\Delta_{-\tau}\|^2
= 2L_f L_P\,\E\|\Delta_0\|^2
\le 2L_f L_P\,C_1 \,\alpha,
\end{equation}
using stationarity and Lemma~\ref{lem:stationary_moment_bounds}.

\medskip
\noindent \emph{(ii) The mixing/inhomogeneity part.}
By the Lipschitzness of $\tilde{f}$, we have 
\begin{equation}\label{eq:TV_bound_tildeJ}
\Bigl\|\int_{\mathbb{X}} \tilde{f} \,d(\mu_{-\tau,1}-\pi_{\theta_{-\tau}})\Bigr\|
\le L_f \E_{ X \sim \mu_{-\tau,1}, X' \sim \pi_{\theta_{-\tau}}} [ \| X - X' \| ]
\end{equation}
Let $P_{\theta_{-\tau}}$ be the frozen kernel at time $-\tau$ and define the random inhomogeneous composition
\[
Q_{-\tau,1}:=P_{\theta_{-\tau}}P_{\theta_{-\tau+1}}\cdots P_{\theta_0},
\qquad\text{so that}\qquad
\mu_{-\tau,1}=\delta_{X_{-\tau}}Q_{-\tau,1}.
\]
Then, by the triangle inequality,
\begin{align*}
&\E_{ X \sim \mu_{-\tau,1}, X' \sim \pi_{\theta_{-\tau}}} [ \| X - X' \| ] \\
& \le
\underbrace{\E_{ \tilde{X} \sim \delta_{X_{-\tau}}P_{\theta_{-\tau}}^{\tau+1}, X' \sim \pi_{\theta_{-\tau}}} [ \| \tilde{X} - X' \| ]}_{(\mathrm{a})}
+
\underbrace{\E_{ \tilde{X} \sim \mu_{-\tau,1}, X' \sim \delta_{X_{-\tau}}P_{\theta_{-\tau}}^{\tau+1}} [ \| \tilde{X} - X' \| ]}_{(\mathrm{b})}. 
\end{align*}
For $(\mathrm{a})$, Assumption~\ref{ass:markov}-(1) gives
\begin{equation}\label{eq:frozen_mixing}
\E_{ \tilde{X} \sim \delta_{X_{-\tau}}P_{\theta_{-\tau}}^{\tau+1}, X' \sim \pi_{\theta_{-\tau}}} [ \| \tilde{X} - X' \| ]
\le C_X \, (1-\rho)^{\tau+1}.
\end{equation}
For $(\mathrm{b})$, Assumption~\ref{ass:markov}-(2) gives
\begin{equation} \label{eq:inhomogeneity_telescoping}
\E_{ \tilde{X} \sim \mu_{-\tau,1}, X' \sim \delta_{X_{-\tau}}P_{\theta_{-\tau}}^{\tau+1}} [ \| \tilde{X} - X' \| ] \leq L_P\sum_{t=-\tau}^{0}\|\theta_t-\theta_{-\tau}\|
\end{equation}
Combining \eqref{eq:frozen_mixing} and \eqref{eq:inhomogeneity_telescoping} yields the \emph{pathwise} bound
\begin{equation}\label{eq:TV_gap_tau_pathwise}
\E_{ X \sim \mu_{-\tau,1}, X' \sim \pi_{\theta_{-\tau}}} [ \| X - X' \| ]
\le
C_X (1-\rho)^{\tau+1}
+
L_P\sum_{t=-\tau}^{0}\|\theta_t-\theta_{-\tau}\|.
\end{equation}

Next we bound the expected size of the drift term.
By the SA recursion, $\theta_{k+1}-\theta_k=\alpha\bigl(g(\theta_k,X_{k+1})+\xi_{k+1}\bigr)$, hence
\[
\|\theta_{k+1}-\theta_k\|
\le
\alpha\|g(\theta_k,X_{k+1})\|+\alpha\|\xi_{k+1}\|.
\]
Using Assumptions~\ref{ass:g_regularity}--\ref{ass:noise} and stationarity, there exists $C_{\mathrm{step}}<\infty$ such that
$\sup_k \E\|\theta_{k+1}-\theta_k\|\le C_{\mathrm{step}}\alpha$. Consequently,
\begin{align}
\E\sum_{t=-\tau}^{0}\|\theta_t-\theta_{-\tau}\|
&\le
\E\sum_{t=-\tau}^{0}\sum_{r=-\tau}^{t-1}\|\theta_{r+1}-\theta_r\|
\le
(\tau+1)\sum_{r=-\tau}^{-1}\E\|\theta_{r+1}-\theta_r\|
\le
(\tau+1)\tau\,C_{\mathrm{step}}\alpha.
\label{eq:theta_drift_sum}
\end{align}

We now return to the term of interest. By \eqref{eq:TV_bound_tildeJ}, \eqref{eq:TV_gap_tau_pathwise},
and Cauchy-Schwarz inequality,
\begin{align}
& \Bigl\|\E\Bigl[ \int_{\mathbb{X}} \tilde{f} \,d(\mu_{-\tau,1}-\pi_{\theta_{-\tau}}) \Delta_{-\tau} \Bigr]\Bigr\|
\nonumber\\
&\le
L_f \,
\bigl(\E\|\Delta_{-\tau}\|^2\bigr)^{1/2}\,
\bigl( \E_{ X \sim \mu_{-\tau,1}, X' \sim \pi_{\theta_{-\tau}}} [ \| X - X' \|^2  ] \bigr)^{1/2}
\nonumber\\
&\le
\sqrt{2} L_f \,
\bigl(\E\|\Delta_{0}\|^2\bigr)^{1/2}\,
\Bigl(C_X (1-\rho)^{\tau+1}
+
L_P\,\bigl(\E(\textstyle\sum_{t=-\tau}^{0}\|\theta_t-\theta_{-\tau}\|)^2\bigr)\Bigr)^{1/2},
\label{eq:Tlag_mix_bound_pre}
\end{align}
where we used stationarity in the second moment of $\Delta$.
Finally, we have
\[
\Bigl(\E(\textstyle\sum_{t=-\tau}^{0}\|\theta_t-\theta_{-\tau}\|)^2\Bigr)^{1/2}
\le
\E\sum_{t=-\tau}^{0}\|\theta_t-\theta_{-\tau}\|
\le
(\tau+1)\tau\,C_{\mathrm{step}}\alpha,
\]
and we have $(\E\|\Delta_0\|^2)^{1/2}\le \sqrt{C_\Delta\alpha}$.
Thus, \eqref{eq:Tlag_mix_bound_pre} implies that for $C < \infty$ independent of $\alpha, \tau$,
\begin{equation}\label{eq:Tlag_mix_bound}
\Bigl\|\E\Bigl[\int_{\mathbb{X}} \tilde{f} \,d(\mu_{-\tau,1}-\pi_{\theta_{-\tau}}) \Delta_{-\tau} \Bigr]\Bigr\|
\le
C\sqrt{\alpha}\,\Bigl( C_X (1-\rho)^{\tau+1}+L_P\,\tau(\tau+1)\alpha\Bigr) .
\end{equation}

Choose the lag as
\begin{equation}\label{eq:tau_choice}
\tau=\tau(\alpha):=\Bigl\lceil \frac{\log(\alpha^{-1/2})}{\log((1-\rho)^{-1})}\Bigr\rceil,
\qquad\text{so that}\qquad (1-\rho)^{\tau}\le \sqrt{\alpha},
\end{equation}
and substituting into \eqref{eq:Tlag_mix_bound} gives
\[
\Bigl\|\E\Bigl[ \int_{\mathbb{X}} \tilde{f} \,d(\mu_{-\tau,1}-\pi_{\theta_{-\tau}}) \Delta_{-\tau} \Bigr]\Bigr\|
= \mathcal{O}( \alpha ).
\]
Together with \eqref{eq:Tlag_main}, this proves that there exists $C_{\mathrm{lag}}<\infty$ such that
\begin{equation}\label{eq:Tlag_final}
\|T_{\mathrm{lag}}\|\;\le\; C_{\mathrm{lag}}\,\alpha
\end{equation}
for all sufficiently small $\alpha$.

\paragraph{Step 3: Bound the increment term $T_{\mathrm{inc}}$.}
From the recursion,
\[
\Delta_{-j}-\Delta_{-(j+1)}
=
\alpha\Bigl(g(\theta_{-(j+1)},X_{-j})+\xi_{-j}\Bigr),
\]
and hence for each $j\ge 0$,
\begin{equation}\label{eq:inc_summand}
\E\bigl[\tilde{f} (X_1)(\Delta_{-j}-\Delta_{-(j+1)})\bigr]
=
\alpha\,\E\Bigl[\tilde{f} (X_1)\,g(\theta_{-(j+1)},X_{-j})\Bigr]
+
\alpha\,\E\Bigl[\tilde{f} (X_1)\,\xi_{-j}\Bigr].
\end{equation}

We now prove and use an explicit correlation--decay bound across a gap of length $j$.
Fix $j\ge 0$ and let $U_{-j}$ be any square-integrable, $\mathcal{F}_{-j}$-measurable random variable,
where $\mathcal{F}_{-j}:=\sigma\bigl((\theta_k,X_k):k\le -j\bigr)$.
Let $\mu_{-j,1}(\cdot):=\mathcal{L}(X_1\mid \mathcal{F}_{-j})$.
Then by the tower property,
\begin{equation}\label{eq:gap_cond_start}
\E\bigl[\tilde f(X_1)\,U_{-j}\bigr]
=
\E\Bigl[U_{-j}\,\E\bigl[\tilde f(X_1)\mid \mathcal{F}_{-j}\bigr]\Bigr]
=
\E\Bigl[U_{-j}\int_{\mathbb{X}}\tilde f(x)\,\mu_{-j,1}(dx)\Bigr].
\end{equation}
Adding and subtracting $\pi_{\theta_{-j}}$ inside the conditional expectation yields
\begin{align}
\E\bigl[\tilde{f} (X_1)\mid \mathcal{F}_{-j}\bigr]
&=
\int_{\mathbb{X}}\tilde{f} \,d\pi_{\theta_{-j}}
+
\int_{\mathbb{X}}\tilde{f} \,d(\mu_{-j,1}-\pi_{\theta_{-j}})
=: m(\theta_{-j}) + r_{-j,1}.
\label{eq:gap_cond_split}
\end{align}
Substituting \eqref{eq:gap_cond_split} into \eqref{eq:gap_cond_start} gives
\begin{equation}\label{eq:gap_main_decomp}
\E\bigl[\tilde{f} (X_1)\,U_{-j}\bigr]
=
\E\bigl[m(\theta_{-j})\,U_{-j}\bigr]
+
\E\bigl[r_{-j,1}\,U_{-j}\bigr].
\end{equation}

For the remainder term, we apply a similar trick as before for $T_{\mathrm{lag}}$,
\begin{equation}\label{eq:gap_r_TV}
\|r_{-j,1}\|
=
\Bigl\|\int_{\mathbb{X}}\tilde{f} \,d(\mu_{-j,1}-\pi_{\theta_{-j}})\Bigr\|
\le
L_f \, \E_{ X \sim \mu_{-j,1}, X' \sim \pi_{\theta_{-j}} } [ \| X - X' \| ].
\end{equation}
By Cauchy--Schwarz,
\begin{equation}\label{eq:gap_second_term_CS}
\Bigl\|\E\bigl[r_{-j,1}\,U_{-j}\bigr]\Bigr\|
\le
L_f \, \big( \E_{ X \sim \mu_{-j,1}, X' \sim \pi_{\theta_{-j}} } [ \| X - X' \|^2 ] \big)^{1/2}
\bigl(\E\|U_{-j}\|^2\bigr)^{1/2}.
\end{equation}
We now bound $\E_{ X \sim \mu_{-j,1}, X' \sim \pi_{\theta_{-j}} } [ \| X - X' \|^2 ]$.
Define
\[
Q_{-j,1}:=P_{\theta_{-j}}P_{\theta_{-j+1}}\cdots P_{\theta_0},
\qquad\text{so that}\qquad
\mu_{-j,1}=\delta_{X_{-j}}Q_{-j,1}.
\]
Then, similarly to Step~2, by the triangle inequality,
\begin{align}
\E_{ X \sim \mu_{-j,1}, X' \sim \pi_{\theta_{-j}} } [ \| X - X' \| ] &\le
C_X (1-\rho)^{j+1}
+
L_P\sum_{t=-j}^{0}\|\theta_t-\theta_{-j}\|,
\label{eq:gap_TV_final}
\end{align}
where the last line uses Assumption~\ref{ass:markov}-(1) for the first term and the telescoping argument
\eqref{eq:inhomogeneity_telescoping} (with $\tau$ replaced by $j$) for the second term.
Combining \eqref{eq:gap_main_decomp}, \eqref{eq:gap_second_term_CS}, and \eqref{eq:gap_TV_final}, we obtain
\begin{align}
& \Bigl\|\E\bigl[\tilde f(X_1)\,U_{-j}\bigr]\Bigr\|
\le
\Bigl\|\E\bigl[m(\theta_{-j})\,U_{-j}\bigr]\Bigr\| \notag
\\
& +
\sqrt{2} L_f
\Bigl(C_X (1-\rho)^{j+1}
+
L_P\bigl(\E(\textstyle\sum_{t=-j}^{0}\|\theta_t-\theta_{-j}\|)^2\bigr)^{1/2}\Bigr)
\bigl(\E\|U_{-j}\|^2\bigr)^{1/2}.
\label{eq:mix_cov_template_full}
\end{align}

Next, we apply \eqref{eq:mix_cov_template_full} with $U_{-j}=g(\theta_{-(j+1)},X_{-j})$ and $U_{-j}=\xi_{-j}$.
First, for the bias term involving $m(\theta_{-j})=\int \tilde f\,d\pi_{\theta_{-j}}$, by \eqref{eq:pi_Lip_tildeJ}
and Cauchy--Schwarz,
\begin{align*}
\Bigl\|\E\bigl[m(\theta_{-j})\,U_{-j}\bigr]\Bigr\|
&\le
\E\bigl[\|m(\theta_{-j})\|\,\|U_{-j}\|\bigr]
\le
2L_1L_P\,\E\bigl[\|\Delta_{-j}\|\,\|U_{-j}\|\bigr]\\
&\le
2L_1L_P\,\bigl(\E\|\Delta_{0}\|^2\bigr)^{1/2}\bigl(\E\|U_{-j}\|^2\bigr)^{1/2}
\le
2L_1L_P\,\sqrt{C_\Delta\alpha}\,\bigl(\E\|U_{-j}\|^2\bigr)^{1/2}.
\end{align*}
Second, we bound the drift term. By the same reasoning as \eqref{eq:theta_drift_sum},
\[
\E\sum_{t=-j}^{0}\|\theta_t-\theta_{-j}\|\le (j+1)j\,C_{\mathrm{step}}\alpha,
\]
and hence by Jensen,
\[
\Bigl(\E(\textstyle\sum_{t=-j}^{0}\|\theta_t-\theta_{-j}\|)^2\Bigr)^{1/2}
\le
\E\sum_{t=-j}^{0}\|\theta_t-\theta_{-j}\|
\le
(j+1)j\,C_{\mathrm{step}}\alpha.
\]

Finally, by Assumptions~\ref{ass:g_regularity}--\ref{ass:noise} and stationarity, we have the uniform second-moment bounds
\[
\sup_{j\ge 0}\E\|g(\theta_{-(j+1)},X_{-j})\|^2<\infty,
\qquad
\sup_{j\ge 0}\E\|\xi_{-j}\|^2<\infty.
\]
Therefore, applying \eqref{eq:mix_cov_template_full} with these choices of $U_{-j}$ yields constants
$C_{g},C_{\xi}<\infty$ such that for all $j\ge 0$ and all sufficiently small $\alpha$,
\[
\Bigl\|\E\bigl[\tilde f(X_1)\,g(\theta_{-(j+1)},X_{-j})\bigr]\Bigr\|
\le
C_{g}(1-\rho)^{j+1}+C_{g}\sqrt{\alpha}+C_{g}\,j(j+1)\alpha,
\]
and similarly
\[
\Bigl\|\E\bigl[\tilde f(X_1)\,\xi_{-j}\bigr]\Bigr\|
\le
C_{\xi}(1-\rho)^{j+1}+C_{\xi}\sqrt{\alpha}+C_{\xi}\,j(j+1)\alpha.
\]
Substituting into \eqref{eq:inc_summand} gives, for a constant $C_{\mathrm{inc},0}<\infty$,
\begin{equation}\label{eq:inc_summand_bound}
\Bigl\|\E\bigl[\tilde f(X_1)(\Delta_{-j}-\Delta_{-(j+1)})\bigr]\Bigr\|
\le
\alpha\,C_{\mathrm{inc},0}\Bigl((1-\rho)^{j+1}+\sqrt{\alpha}+j(j+1)\alpha\Bigr),
\qquad j\ge 0.
\end{equation}

Summing \eqref{eq:inc_summand_bound} over $j=0,\dots,\tau-1$ yields
\begin{align}
\|T_{\mathrm{inc}}\|
&\le
\alpha\,C_{\mathrm{inc},0}\sum_{j=0}^{\tau-1} (1-\rho)^{j+1}
+
\alpha\,C_{\mathrm{inc},0}\sum_{j=0}^{\tau-1}\sqrt{\alpha}
+
\alpha\,C_{\mathrm{inc},0}\sum_{j=0}^{\tau-1}j(j+1)\alpha
\nonumber\\
&\le
\alpha\,C_{\mathrm{inc},0}\frac{1-\rho}{\rho}
+
\alpha\,C_{\mathrm{inc},0}\,\tau\sqrt{\alpha}
+
\alpha^2\,C_{\mathrm{inc},0}\sum_{j=0}^{\tau-1}j(j+1).
\label{eq:Tinc_sum_pre}
\end{align}
Using $\sum_{j=0}^{\tau-1}j(j+1)\le \tau^3$ and the choice \eqref{eq:tau_choice} (so $\tau\asymp \log(1/\alpha)$),
we get
\[
\alpha\,\tau\sqrt{\alpha}\;\lesssim\;\alpha^{3/2}\log(1/\alpha),
\qquad
\alpha^2\tau^3\;\lesssim\;\alpha^2\log^3(1/\alpha),
\]
both of which are $o(\alpha)$ as $\alpha\downarrow 0$. Hence \eqref{eq:Tinc_sum_pre} implies that for sufficiently small $\alpha$,
\begin{equation}\label{eq:Tinc_final}
\|T_{\mathrm{inc}}\|
\le
C_{\mathrm{inc}}\,\alpha
\end{equation}
for some constant $C_{\mathrm{inc}}<\infty$ independent of $\alpha$.

Combining \eqref{eq:main_split_C8}, \eqref{eq:Tlag_final}, and \eqref{eq:Tinc_final}, we obtain
\[
\Bigl\|\E\bigl[\tilde f(X_1)\Delta_0\bigr]\Bigr\|
\le
\|T_{\mathrm{lag}}\|+\|T_{\mathrm{inc}}\|
\le
(C_{\mathrm{lag}}+C_{\mathrm{inc}})\alpha
=: C \,\alpha,
\]
for all $\alpha$ sufficiently small. This proves Lemma~\ref{lem:II_fluct}.
\end{proof}

\section{Examples Satisfying the \textsf{WD$^*$} Condition} \label{app:poisson-gateaux-example}
This section discusses and provides details to Examples \ref{ex:mh}, \ref{ex:clipped}, \ref{ex:langevin}. In particular, we show that they satisfy the {\sf WD$^*$} condition \eqref{eq:wd_cond}, but is otherwise incompatible with the global smoothness conditions considered in prior works. 

\paragraph{Details for Example \ref{ex:mh}} We first consider a random--walk Metropolis--Hastings (MH) kernel with a
parameter-dependent target density. Let $\mathbb{X} = \mathbb{R}^m$ and fix a proposal density $q(\cdot)$ on $\mathbb{R}^m$ that is strictly positive and smooth (e.g.\ Gaussian). For each parameter $\theta \in \Theta \subset
\mathbb{R}^d$ we consider a target density of the form
\[
    \pi_\theta(x) \propto \exp\big(-U(x;\theta)\big),
    \qquad x \in \mathbb{X},
\]
where the potential $U(\cdot;\theta)$ depends on~$\theta$.
We assume the following:
\begin{enumerate}[label=(\roman*)]
    \item For each $\theta$ in a neighbourhood of $\theta^\ast$, the
    function $x \mapsto U(x;\theta)$ is $C^2$ with globally Lipschitz gradient
    and satisfies a standard dissipativity condition (e.g.\ strongly convex
    outside a compact set). This ensures that the MH chain with target
    $\pi_\theta$ is geometrically ergodic and has a unique invariant
    distribution $\pi_\theta$.

    \item For each $x$, the map $\theta \mapsto U(x;\theta)$ is $C^2$ with
    first and second derivatives bounded by an integrable envelope in $x$
    on a neighbourhood of $\theta^\ast$. In particular, for each direction
    $u \in \mathbb{R}^m$ the directional derivative
    $\partial_u U(x;\theta)$ exists and is dominated by a function with finite
    $\pi_{\theta^\ast}$-moment.
\end{enumerate}
Under these conditions the fixed kernel $P_{\theta^\ast}$ is geometrically
ergodic and the Poisson equation
\[
    (I - P_{\theta^\ast}) \hat{g}(\cdot) = g(\theta^\ast,\cdot) - \bar g(\theta^\ast)
\]
has a bounded solution $\hat{g}$ with suitable moment control.
As discussed in Example~\ref{ex:mh}, given $X_k = x$ with the proposal $Y_{k+1} = x + Z_{k+1}$, $Z_{k+1} \sim q(\cdot)$. We note that the MH kernel takes the form of
\[
\begin{split}
    & P_\theta(x, \mathrm{d}y)
    = \alpha_\theta(x,y)\,q(y-x)\,\mathrm{d}y
      + r_\theta(x)\,\delta_x(\mathrm{d}y),\\
    & \alpha_\theta(x,y)
    = \min\!\Big\{1,\,
        \frac{\pi_\theta(y) q(x-y)}{\pi_\theta(x) q(y-x)}
      \Big\}
    = \min\{1, r_\theta(x,y)\},
\end{split}
\]
and $r_\theta(x,y) = \pi_\theta(y)/\pi_\theta(x)$ is the Hastings ratio.
Note that, the map $\theta \mapsto P_\theta$ is not differentiable and the kernel violates the smoothness assumption.

On the other hand, we can verify $({\sf WD}^\ast)$ which requires us to study the composite map $\theta \mapsto P_\theta h$ at $\theta^\ast$ against the invariant
measure of $P_{\theta^\ast}$. Define
\[
    \Psi(\theta)
    := \mathbb{E}_{X \sim \pi_{\theta^\ast}}\big[ P_\theta h(X) \big].
\]
By expanding the kernel, we obtain
\begin{align*}
    \Psi(\theta)
    &= \int_{\mathbb{R}^d}
       \int_{\mathbb{R}^d}
          \alpha_\theta(x,y)\,h(y)\,q(y-x)\,\mathrm{d}y\,
       \pi_{\theta^\ast}(\mathrm{d}x)
       + \mathbb{E}_{\pi_{\theta^\ast}}[h(X)] \\
    &= \mathbb{E}_{\pi_{\theta^\ast}}[h(X)] +
       \iint \alpha_\theta(x,y)\big(h(y)-h(x)\big)q(y-x)\,\mathrm{d}y\,
       \pi_{\theta^\ast}(\mathrm{d}x).
\end{align*}
The second term captures the dependence on $\theta$:
\[
    \mathcal{I}(\theta)
    := \iint
        \min\{1, r_\theta(x,y)\}
        \,\big(h(y)-h(x)\big)\,q(y-x)\,\mathrm{d}y\,
        \pi_{\theta^\ast}(\mathrm{d}x).
\]
Introducing
\[
    f(\theta,x,y)
    := \log r_\theta(x,y)
    = -U(y;\theta) + U(x;\theta).
\]
We observe that the only non-smoothness in $\theta$ arises
through the composition
\[
    \min\{1, e^{f(\theta,x,y)}\}
    \quad\text{with kink set:}\quad
    \mathcal{C}
    := \big\{(x,y) : f(\theta^\ast,x,y) = 0\big\}.
\]
By (ii) above and smoothness of $U$ in $\theta$, the map
$\theta \mapsto f(\theta,x,y)$ is $C^1$ for each fixed $(x,y)$, and the set
$\mathcal{C}$ is a level set of a smooth function in $(x,y)$ for fixed
$\theta^\ast$. As the proposal density $q(y-x)$ is absolutely
continuous and $\pi_{\theta^\ast}$ has a density (by the assumptions on $U$), the set $\mathcal{C}$ has zero measure under the product
$\pi_{\theta^\ast}(\mathrm{d}x)\,q(y-x)\,\mathrm{d}y$.

For $(x,y) \notin \mathcal{C}$ the map
\[
    \theta \longmapsto
    \min\{1, e^{f(\theta,x,y)}\}
\]
is $C^1$ in a neighbourhood of $\theta^\ast$, with directional derivative
in direction $u$ given by
\[
    \partial_u \big[\min\{1, e^{f(\theta,x,y)}\}\big]
    =
    \begin{cases}
        0,
            & f(\theta,x,y) < 0, \\[4pt]
        e^{f(\theta,x,y)}\,\partial_u f(\theta,x,y),
            & f(\theta,x,y) > 0.
    \end{cases}
\]
At $\theta = \theta^\ast$ this derivative is bounded by
\[
    \big|\partial_u \big[\min\{1, e^{f(\theta^\ast,x,y)}\}\big]\big|
    \;\le\; C\,\|\partial_u f(\theta^\ast,x,y)\|,
\]
for some constant $C$ depending on local bounds for $U$. Since $h$ is
bounded and $\partial_u f(\theta^\ast,x,y)$ is integrably dominated in
$(x,y)$, the integrand
\[
    \min\{1, e^{f(\theta,x,y)}\}\,\big(h(y)-h(x)\big)\,q(y-x)
\]
admits a directional derivative at $\theta^\ast$ for almost every
$(x,y)$, bounded by an integrable envelope.

By dominated convergence and the fact that the kink set $\mathcal{C}$ is null, we can differentiate $\mathcal{I}(\theta)$ under the integral sign, i.e.,
for each $u \in \mathbb{R}^m$,
\[
    \partial_u \mathcal{I}(\theta^\ast)
    =
    \iint
        \partial_u\!\Big[
            \min\{1, r_\theta(x,y)\}
        \Big]_{\theta=\theta^\ast}
        \big(h(y)-h(x)\big)
        q(y-x)\,\mathrm{d}y\,
        \pi_{\theta^\ast}(\mathrm{d}x).
\]
This defines a bounded linear functional in $u$, and hence a bounded linear
operator $\Lambda_{\theta^\ast}$ on $\mathbb{R}^m$ with values in
$L^\infty(\mathsf{X})$. Let
\[
    \big(\Lambda_{\theta^\ast}[u]\big)(x)
    := \int
         \partial_u\!\Big[
            \min\{1, r_\theta(x,y)\}
         \Big]_{\theta=\theta^\ast}
         \big(h(y)-h(x)\big)
         q(y-x)\,\mathrm{d}y ,
\]
we have
\[
    \mathcal{I}(\theta^\ast + t u)
    = \mathcal{I}(\theta^\ast)
      + t\,\mathbb{E}_{\pi_{\theta^\ast}}
         \big[\Lambda_{\theta^\ast}[u](X)\big]
      + o(t),
    \qquad t \to 0.
\]

Finally, the local Lipschitz continuity of
$\vartheta \mapsto \Lambda_\vartheta$ near $\theta^\ast$ follows from the
$C^2$ dependence of $U(\cdot;\vartheta)$ on $\vartheta$ and the same
dominated convergence argument, since the derivative of
$\min\{1, r_\vartheta(x,y)\}$ is continuous in~$\vartheta$ away from the
null kink set and the envelope bounds are uniform on a small neighbourhood
of~$\theta^\ast$. We conclude that Assumption~\ref{ass:wdstar} holds for
the MH kernel.

\paragraph{Details for Example \ref{ex:clipped}}
We consider a simple one-dimensional autoregressive dynamics with
clipping, which underlies several of the ``clipped'' kernels in our
experiments. Let $\mathsf{X} = [-C,C]$ and consider
\[
    X_{k+1}
    = \operatorname{clip}\big(\rho X_k + m(\theta) + \sigma(\theta)\,\xi_{k+1},
                          -C, C\big),
\]
where $|\rho| < 1$, $(\xi_k)$ are i.i.d.\ with a smooth density
$\varphi$ on $\mathbb{R}$ (e.g.\ standard Gaussian), and
$\operatorname{clip}(y,-C,C) := \min\{\max\{y,-C\},C\}$. The functions
$m(\theta)$ and $\sigma(\theta)$ encode the decision-dependence of the drift and noise level which are typically smooth or piecewise smooth.
Let
\[
    Y_{k+1}(\theta)
    := \rho X_k + m(\theta) + \sigma(\theta)\,\xi_{k+1}
\]
denote the pre-clipped variable. Conditioned on $X_k = x$, $Y_{k+1}(\theta)$
has a density
\[
P_\theta(x,y)
    = \frac{1}{|\sigma(\theta)|}\,
      \varphi\!\left(
        \frac{y - \rho x - m(\theta)}{\sigma(\theta)}
      \right),
    \qquad y \in \mathbb{R},
\]
which is smooth in $\theta$ on any compact subset where $\sigma(\theta)$
stays bounded away from $0$. 
We may now derive properties for the state-transition kernel of the clipped dynamics. Given $X_k = x$, the next state is
$X_{k+1} = \operatorname{clip}(Y_{k+1}(\theta),-C,C)$. For any bounded measurable $h(\cdot)$ we can write
\begin{align*}
    \mathbb{E}[h(X_{k+1}) \mid X_k = x]
    &= h(-C)
       \int_{(-\infty,-C)} p_\theta(x,y)\,\mathrm{d}y
     + \int_{[-C,C]} h(y)\,p_\theta(x,y)\,\mathrm{d}y + h(C)
       \int_{(C,\infty)} p_\theta(x,y)\,\mathrm{d}y.
\end{align*}
Equivalently,
\[
    (P_\theta h)(x)
    = h(-C)\,a_\theta(x)
      + \int_{-C}^C h(y)\,p_\theta(x,y)\,\mathrm{d}y
      + h(C)\,b_\theta(x),
\]
where $a_\theta(x)$ and $b_\theta(x)$ are the left and right tail masses of
$Y_{k+1}(\theta)$.

To verify {\sf WD}$^*$, we fix $\theta^\ast$ and set again $h(\cdot) = \hat g(\theta^\ast,\cdot)$.
Assume that $m(\theta)$ and $\sigma(\theta)$ are $C^1$ in a neighbourhood of $\theta^\ast$ and that their derivatives are bounded on that neighbourhood. Then the following conclusions hold:
\begin{enumerate}[label=(\roman*)]
    \item For each $x$ and $y$, the map
    $\theta \mapsto p_\theta(x,y)$ is $C^1$ near $\theta^\ast$, with
    directional derivative $\partial_u p_\theta(x,y)$ bounded by an
    integrable envelope in~$y$ (using the smoothness of $\varphi$ and the
    chain rule).

    \item The tail masses $a_\theta(x)$ and $b_\theta(x)$ are differentiable
    in $\theta$, and
    \[
        \partial_u a_\theta(x)
        = \int_{(-\infty,-C)} \partial_u p_\theta(x,y)\,\mathrm{d}y,
        \qquad
        \partial_u b_\theta(x)
        = \int_{(C,\infty)} \partial_u p_\theta(x,y)\,\mathrm{d}y,
    \]
    with derivatives bounded uniformly in~$x$ in terms of the same envelope.
\end{enumerate}
For any fixed $x$ and direction $u$, the difference quotient
\[
    \frac{
        (P_{\theta^\ast + tu} h)(x)
        - (P_{\theta^\ast} h)(x)
    }{t}
\]
splits into three pieces. Each piece
is an integral of $h(\cdot)$ against the difference quotient of
$p_\theta(x,y)$ or the tail masses. Since $h$ is bounded and the
difference quotients are dominated by an integrable envelope in~$y$, the
dominated convergence theorem yields the limit
\[
\begin{split}
& \lim_{t\to 0}
    \frac{
        (P_{\theta^\ast + tu} h)(x)
        - (P_{\theta^\ast} h)(x)
    }{t} \\
& =
    h(-C)\,\partial_u a_{\theta^\ast}(x)
      + \int_{-C}^C h(y)\,\partial_u p_{\theta^\ast}(x,y)\,\mathrm{d}y
      + h(C)\,\partial_u b_{\theta^\ast}(x)
    =: \big(\Lambda_{\theta^\ast}[u]\big)(x).
\end{split}
\]
Thus $\theta \mapsto P_\theta h$ is Gateaux differentiable at $\theta^\ast$
with derivative $\Lambda_{\theta^\ast}$. Local Lipschitz continuity of
$\vartheta \mapsto \Lambda_\vartheta$ follows from the local Lipschitz
dependence of $m(\vartheta)$ and $\sigma(\vartheta)$ and the same envelope
bounds, exactly as in the projected Langevin case. Consequently,
Assumption~\ref{ass:wdstar} holds for the clipped state dynamics.

\paragraph{Details for Example \ref{ex:langevin}} Finally, we consider a projected Langevin-type Markov chain with
parameter-dependent drift. Let $\mathbb{X} = \mathbb{R}^d$ and let
$(\xi_k)_{k\ge 1}$ be i.i.d.\ standard Gaussian random variables in
$\mathbb{R}^d$. For each $\theta \in \Theta$ we are given a potential
$U_\theta : \mathbb{R}^d \to \mathbb{R}$ and define the update
\[
    X_{k+1}
    = \Pi_{\mathcal{K}}\!\big(
        X_k - \eta\,\nabla U_\theta(X_k)
        + \sqrt{2\eta}\,\xi_{k+1}
      \big),
\]
where $\eta > 0$ is fixed and $\Pi_{\mathcal{K}}$ denotes Euclidean
projection onto a nonempty closed convex set $\mathcal{K} \subset
\mathbb{R}^d$.
Under standard assumptions (e.g.\ $U_\theta$ strongly convex on
$\mathcal{K}$ with Lipschitz gradient, tails controlled outside
$\mathcal{K}$, and $\eta$ sufficiently small) the projected chain admits a
unique invariant measure $\pi_\theta$ and is geometrically ergodic. The
projection $\Pi_{\mathcal{K}}$ is $1$-Lipschitz but non-differentiable on
the boundary of $\mathcal{K}$ and along the normal cones. Hence the kernel
$\theta \mapsto P_\theta$ is globally non-smooth as a map into operators.

To derive the density representation, observe that when conditioned on $X_k = x$, the pre-projection variable is
\[
    Z_{k+1}(\theta)
    := x - \eta\,\nabla U_\theta(x)
       + \sqrt{2\eta}\,\xi_{k+1}.
\]
The distribution of $Z_{k+1}(\theta)$ is Gaussian with mean
$m_\theta(x) := x - \eta\,\nabla U_\theta(x)$ and covariance $2\eta I$. Let
$p_\theta(x,z)$ denote its density with respect to Lebesgue measure. Then
for any bounded measurable $h$ we have
\[
    (P_\theta h)(x)
    = \int_{\mathbb{R}^d}
         h\big(\Pi_{\mathcal{K}}(z)\big)\,
         p_\theta(x,z)\,\mathrm{d}z.
\]
Observe that the dependence on $\theta$ enters only through the Gaussian mean $m_\theta(x)$, the projection $\Pi_{\mathcal{K}}$
and the test function $h$ do not depend on~$\theta$.

To verify $({\sf WD}^\ast)$, we fix $\theta^\ast$ and let $h = \hat g(\theta^\ast,\cdot)$ as before. The previous assumptions on $U_\theta(\cdot)$ guarantee:
\begin{enumerate}[label=(\roman*)]
    \item For each $x$ and $z$, the map
    $\theta \mapsto p_\theta(x,z)$ is $C^1$ in a neighbourhood of
    $\theta^\ast$, with directional derivative $\partial_u p_\theta(x,z)$
    bounded by a Gaussian envelope that is integrable against $|h(\Pi_{\mathcal{K}}(z))|$.

    \item For each $x$, the map
    $z \mapsto h(\Pi_{\mathcal{K}}(z))$ is bounded and continuous, since
    $h$ is bounded and $\Pi_{\mathcal{K}}$ is continuous.
\end{enumerate}
Fix $x$ and $u \in \mathbb{R}^m$. For $t \in \mathbb{R}$, consider
\[
    \frac{
        (P_{\theta^\ast + tu} h)(x)
        - (P_{\theta^\ast} h)(x)
    }{t}
    =
    \int_{\mathbb{R}^d}
        h(\Pi_{\mathcal{K}}(z))\,
        \frac{p_{\theta^\ast+tu}(x,z) - p_{\theta^\ast}(x,z)}{t}
        \,\mathrm{d}z.
\]
By (i), the pointwise limit
\[
    \frac{p_{\theta^\ast+tu}(x,z) - p_{\theta^\ast}(x,z)}{t}
    \;\longrightarrow\;
    \partial_u p_{\theta^\ast}(x,z)
\]
exists for each $z$, and the difference quotient is dominated by an
integrable envelope independent of $t$ (coming from Gaussian tails and
bounded local derivatives of $m_\theta(x)$ in $\theta$). Therefore, by
dominated convergence,
\[
    \lim_{t\to 0}
    \frac{
        (P_{\theta^\ast + tu} h)(x)
        - (P_{\theta^\ast} h)(x)
    }{t}
    =
    \int_{\mathbb{R}^d}
        h(\Pi_{\mathcal{K}}(z))\,
        \partial_u p_{\theta^\ast}(x,z)\,
        \mathrm{d}z
    =: \big(\Lambda_{\theta^\ast}[u]\big)(x).
\]
This defines a bounded linear map $u \mapsto \Lambda_{\theta^\ast}[u]$ from
$\mathbb{R}^m$ into $L^\infty(\mathsf{X})$, and we have shown that
$\theta \mapsto P_\theta h$ is Gateaux differentiable at $\theta^\ast$.

Finally, the local Lipschitz continuity of $\vartheta \mapsto \Lambda_\vartheta$ in a
neighbourhood of $\theta^\ast$ follows from the $C^2$-dependence of
$U_\vartheta$ on $\vartheta$ and the same dominated convergence argument:
the derivatives $\partial_u p_\vartheta(x,z)$ vary locally Lipschitz in
$\vartheta$ under uniform integrable envelopes, hence so do the integrals
defining $\Lambda_\vartheta[u]$. Consequently,
Assumption~\ref{ass:wdstar} holds for the projected Langevin kernel.

\bibliography{refs.bib}
\bibliographystyle{plainnat}


\end{document}